\documentclass[reqno]{amsart}

\usepackage{amsmath,amssymb,enumerate, amsfonts,mathtools}

\usepackage{epsfig,amsthm}
\usepackage[dvipsnames]{xcolor}
\usepackage{soul}

\newtheorem{maintheorem}{Theorem}

\newtheorem*{theorem*}{Theorem}
\newtheorem*{corollary*}{Corollary}
\newtheorem{theorem}{Theorem}[section]

\newtheorem{example}[theorem]{Example}
\newtheorem{proposition}[theorem]{Proposition}
\newtheorem{lemma}[theorem]{Lemma}
\newtheorem{corollary}[theorem]{Corollary}

\newtheorem{definition}[theorem]{Definition}

\newtheorem{claim}[theorem]{Claim}

 \theoremstyle{remark}
\newtheorem{remark}[theorem]{Remark}

\newcounter{jbStepCounter}

\newcommand\diam{{\operatorname{diam}}}
\newcommand\Diff{\operatorname{Diff}}
\newcommand\eps{\varepsilon}

\newcommand\cR{\mathcal R}

\newcommand\cL{\mathcal L}
\newcommand\cT{\mathcal T}

\newcommand\loc{{\operatorname{loc}}}

\newcommand\NN{{\mathbb N}}

\newcommand\Prob{{\mathbb P}}

\newcommand\RR{{\mathbb R}}

\renewcommand\top{{\operatorname{top}}}

\newcommand\ZZ{{\mathbb Z}}

\newcommand\hF{\widehat{F}}

\newcommand\hK{\widehat K}
\newcommand\hW{\widehat W}
\newcommand\hz{\widehat z}

\newcommand\Neutral{\mathfrak N}

\newcommand\vf{\varphi}
\newcommand\hm{\widehat{m}}
\renewcommand\emptyset{\varnothing}
\newcommand\dist{{\mathrm{dist}}}
\newcommand\wh{\widehat}
\newcommand\R{\mathbb R}

\newcommand\weakstar{weak-$*$}
\newcommand\Card{\operatorname{Card}}
\newcommand\rbd[1]{\overset{\scriptscriptstyle #1-\text{bd}}\longrightarrow}

\begin{document}

\title[Continuity properties of Lyapunov exponents]{Continuity properties of Lyapunov exponents\\ for surface diffeomorphisms}

\begin{abstract}
We study the entropy and  Lyapunov exponents of
invariant measures $\mu$ for smooth surface diffeomorphisms $f$, as functions of $(f,\mu)$. The main result is an inequality relating
the discontinuities  of these functions.
One consequence  is that for a $C^\infty$  surface diffeomorphisms,
on any set of ergodic measures with entropy bounded away from zero, continuity of the entropy implies continuity of the exponents.
Another consequence is the upper semi-continuity of the Hausdorff dimension on the set of  ergodic invariant measures with entropy bounded away from zero.
We also obtain a new criterion for the existence of SRB measures with positive entropy.

\end{abstract}

\renewcommand{\keywordsname}{Keywords}

\author{J\'er\^ome Buzzi}
\address{J.~Buzzi, \rm Laboratoire de Math\'ematiques d'Orsay, CNRS - UMR 8628.
Universit\'e Paris-Saclay, 91405 Orsay, France. \emph{E-mail address:}
\tt{jerome.buzzi@math.cnrs.fr}}
\thanks{J.B. was partially supported by the ISDEEC project ANR-16-CE40-0013}

\author{Sylvain Crovisier}
\address{S.~Crovisier, \rm Laboratoire de Math\'ematiques d'Orsay, CNRS - UMR 8628.
Universit\'e Paris-Saclay, 91405 Orsay, France. \emph{E-mail address:}
\tt{sylvain.crovisier@math.u-psud.fr}}
\thanks{S.C. was partially  supported  by  the  ERC  project 692925 \emph{NUHGD}}

\author{Omri Sarig}
\address{O.~Sarig, \rm Weizmann Institute of Science, 234 Herzl Street, 7610001 Rehovot, Israel. \emph{E-mail address:}
\tt{omsarig@gmail.com}}

\date{\today}
\thanks{Part of this work was done when O.S. was visiting Universit\'e Paris-Sud and IH\'ES, and he would like to thank these institutions for their hospitality and excellent working conditions. O.S. also acknowledges partial support of ISF grant  1149/18. }


\maketitle

\newcommand\HC{\operatorname{HC}}
\newcommand\cO{\mathcal O}
\newcommand\Probhyp{\Prob_{\rm hyp}}
\newcommand\heps{\widehat \eps}
\newcommand\hM{\widehat M}
\newcommand\hP{\widehat P}
\newcommand\hU{\widehat U}
\newcommand\hV{\widehat V}
\newcommand\hY{\widehat Y}
\newcommand\hf{\widehat f}
\newcommand\hg{\widehat g}
\newcommand\hx{\widehat x}
\newcommand\hy{\widehat y}
\newcommand\hlambda{\widehat \lambda}
\newcommand\hsigma{\widehat \sigma}
\newcommand\hpi{\widehat \pi}
\newcommand\hmu{\widehat \mu}
\newcommand\hnu{\widehat \nu}
\newcommand\Lip{\operatorname{Lip}}
\newcommand{\supnorm}{{\sup}}
\newcommand\oh{\overline{h}}
\newcommand\Crr{$C^r$}
\newcommand\cU{\mathcal U}
\newcommand{\cV}{{\mathcal V}}

\newcommand\proofitem{\medbreak\noindent--\ }
\newcommand\magenta{\color{magenta}}

\newcommand\mumax{\mu_{\max}}

\section*{Introduction}

Entropy and Lyapunov exponents play a major role in the study of  differentiable dynamical
systems, and their dependence on the measure and the map is of great interest.
This dependence is sometimes continuous, but not always (for entropy, see
~\cite{misiurewicz,Newhouse1990,BuzziPhD,Boyle-Downarowicz-2004,Downarowicz-Newhouse-2005,Burguet},
and for Lyapunov exponents, see ~\cite{Ruelle-exponent-continuity,Furstenberg-Kifer,Bochi,Bocker-Viana, Avila-Viana, Viana-survey}).
{While there are many works relating the values of the  entropy to the  values of the Lyapunov exponents ~\cite{Ruelle-Entropy-Inequality,pesin-formula,Ledrappier-Young-I,Ledrappier-Young-II},
the relation between the (dis)continuity of these objects as functions of the measure and the diffeomorphism has not yet been studied. The purpose of this work is to fill this gap, in the smooth two-dimensional case.} 
For instance, we show:

\begin{theorem*}
Let $f$ be a $C^\infty$ diffeomorphism of a compact surface without boundary.  Let
$\nu_1,\nu_2,\dots$ be ergodic measures for $f$, which  converge in the \weakstar\  topology
to an ergodic measure $\mu$ with positive entropy.

If the entropy of $\nu_k$ converges to the entropy of $\mu$,
then the Lyapunov exponents of $\nu_k$ converge to the Lyapunov exponents of $\mu$.
\end{theorem*}

This has strong dynamical consequences, some of which we will discuss here, and some of which we will discuss in  a companion paper.
For  example, we have the following  application  to the problem of the existence of SRB measures.
Let $\delta^u(\mu)$ denote the unstable dimension of $\mu$ in the sense of Ledrappier and Young (see Section \ref{ss.applications}).

\begin{corollary*}
Let $f$ be a $C^\infty$ diffeomorphism of a compact surface without boundary.
If there exist ergodic invariant probability measures $\nu_k$, with entropy uniformly bounded away from $0$, and such that
$\delta^u(\nu_k)\to1$,  then $f$ admits an ergodic SRB measure with positive entropy.
\end{corollary*}

\noindent
For other consequences, including the upper semi-continuity of
{the unstable dimension and of the Hausdorff dimension of ergodic measures with positive entropy, see Section \ref{ss.applications}.

\medbreak
These results follow from  inequalities between the multiplicative size of the  defects in continuity of the entropy and the top Lyapunov exponent. These inequalities, which are the main results of this work, are described in detail in the next section.

\section{Main results}
Throughout this paper,  $M$ is a two-dimensional compact  $C^\infty$ Riemannian manifold without boundary. Let $\Diff^r(M)$ denotes the class of $C^r$ diffeomorphisms on $M$ (see \S\ref{ss.size}).

Suppose  $\mu$ is an $f$-invariant probability measure for some $f\in\Diff^1(M)$. The {\em Kolmogorov-Sina{\u\i} entropy} of $\mu$ will be denoted by
$
h(f,\mu)
$.
Almost every $x\in M$ has two well-defined Lyapunov exponents $\lambda^-(f,x)\leq \lambda^+(f,x)$. The {\em Lyapunov exponents of $\mu$}  are defined by
$$
\lambda^-(f,\mu):=\lambda^-_\mu:=\int \lambda^-(f,x)d\mu\ , \ \lambda^+(f,\mu):=\lambda^+_\mu:=\int \lambda^+(f,x)d\mu.
$$
We are interested in the regularity of $(f,\mu)\mapsto h(f,\mu)$ and $(f,\mu)\mapsto\lambda^{\pm}(f,\mu)$.

In the $C^\infty$ case, these functions are semi-continuous. Specifically,
suppose $f_k,f\in\Diff^\infty(M)$, $\nu_k$ are ergodic $f_k$-invariant measures,  $f_k\to f$ in $C^\infty$ and $\nu_k\to\mu$ \weakstar. Then
\begin{equation}\label{usc-ineq}
\begin{aligned}
\limsup\limits_{k\to\infty}h(f_k,\nu_k&)    \;\; \leq\;  h(f,\mu),\\
\limsup\limits_{k\to\infty}\lambda^+(f_k,\nu_k) \leq \lambda^+(f,\mu), \quad &\liminf\limits_{k\to\infty}\lambda^-(f_k,\nu_k) \geq \lambda^-(f,\mu).
\end{aligned}\end{equation}
(See Section \ref{ss.related-work} for the history of these results.)

By \eqref{usc-ineq} and Ruelle's inequality, if $\limsup_k h(f_k,\nu_k)>0$, then $\lambda^+(f,\mu)\geq h(f,\mu)\geq \limsup_k h(f_k,\nu_k)>0$, and
$$ \frac{\limsup\limits_{k\to\infty}\lambda^+(f_k,\nu_k)}{\lambda^+(f,\mu)} \ ,\ \frac{\limsup\limits_{k\to\infty}h(f_k,\nu_k)}{h(f,\mu)}\in (0,1].
$$
We call these quantities the {\em discontinuity ratios}, and think of them as measures for the difference between the two sides of the inequalities in \eqref{usc-ineq}.

We will provide inequalities relating the discontinuity ratio of the entropy to the discontinuity ratio of $\lambda^+$. (It is enough to consider $\lambda^+$, because $\lambda^-(f,\mu)=-\lambda^+(f^{-1},\mu)$.)

\subsection{The ergodic $C^\infty$ case.}
Our results are simplest and strongest when the maps are $C^\infty$ and the limiting measure is ergodic:

\begin{maintheorem}\label{t.ergodic}
For every $k\geq 1$, let $f_k\in\Diff^\infty(M)$ and
let $\nu_k$ be  an $f_k$-ergodic invariant measure. Suppose
 \begin{itemize}
   \item[--] $\lim_k \lambda^+(f_k,\nu_k)$ and $\lim_k h(f_k,\nu_k)$ exist and are  positive,
  \item[--] $f_k$ converge in the $C^\infty$ topology to a diffeomorphism $f\in\Diff^\infty(M)$,
    \item[--] $\nu_k$ converge \weakstar\ to a probability measure $\mu$ (necessarily $f$-invariant).
\end{itemize}
If $\mu$ is $f$-ergodic, then
$\displaystyle
    \lim_{k\to\infty}\frac{h(f_k,\nu_k)}{{h(f,\mu)}} \leq \lim_{k\to\infty}  \frac{\lambda^+(f_k,\nu_k)}{\lambda^+(f,\mu)}.
 $
\end{maintheorem}

\medskip

The following result is an immediate consequence of this and \eqref{usc-ineq}, and was the original aim of our work:
\begin{corollary}
For every $k\geq 1$,
let $f_k\in\Diff^\infty(M)$ and let $\nu_k$ be an $f_k$-ergodic invariant measure. Suppose
$f_k\to f$ in the $C^\infty$ topology, and  $\nu_k\to\mu$ \weakstar\, where
$\mu$ is an $f$-ergodic invariant measure with positive entropy. If $h(f_k,\nu_k)\to h(f,\mu)$, then:
$$\lambda^+(f_k,\nu_k)\to  \lambda^+(f,\mu) \text{ and }
\lambda^-(f_k,\nu_k)\to  \lambda^-(f,\mu).$$
\end{corollary}
\noindent
The result mentioned in the introduction is the special case $f_1=f_2=\cdots=f$.

\subsection{The ergodic $C^r$ case}\label{ss.Cr}
The following result extends Theorem \ref{t.ergodic} to the $C^r$ case up to an extra term similar to what happens in Yomdin's theory.
We define the {\em asymptotic dilation} of a $C^1$ map $f:M\to M$ to be
\begin{equation}\label{asymp-dil}
 \lambda(f):=\lim_{n\to +\infty}\frac1n\log\|D f^n\|_\supnorm,\text{ where }   \|Df\|_\supnorm:=\sup_{x\in M}\underset{v\neq 0}{\sup_{v\in T_x M}} \frac{\|Df_x v\|_{f(x)}}{\|v\|_x}.
\end{equation}
Since $M$ is compact, $\lambda(f)$ is independent of the choice of a Riemannian metric $\|\cdot\|_x$.

\begin{maintheorem}\label{t.ergodic-Cr}
Fix $r>2$.
For every $k\geq 1$, let $f_k\in\Diff^r(M)$ and
let $\nu_k$ be  an $f_k$-ergodic invariant measure. Suppose
 \begin{itemize}
   \item[--] $\lim\limits_{k\to\infty} \lambda^+(f_k,\nu_k)$ and $\lim\limits_{k\to\infty} h(f_k,\nu_k)$ exist
   and are positive,
  \item[--] $f_k\to f$ in the $C^r$-topology,
    \item[--] $\nu_k$ converge \weakstar\ to a probability measure $\mu$.
\end{itemize}
If $\mu$ is $f$-ergodic and has positive entropy, then
\begin{equation*}
\frac{\lim\limits_{k\to\infty}h(f_k,\nu_k)}{h(f,\mu)}-\frac{1}{h(f,\mu)}\frac{{\lambda}(f)+ \lambda(f^{-1})}{r-1}
\leq\;
\frac{\lim\limits_{k\to\infty}\lambda^+(f_k,\nu_k)}{\lambda^+(f,\mu)}.
\end{equation*}
\end{maintheorem}
\noindent
By~\cite{Burguet}, the condition $h(f,\mu)>0$ holds once $\lim\limits_{k\to\infty} h(f_k,\nu_k)> \frac{\min(\lambda(f),\lambda(f^{-1}))}{r}$.

As we will explain in Section~\ref{s.comments}, the smoothness index $r$ above does not need to be an integer.

\subsection{The non-ergodic case}
The  assumption that the limiting measure is ergodic is often difficult to check, and we now explain what can be said in its absence  (a more general but also more technical result, Theorem~\ref{t.ergodic4}, will be given in section~\ref{s.conclude}).

\begin{maintheorem}\label{t.ergodic3}
Fix $r>2$.
For every $k\geq 1$, let  $f_k\in\Diff^r(M)$ and
let $\nu_k$ be an $f_k$-ergodic invariant measure. Suppose
 \begin{itemize}
   \item[--] $\lim\limits_{k\to\infty}\lambda^+(f_k,\nu_k)$ and $\lim\limits_{k\to\infty} h(f_k,\nu_k)$ exist
and are positive,
  \item[--] $f_k\to f$ in the $C^r$-topology,
  \item[--] $\nu_k\to \mu$ \weakstar\ for some $f$-invariant probability measure $\mu$ (perhaps non-ergodic).
 \end{itemize}
If $\lim\limits_{k\to\infty} h(f_k,\nu_k)>\tfrac{\lambda(f)+\lambda(f^{-1})}{r-1}$ then
there exist $\beta\in (0,1]$, and  two  $f$-invariant probability measures $\mu_0,\mu_1$
with $h(f,\mu_1)>0$ such that
$\mu=(1-\beta)\mu_0+\beta \mu_1$  and
\begin{equation}\label{e.eqC2}
\frac{\lim\limits_{k\to\infty}h(f_k,\nu_k)}{h(f,\mu_1)} \;-\;\frac{1}{h(f,\mu_1)}\frac{{\lambda}(f)+ \lambda(f^{-1})}{r-1}
\;\leq\;
\beta=\frac{\lim\limits_{k\to\infty}\lambda^+(f_k,\nu_k)}{\lambda^+(f,\mu_1)}.
\end{equation}
Moreover  $\lambda^+(f,x)>0\geq \lambda^-(f,x)$  for $\mu_1$-a.e. $x\in M$.
\end{maintheorem}

\noindent
Note that for $C^\infty$ diffeomorphisms, the
term $(\lambda(f)+\lambda(f^{-1}))/(r-1)$ can be replaced by zero, because the theorem can be applied with $r$ arbitrarily large.

The decomposition $\mu=(1-\beta)\mu_0+\beta \mu_1$ depends on the sequences $(f_k)_{k\geq 1}$, $(\nu_k)_{k\geq 1}$, and not just on their limits.
We give a  heuristic description of this decomposition 
in Section~\ref{heuristic-neutral}.

\subsection{Additional comments}\label{s.comments}
We now supplement Theorems~\ref{t.ergodic}, \ref{t.ergodic-Cr} and~\ref{t.ergodic3} by some examples, comments, strengthenings, and generalizations. The proofs can be found in Section~\ref{s.supplements}.

\subsubsection{Examples of discontinuities}
Theorem~\ref{t.ergodic} is sharp in the following sense:

\begin{example}\label{e.example}
For every $0<\alpha\leq\beta\leq 1$, there exist a $C^\infty$ surface diffeomorphism $f$ and a sequence of ergodic and invariant measures $\nu_k$ converging \weakstar\;  to an ergodic invariant probability measure $\mu$
such that $\lim_k h(f,\nu_k)>0$, and
  $$
    \lim_{k\to\infty} h(f,\nu_k)/h(f,\mu) = \alpha \ ,\  \lim_{k\to\infty} \lambda^+(f,\nu_k)/\lambda^+(f,\mu)=\beta.
  $$
\end{example}

\subsubsection{Variant inequality}
Theorem~\ref{t.ergodic} does not use the symmetry between a diffeomorphism and its inverse. When $0<-\lambda^-(f,\mu)<\lambda^+(f,\mu)$, this symmetry yields a sharper bound:
\begin{corollary}\label{c.variant} Under the assumptions of Theorem~\ref{t.ergodic},
\begin{equation}\label{stronger-ineq}
\lambda^+(f,\mu)-\lim_{k} \lambda^+(f_k,\nu_k)
      \leq |\lambda^-(f,\mu)|\left(1-\frac {\lim_{k} h(f_k,\nu_k)}{h(f,\mu)}\right).
\end{equation}
\end{corollary}

\subsubsection{Sequences of non-ergodic measures}
Our results can be extended to the case when  the invariant measures $\nu_k$ are not ergodic, but this requires stronger assumptions on the Lyapunov exponents of $\nu_k$, which in the non-ergodic case are functions and not constants.
See Corollary~\ref{c.non-ergodic-nu}.  

\subsubsection{Lifted version}
 $f$ induces a dynamical system $\hf$ on the projective tangent bundle $\hM$, see Section~\ref{ss.projective}. It turns out that $\mu=(1-\beta)\mu_0+\beta\mu_1$ is a projection of a decomposition of a limit point $\hmu:=\lim\hnu_{k_i}$, where $\hnu_k$ are lifts of $\nu_k$ to $\hf$-invariant measures on $\hM$. The decomposition of $\hmu$ contains more information than the decomposition of $\mu$, and leads to a stronger statement,  Theorem~\ref{t.ergodic4}, in Section~\ref{s.conclude}. This strengthening is essential to the proof of 
Corollary~\ref{c.non-ergodic-nu} on the case when $\nu_k$ are not ergodic.

\subsubsection{Convergence of $C^r$-diffeomorphisms}
In finite differentiability, Theorem~\ref{t.ergodic4} allows a weaker convergence assumption (denoted by $f_k\rbd{r}f$, see
Section~\ref{ss.size}), and $r$ does not have to be an integer.

\subsubsection{Entropy upper semi-continuity in the $C^r$ case}
It is well-known that for $C^r$-diffeo\-morphisms the entropy may fail to be upper semi-continuous,
but the defect in upper semi-continuity can be bounded (see the discussion in Section~\ref{ss.related-work} below). This bound manifests itself in  Theorems~\ref{t.ergodic-Cr} and~\ref{t.ergodic3} in the expression $\frac{\lambda(f)+\lambda(f^{-1})}{r-1}$.
But our proof gives a slightly stronger bound $\frac{\lambda(\widehat{f})}{r-1}$, in terms of the dynamics of $\hf$ on the projective tangent bundle, see \S\ref{ss.projective}, \S\ref{ss.size-lift}, and  Theorem~\ref{t.ergodic4}.

In some special cases, {even this stronger} bound can be improved.
For instance,  if $\lambda^+(f_k,\nu_k)\to\lambda^+(f,\mu)$, then
\begin{equation}\label{e.improved-eq.2}
\lim_{k\to\infty} h(f_k,\nu_k)\leq h(f,\mu)+\frac{\min\{\lambda(f),\lambda(f^{-1})\}}{r},
\end{equation}
which is stronger than the assertion of Theorem \ref{t.ergodic-Cr}.
When $f_1=f_2=\cdots =f$, \eqref{e.improved-eq.2} is a refinement of a classical inequality of Yomdin and Newhouse; it follows from bounds on the tail entropy, which were first written explicitly in \cite{Burguet}, and which
are consequences of {the} Downarowicz variational principle \cite{Downarowicz-Entropy}.
When the sequence $(f_k)$ is non-constant, it follows  from a bound on  robust tail entropy, which can be shown using techniques in \cite{BurguetFibered2017}. We thank David Burguet for explaining this to us.

\subsection{Applications to dimension theory and to SRB measures}\label{ss.applications}
The Hausdorff dimension of a Borel measure $\mu$ on $M$ is defined to be the infimum of the Hausdorff dimensions of all Borel sets of full $\mu$-measure \cite[p.115]{young}. We denote this by $\mathrm{HD}(\mu)$.

\begin{corollary}\label{c.HD}
For every $k\geq 1$, let  $f_k\in\Diff^\infty(M)$ and
let $\nu_k$ be an $f_k$-ergodic invariant measure. Suppose
$f_k\to f$ in the $C^\infty$ topology, $\nu_k\to\mu$ \weakstar, and
{$\liminf\limits_{k\to\infty} h(f_k,\nu_k)>0$}. If $\mu$ is ergodic, then
$
\displaystyle{\limsup\limits_{k\to\infty}\mathrm{HD}(\nu_k)\leq \mathrm{HD}(\mu).}
$
\end{corollary}
\begin{proof}
One can always take a subsequence such that
$\lim_k h(f_k,\nu_k)>0$, $\lim_k \lambda^+(f_k,\nu_k)>0$ and $\lim_k \lambda^-(f_k,\nu_k)<0$ exist,
and such that $\mathrm{HD}(\nu_k)$ converges to the limsup of the initial sequence.
In \cite{young}, Young  gives the following formula for the Hausdorff dimension:
 \begin{equation}\label{eq-Young}
    \mathrm{HD}(\mu)=h(f,\mu)\left(1/\lambda^+(f,\mu)+1/|\lambda^-(f,\mu)|\right).
 \end{equation}
Writing
$
\displaystyle\lim_{k}\frac{h(f_k,\nu_k)}{\lambda^+(f,\nu_k)}=\left(\frac{h(f,\mu)}{\lambda^+(f,\mu)}\right)
\biggl(\frac{\lim_k h(f_k,\nu_k)}{h(f,\mu)}\biggr)\biggl(\frac{\lim_k \lambda^+(f_k,\nu_k)}{\lambda^+(f,\mu)}\biggr)^{-1}
$, we conclude from Theorem \ref{t.ergodic} that
$\lim_{k}\frac{h(f_k,\nu_k)}{\lambda^+(f,\nu_k)}\leq \frac{h(f,\mu)}{\lambda^+(f,\mu)}$. Working with $f_k^{-1}$ and $f^{-1}$ we obtain in a similar way that $\lim_{k}\frac{h(f_k,\nu_k)}{|\lambda^-(f,\nu_k)|}\leq \frac{h(f,\mu)}{|\lambda^-(f,\mu)|}$.
\end{proof}

Suppose $f\in\Diff^2(M)$. An $f$ invariant probability measure is called a {\em Sinai-Ruelle-Bowen (SRB) measure}, if $\lambda^+(f,x)>0$ $\mu$-a.e.,  and if the conditional measures of $\mu$ obtained by disintegrating it with respect to a measurable partition subordinated to the lamination by unstable manifolds are a.e. absolutely continuous with respect to the induced Riemannian measures. We recall two classical characterizations of SRB measures from \cite{Ledrappier-Young-I} and  \cite{Ledrappier-Young-II}.

Suppose $\lambda^+(f,x)>0$ $\mu$-almost everywhere. The \emph{geometric pressure of $\mu$} is $$P^u(\mu)=P^u(f,\mu):=h(f,\mu)-\lambda^+(f,\mu).$$
By Ruelle's inequality, $P^u(\mu)\leq 0$ for all invariant probability measures, and
by Ledrappier-Young \cite{Ledrappier-Young-I}, $P^u(\mu)=0$ iff $\mu$ is an SRB measure.

Next suppose $\mu$ is ergodic and $\lambda^+(f,\mu)>0$. The {\em unstable dimension} of $\mu$  is
$$\displaystyle\delta^u(f,\mu):=\frac{h(f,\mu)}{\lambda^+(f,\mu)}.$$
By Ledrappier-Young Theory \cite{Ledrappier-Young-II}, $\delta^u(f,\mu)$ is the a.s. value of the Hausdorff dimension of the conditional measures of $\mu$ on a measurable partition subordinated to the local unstable manifolds, and $\mu$ is an SRB measure iff $\delta^u(\mu)=1$.

Theorem \ref{t.ergodic} immediately implies the following:
\begin{corollary}
For every $k\geq 1$, let  $f_k\in\Diff^\infty(M)$ and
let $\nu_k$ be an $f_k$-ergodic invariant measure. Suppose
$f_k\to f$ in the $C^\infty$ topology, $\nu_k\to\mu$ \weakstar, and
{$\lim h(f_k,\nu_k)$ exists and is positive}.
If $\mu$ is ergodic, then:
\begin{enumerate}
\item $\limsup\limits_{k\to\infty}  \delta^u(\nu_k)\leq  \delta^u(\mu)$
\item $\limsup\limits_{k\to\infty} P^u(f_k,\nu_k)\leq  P^u(f,\mu)\frac{\lim\limits_{k\to\infty}
    h(f_k,\nu_k)}{h(f,\mu)}$
\item if $P^u(f_k,\nu_k)\to 0$ or $\delta^u(\nu_k)\to 1$, then $\mu$ is an SRB measure.
\end{enumerate}
\end{corollary}
\begin{proof}
The proof of Corollary \ref{c.HD} also shows (1);  $(1)\Rightarrow (2)$ is a simple algebraic manipulation; and (1)+(2)$\Rightarrow(3)$ by the Ledrappier-Young characterizations of SRB measures as measures with zero pressure and/or  unstable dimension equal to one. \end{proof}

We can remove the assumption that $\mu$ is ergodic, using Theorem \ref{t.ergodic3}:
\begin{corollary}
For every $k\geq 1$, let  $f_k\in\Diff^\infty(M)$ and
let $\nu_k$ be an $f_k$-ergodic invariant measure. Suppose
$f_k\to f$ in the $C^\infty$ topology, $\nu_k\to\mu$ \weakstar, and
{$\lim h(f_k,\nu_k)$ exists and is positive}.
Then there are ergodic components $\mu',\mu''$ of $\mu$ satisfying $\lambda^+(\mu'), \lambda^+(\mu'')>0$
and
 $$
    \delta^u(\mu'') \geq \lim_{k\to\infty}  \delta^u(\nu_k)    \text{ , }  P^u(f,\mu') \geq \lim_{k\to\infty} P^u(f_k,\nu_k).
 $$
\end{corollary}
\begin{proof}
Consider the decomposition $\mu=(1-\beta)\mu_0+\beta\mu_1$ and take suitable ergodic components of $\mu_1$.
\end{proof}

{The following statement implies  the corollary in the introduction.}
\begin{corollary}
Let $f$ be a $C^\infty$ diffeomorphism of a compact smooth  surface without boundary, and fix some $h>0$. The following are equivalent:
 \begin{enumerate}[\quad(i)]
  \item $f$ admits an SRB measure with entropy at least $h$;
  \item $\sup \{ P^u(f,\mu): \mu \text{ ergodic measure for $f$ s.t. }h(f,\mu)\geq h\}=0$;
  \item $\sup \{ \delta^u(\mu): \mu \text{ ergodic measure for $f$ s.t. }h(f,\mu)\geq h\} = 1$.
 \end{enumerate}
\end{corollary}
\begin{proof}
 (i)$\Rightarrow$(ii) is due to Ledrappier \& Strelcyn \cite{Ledrappier-Strelcyn}, and  (ii)$\Leftrightarrow$(iii) is trivial.

To see (ii)$\Rightarrow$ (i), we take a sequence of $f$-ergodic measures $\nu_k$  with $h(f,\nu_k)\to h'\geq h$ and $P^u(f,\nu_k)\to 0$.  We select  a  subsequence $(\nu_{k_i})_{i\geq 1}$ s.t. $\lambda^+:=\lim \lambda^+(\nu_{k_i})$ and $\mu:=\lim \nu_{k_i}$ exist. By Ruelle's inequality, $\lambda^+\geq h'>0$, and by Theorem \ref{t.ergodic3}, $\mu=\beta\mu_1+(1-\beta)\mu_0$ where
$$
\lim_{i\to\infty}\lambda^+(f,\nu_{k_i})=\beta\lambda^+(f,\mu_1)
\text{, and }\lim_{i\to\infty}h(f,\nu_{k_i})\leq \beta h(f,\mu_1).
$$
It follows that $0=\lim_{i\to\infty}P^u(f,\nu_{k_i})\leq \beta P^u(f,\mu_1)\leq 0$, where the last inequality is Ruelle's inequality for $\mu_1$.
Since $\beta \lambda^+(f,\mu_1)=\lim\lambda^+(f,\nu_{k_i})=\lambda^+\neq 0$, it must be the case that $\beta\neq 0$, and $P^u(f,\mu_1)=0$. If $\mu_1=\int\mu_\xi' d\xi$ is the ergodic decomposition of $\mu_1$, then $\int P^u(f,\mu_\xi')d\xi=P^u(f,\mu_1)=0$. By  Ruelle's inequality, the integrand is non-positive, so  $P^u(f,\mu'_\xi)=0$ for $\mu_1$-a.e. ergodic component. At the same time, $h(f,\mu_1)\geq \beta^{-1}\lim h(f_k,\nu_k)>h'$, so some  of these ergodic components must have entropy $\geq h'$. Thus $\mu_1$ has ergodic components with entropy bigger than $h$, and zero pressure. By Ledrappier-Young Theory, these are SRB measures with entropy bigger than $h$, and (i) is proved.
\end{proof}
Notice that it is essential in this proof to be able to deal with non-ergodic limits, since we have no control of $\lim\nu_k$.

\subsection{Related works}\label{ss.related-work}
In this paper we relate the continuity properties of the entropy to those of the Lyapunov exponents. The continuity of these objects has been studied separately before in several works, which we now recall.

\medskip
\paragraph{\it Entropy.}
In general, the entropy map $(f,\mu)\mapsto h(f,\mu)$ is not lower semi-continuous,   even in the uniformly hyperbolic case. For example, it is easy to construct sequences of atomic measures (with zero entropy) on a basic set,  which converge to limits  with positive entropy.

However, for $C^\infty$ diffeomorphisms on compact manifolds, the entropy map is upper semi-continuous: This is due to  Newhouse \cite{Newhouse1990}.
For $C^r$ diffeomorphisms with finite $r$, even upper semi-continuity may fail (for examples in dimension four see \cite{misiurewicz}, and for examples in dimension two see \cite{BuzziNoMax}). However,
the (additive) defect in semi-continuity:
$$D(f,\mu):=\limsup_{(g,\nu)\to(f,\mu)} h(g,\nu)-h(f,\mu)$$
can be  bounded from above
by $\min(\lambda(f),\lambda(f^{-1}))/r$,
using Yomdin theory~\cite{Newhouse1990,BuzziPhD,Burguet,BurguetFibered2017}.
{A subject of more recent interest is the loss of semi-continuity due to non-compactness. This has been studied for} countable Markov shifts \cite{Iommi-Todd-Velozo,Iommi-Todd-Velozo-USC}, geodesic flows on non-compact homogeneous spaces \cite{Einsiedler-Kadyrov-Pohl,Kadyrov-Pohl}, and geodesic flows on non-compact manifolds with negative sectional curvatures \cite{Iommi-Riquelme-Velozo,Riquelme-Velozo}.
\medskip

\paragraph{\it Lyapunov exponents.}
The {top} Lyapunov exponent map $(f,\mu)\mapsto \lambda^+(f,\mu)$ varies continuously
for uniformly hyperbolic systems on surfaces. It even depends analytically on the diffeomorphism $f$~\cite{Ruelle-exponent-continuity}.  Moreover  if $\mu_{\max}$ is the unique measure of maximal entropy of a mixing Anosov surface diffeomorphism, then \cite{Kadyrov-Effective} (see also \cite{Polo,Ruhr}) implies that $|\lambda^+(\mu_{\max})-\lambda^+(\nu)|\leq c\sqrt{|h(f,\mu_{\max})-h(f,\nu)|}$, where $c$ only depends on $f$.

In the non-uniformly hyperbolic case, the situation is different. For example,  \cite{Bochi} proves that among conservative systems
the Lyapunov exponents of the volume measure are discontinuous when the diffeomorphism varies in the $C^1$-topology, unless they vanish.

 We are not aware of other general results on
the continuity of the Lyapunov exponents
for general non-uniformly hyperbolic surface diffeomorphisms.

By contrast, much is known on the continuity of Lyapunov exponents  of random products of independent identically distributed $SL(2,\R)$ matrices, as functions of the underlying   Bernoulli process, see \cite{Furstenberg-Kifer,Bocker-Viana}.
More general  H\"older continuous matrix cocycles with holonomies are considered in~\cite{Backes-Brown-Butler},
and a higher-dimensional extension has been announced in~\cite{Viana-survey}.
\medskip

\paragraph{\it Dimension.} L.-S. Young gave the famous formula \eqref{eq-Young} for the dimension of hyperbolic  invariant measures in \cite{young} in terms of the entropy and the Lyapunov exponents of the measure. For further dimension theoretic properties of hyperbolic invariant measures, see  \cite{Barreira-Pesin-Schmeling} and \cite{Barreira-Gelfert}.
The continuity of the  dimension of invariant sets and measures for hyperbolic systems have been considered in numerous works, for instance \cite{palis-viana} proves that
basic sets on surface have a Hausdorff dimension which varies continuously with the diffeomorphism,
\cite{Barreira-Wolf} proves that the supremum of the Hausdorff dimensions of ergodic measures on such a basic set
is attained by a measure of maximal dimension
and \cite{Barreira-Gelfert} discusses some non-uniformly hyperbolic cases.

\section{A heuristic overview of the proof}

All our results follow from Theorem \ref{t.ergodic3}, and the remainder of the paper is dedicated to the proof of this theorem.
Here we give a heuristic overview of the proof, in the special case when $f_1=f_2=\cdots=f$
is a $C^\infty$ diffeomorphism.

\subsection{The origin of the discontinuities in $\lambda^+$}
As Furstenberg discovered, the Lyapunov exponents are easier to study in terms of the \emph{projective dynamics} $\hf(x,E)=(f(x),Df_x(E))$ on the projective tangent bundle
\begin{align*}
&\hM:=\{(x,E):x\in M,\ E\subset T_x M\text{ is a one-dimensional linear space}\}.
\end{align*}
Indeed by Ledrappier's work, a  Lyapunov exponent of an $f$-ergodic measure $\mu$ is  simply the integral of the continuous function
 $$
   \vf(x,E):=\log\|Df_x|_E\|,
 $$
 with respect to the lift of $\mu$ to the bundle of the associated Oseledets spaces.

Suppose $\nu_k$ are ergodic  measures with positive entropy such that $\nu_k\to\mu$ weak$^\ast$, and suppose for the moment that $\mu$ is ergodic and with positive entropy.
Since $\dim(M)=2$,  $\nu_k$ have two simple Lyapunov exponents, and there are exactly two ergodic lifts $\hnu^+_k$ and $\hnu^-_k$, one carried by the   bundle $\mathcal E^u$ of unstable Oseledets spaces,  and the other carried by the bundle $\mathcal E^s$ of the stable Oseledets spaces.
(The third bundle $\mathcal E^0$ associated to the zero exponent has measure zero for all lifts of $\nu_k$.)
Hence,
 $$
   \lambda^\pm(f,\nu_k)=\int \varphi d\hnu_k^\pm.
 $$
Suppose $\hnu^+_k$ converge weak-star on $\hM$ to an $\hf$-invariant probability
measure $\hmu$ (this is true for a subsequence). Since $\vf:\hM\to\R$ is continuous,
 $$
 \lim_{k\to\infty} \lambda^+(f,\nu_k)=\lim_{k\to\infty} \hnu_k(\vf) = \hmu(\vf).
 $$
 The limiting measure $\hmu$ is a lift of $\mu$, but this does not have to be  the lift of $\mu$ to $\mathcal E^u$, $\hmu^+$. If  $\hmu(\vf)\neq \hmu^+(\vf)$,  then $ \lim_{k\to\infty} \lambda^+(f,\nu_k)\neq \lambda^+(f,\mu)$.

It is certainly possible that $\hmu\neq \hmu^+$: The Oseledets bundle $\mathcal E^u$ carrying the lifts $\hnu_k^+$ is not necessarily bounded away from $\mathcal E^s$,  and some mass $0\leq \rho\leq1$ on $\mathcal E^u$ can escape to $\mathcal E^s$.

Escape of mass to $\mathcal E^s$ is reflected in long stretches of time when  $\nu_k$-typical orbits do not experience the exponential growth of  $\mathcal E^u$-directions predicted by $\lambda^+$. Instead, they see, temporarily, exponential decay at rate $\lambda^-$, cancelling some of the previous growth.
If $\hmu^+,\hmu^-$ denote the two ergodic lifts. we must have
 $
    \hmu=(1-\rho)\hmu^+ + \rho\hmu^-$  and thus,
 $$
    \lim_{k\to\infty} \lambda^+(f,\nu_k) = (1-\rho)\hmu^+(\vf) + \rho\hmu^-(\vf)
     = \lambda^+_\mu-\rho(\lambda^+_\mu-\lambda^-_\mu).
  $$
In the language of Theorem \ref{t.ergodic3} (and since  $\mu=\mu_1$ by ergodicity), the discontinuity ratio $\beta$ is:
 $$
   \beta := \frac{ \lim_k \lambda^+(f,\nu_k)}{\lambda^+(f,\mu)} = 1 - \rho(1-\lambda^-_\mu/\lambda^+_\mu).
 $$
(A different description of $\beta$ will be given below.) So if $\mu$ is ergodic, then $\beta$ is a function of $\rho$, whence of the amount of mass which escapes to $\mathcal E^s$.

In the case where $\mu$ is not ergodic, the different ergodic components of $\mu$ have to be considered, and some of them may have zero Lyapunov exponents. The way in which $\nu_k$-typical orbits approximate those ergodic components determine the possible cancellations. So if $\mu$ is not ergodic, then  $\beta$ may depend on the entire sequence $(\nu_k)$, not just on its limit $\mu$.

\subsection{Neutral  blocks, the decomposition of $\mu$, and the parameter $\beta$}
\label{heuristic-neutral}
 Recall the measurable $\hf$-invariant decomposition
$\mathcal{X}=\mathcal{E}^s\cup \mathcal{E}^u \cup \mathcal{E}^0$ defined by the Oseledets theorem according to the sign of the limit $\tfrac 1 n \log \|Df^n_x|_E\|$. It has full measure with respect to any $\hf$-invariant measure.
To get quantitative estimates, we select compact subsets $K^*\subset \mathcal{E}^*$, for $*\in \{s,u,0\}$,
from which the contraction, expansion, or ``central" behavior of the sequence $\|Df^n_{x}|_{E}\|$ are uniformly controlled,
and such that the $\hmu$-measure of $K:=K^s\cup K^u\cup K^0$ is close to $1$. Since each $\mathcal E^*$ is invariant, we can choose these compact sets to be nearly invariant: Points in a very small neighborhood stay close for a long time.

Hence, if $\hx_0$ is a $\hnu_k^+$-typical  point for some very large $k$, its orbit under $\hf$ spends nearly all its time close to $K$ and
every visit in a small neighborhood of $K^s\cup K^0$ is the beginning of a long period of uniform contraction (or weak expansion/contraction). One expects no entropy creation not only during this period, but also during the ``recovery period" which follows, i.e., until the expansion  predicted by the Lyapunov exponent of $\nu_k$ cancels this period of contraction (or weak expansion/contraction).

We select such long time intervals along the orbit of $\hx_0$ in the following greedy way. Fixing $\alpha>0$ small and $L$ large, an \emph{$(\alpha,L)$--neutral block} is a maximal interval of integers
$(n_0,\dots,n_0+\ell)$ such that $\ell\geq L$ and
$$\|Df^n_{f^{n_0}(x_0)}|_{E^u}\|\leq \exp(\alpha(n-n_0)) \text{ for all }0<n\leq \ell.$$
We will check that indeed, there is very little if any entropy creation during neutral blocks.

Our estimates will be in terms of the distribution of these long neutral blocks.
Let $\hx_k=(x_k,E^u(x_k))$ be $\hnu_k^+$-generic points and let $\mathfrak N_{\alpha,L}(\hx_k)$ denote the union of all $(\alpha,L)$-neutral blocks of the orbit of $\hx_k$.
In Section \ref{s.neutral} we show that it is possible to choose a subsequence $k_i\to\infty$ so that following limits
make sense \weakstar\ on $\hM$ for $(\hnu^+_{k_1}\times\hnu^+_{k_2}\times\cdots)$--a.e. $(\hx_{k_1},\hx_{k_2},\ldots)$:
\begin{align*}
&\hm_0:=\underset{L\to\infty}{\lim\limits_{\alpha\to 0}}\lim_{i\to\infty}\left(\lim_{N\to\infty} \frac{1}{N}\sum_{j\in [0,N)\cap\mathfrak N_{\alpha,L}(\hx_k)} \delta_{\hf^j(\hx_{k_i})}\right)\\
&\hm_1:=\underset{L\to\infty}{\lim\limits_{\alpha\to 0}}\lim_{i\to\infty}\left(\lim_{N\to\infty} \frac{1}{N}\sum_{j\in [0,N)\setminus\mathfrak N_{\alpha,L}(\hx_k)}\delta_{\hf^j(x_{k_i})}\right)\, .
\end{align*}
Notice that the sum of the two limits in the brackets is a.s. $\hnu_{k_i}^+$, because this is the limit of the empirical measure of $\hx_{k_i}$, and $\hx_{k_i}$ are all  a.s. $\hnu_{k_i}^+$--generic. So
$$
\hm_0+\hm_1=\lim_{k\to\infty}\hnu_k^+=\hmu.
$$

The measures $\hm_0,\hm_1$ are $\hf$-invariant. We will see that $\int \varphi d\hm_0=0$, and that $\hm_1$ is carried by $\mathcal{E}^u$. The decomposition
$\mu=\beta\mu_1+(1-\beta)\mu_0$ in Theorem \ref{t.ergodic3} is defined by
 $$
   \beta:=1-\hmu_0(\hM),\; (1-\beta)\mu_0:=\hpi_*(\hmu_0),\text{ and }\beta\mu_1:=\hpi_*(\hmu-\hmu_0),
 $$
where $\hpi:\hM\to M$ is the natural projection.
Note that $\beta$ is indeed the discontinuity ratio $\lim_k \lambda^+(f,\nu_k)/\lambda^+(f,\mu_1)$ and the quantity $1-\beta$ coincides with the fraction of the time spent in maximal neutral blocks.
The measures $\mu_0,\mu_1$ and $\beta$ depend not just on  $\hmu$, but also on the way the measures $\hnu^+_k$ accumulate on  $\hmu$.

\subsection{Upper bound on the entropy}  To complete the proof of the theorem it remains to show that
$
\displaystyle\lim_{k\to\infty}h(f,\nu_k)\leq \beta h(f,\mu_1)$. This is the heart of the proof, and where most of the difficulties lie. We use Ledrappier-Young Theory  and Yomdin Theory.
\begin{itemize}
\item \emph{Ledrappier-Young theory} bounds $h(f,\nu_k)$ by the exponential rate of growth of the minimal number of $(n,\epsilon)$-balls needed to cover a definite fraction of a local unstable manifold $W^u_{loc}(x_k)$,
where $x_k$ is a fixed $\nu_k$-typical point and the scale $\eps$ tends to zero. The ``fraction" is measured using the conditional measure $\nu^u_{x_k}$ of $\nu_k$ on $W^u_{loc}(x_k)$. In particular, it suffices to follow points $x\in W^u_\loc(x_k)$ with $T_x W^u_\loc(x_k)=\mathcal E^u(x)$.

\medskip
\item \emph{Yomdin theory} provides tools for  controlling the number  of $(n,\eps)$-balls needed for such covers, for $C^r$ maps. Instead of working with $(n,\eps)$-balls, one works with parametrized pieces of unstable manifolds which lie inside $(n,\eps)$-balls and which have uniformly bounded $C^r$ size, and Yomdin Theory allows to bound the number of such pieces.  Here the regularity assumptions on $f$ come into play.  The expression $\frac{\lambda(f)+\lambda(f^{-1})}{r-1}$ in \eqref{e.eqC2} is due to Yomdin theory (see section \ref{ss.expansion-in-neutral}).
\end{itemize}

\medskip
Let us sketch our argument for the upper bound on the entropy using the neutral blocks.
Since the unstable lift of $\nu_{x_{k}}^u$-almost every point is $\hnu^+_k$-typical, neutral blocks represent roughly a fraction $1-\beta$ of their time.
During a neutral block,  typical points on a small piece of $f^n(W^u_{loc}(x_k))$ do not separate much, therefore this piece remains small (or can be kept small by a subdivision into a small exponential number of pieces).
For the rest of the time, these subcurves follow the ergodic components of $\mu_1$, hence they experience entropic separation at an exponential rate given by $h(f,\mu_1)$. Since the time outside neutral blocks is a proportion $\beta$ of the total time, this leads to  the bound
 $$
  h(f,\nu_k)\leq \beta h(f,\mu_1).
 $$
This argument explains the link between the entropy bound and the semicontinuity defect of the Lyapunov exponents.

This sketch glosses over several difficulties. We will only comment on the main issue: How to  use non-expansion of the linearization $Df$ at $(x_k,E^u(x_k))$ during a neutral block, to infer non-expansion of the map $f$ itself on a small piece of $W^u_{loc}(x_k)$  during this neutral block. The difficulty is in controlling $Df$ on $(x'_k,E^u(x'_k))$ for $x_k'$ close to $x_k$.

\subsection{Control of the expansion during neutral blocks}\label{ss.expansion-in-neutral}
This is one of the most delicate points in the proof. To deduce the non-expansion of the small piece of the unstable manifold containing this point, we need to know that not only the diameter of this curve is small but that its tangent is almost constant too. This forces us to work with pieces of $W^u_\loc(x_k)$ whose lifts to $\hM$  are also small: The size in the fiber of $\hM$ measures the variability of the tangent directions.

How small is small enough? To use information on $Df$ to control what happens on $W^u_{loc}(x_k)$, we need the fluctuations of the tangent direction along any piece to be smaller than some $\heps>0$,   determined (mostly) by the modulus of continuity of $Df$. Using the uniform continuity of the measurable unstable bundle on a set of large measure, we find an $\eps>0$ such that if the diameter of the projection to $M$ is less than $\eps$, then  the fluctuation of the tangent is smaller than $\heps$.

The price we pay for this solution is that we need to work with different scales in $M$ and along the fibers of the bundle $\hM\to M$. This leads us to introduce \emph{fibered $(n,\varepsilon,\heps)$-balls}, and to work with Yomdin theory for $\hf:\hM\to\hM$.  When dealing with diffeomorphisms of finite regularity, there is an additional price to pay:  If $f$ is $C^r$, then  $\hf$ is only $C^{r-1}$, and this  accounts for the extra term from Yomdin theory
 $$
   \frac{\hlambda(\hf)}{r-1}\leq \frac{\lambda(f)+\lambda(f^{-1})}{r-1} \quad \text{ in \eqref{e.eqC2}}.
 $$

\subsection{Organisation of the paper}
The different ingredients of the proof appear as follows in the text.
\begin{description}
\item[\it Section \ref{s.lift}] background on tangent dynamics and Lyapunov exponents.
\item[\it Section \ref{sec-entropy}] results from Ledrappier-Young and Yomdin theories on the entropy in differentiable dynamics.
\item[\it Section \ref{s.reparametrization}] reparametrization lemmas estimating the entropy from neutral blocks and other time intervals.
\item[\it Section \ref{s.neutral}] neutral decomposition of typical orbits.
\item[\it Section \ref{s.conclude}] proof of the technical version of our main theorem.
\item[\it Section \ref{s.supplements}] proof of the remaining statements.
\end{description}

\medskip

\noindent{\it A remark on style.} Our constructions, estimating entropy for a sequence of measures converging to a nonergodic one, require many parameters. We have chosen to make the dependences as explicit as possible to help the reader check that there is no circular argument.

{\subsection{Standing notations for the duration of the paper} We collect here some notations that we will use frequently below. \begin{enumerate}[\quad $\bullet$]
\item $|X|$ or $\Card(X)$: the cardinality of a set $X$.
\item $M$ is a compact Riemannian $C^\infty$ manifold without boundary, with tangent bundle $TM$, tangent spaces $T_x M$, and Riemannian norm $\|\cdot\|_x$. Derivatives of maps $f:M\to M$ are denoted by $Df:TM\to TM$ or $Df_x:T_x M\to T_{f(x)}M$.
\item $h(f,\mu)$, $\lambda^\pm_\mu$, $\lambda^\pm_x$, $TM=E^+\oplus E^-$: the entropy, average and pointwise Lyapunov exponents, Osededets splitting
associated to a measure $\mu$ (also denoted $\lambda^{s/u}_\mu$, $E^{s/u}$ when the measure is hyperbolic of saddle type), see section~\ref{ss.exponent}.
\item $\hM$, $\hf$, $\hx=(x,E)$: the projective tangent bundle, the lift of a diffeomorphism $f$ and of a point $x$, see section~\ref{ss.projective}.
\item $B_f(x,n,\eps)$, $r_f(n,\eps,X)$: an $(n,\eps)$-Bowen ball for $f$ and the $(n,\eps)$-covering number of a set $X\subset M$, see section~\ref{sec-Bowen-Katok}.
\item $\|Df\|_\supnorm$: the sup-norm of the tangent map, see section~\ref{ss.Cr}.
\item $\lambda(f):=\lim_{n\to +\infty}\frac1n\log\|D f^n\|_{\sup}$: the asymptotic dilation.
\item $\|f\|_{C^r}$: the $C^r$ size of $f$, see section~\ref{ss.size}.
\item $Q_{r,N}(f)$: the supremum of the $C^r$ sizes of $f,f^2,\dots,f^N$ and of the $C^{r-1}$ semi-norm of $\wh f,\wh f^2,\dots,\wh f^N$ see section~\ref{ss.yomdin}.
\item If $v$ is a vector, then $\R.v:=\{tv:t\in\R\}$.
\end{enumerate}}

\section{Tangent dynamics and the semi-continuity of Lyapunov exponents}\label{s.lift}

Let $M$ be a smooth compact Riemannian surface without boundary.

{\subsection{Review of the $C^r$ size of maps}\label{ss.size}
Let $U$ be an open subset of $\R^n$.

Given $k\in \NN$, we say that a map $F:U\to\R^d$ is $C^k$ if
for all $\omega\in (\NN\cup\{0\})^n$ such that $|\omega|:=\omega_1+\cdots+\omega_n=k$,
the partial derivative
$$\partial^\omega F:=\frac{\partial^{\omega_1+\cdots+\omega_n}F}{\partial^{\omega_1}x_1\cdots\partial^{\omega_n}x_n}
$$
exists and is  continuous on $U$. For any compact subset $K\subset U$, we then define the $C^k$ \emph{size}
$$\|F\|_{C^k,K}:=\max_{1\leq |\omega|\leq k}\max_{x\in K} \|\partial^\omega F(x)\|.$$
Given $\alpha\in (0,1)$, we say that a map $F$ is $C^\alpha$ if the following quantity is finite for any compact set $K\subset U$,
 $$
     \| F \|_{C^\alpha,K} := \sup_{\tiny\begin{array}{c} {x,y\in K}\\x\ne y\end{array}} \frac{\|F(x)-F(y)\|}{\|x-y\|^\alpha}.
 $$
Given $r>1$ which is not an integer, we decompose it as $r=k+\alpha$, with $k=\lfloor r \rfloor$ and $\alpha\in (0,1)$.
We say that $F$ is $C^r$ if it is $C^k$ and each partial derivative $\partial^\omega F$, $|\omega|=k$ is $C^\alpha$.
For any compact set $K\subset U$, we define the $C^r$ size
$$\| F \|_{C^r,K}:=\|F\|_{C^k,K} + \max_{|\omega|=k} \| \partial^\omega F\|_{C^\alpha, K}.$$

Let $\Omega$ be a compact subset of $\R^n$ which is equal to the closure of its interior (we mostly need $[0,1]^n$). A map $F:\Omega\to\R^d$ is $C^r$ if $F$ has a $C^r$ extension to an open neighborhood of $\Omega$.
In this case, the {\em $C^r$ size of $F$ on $\Omega$} is
$$
\|F\|_{C^r}:=\;\sup_{K\subset \mathrm{int}(\Omega)}\;\|F\|_{C^r,K}.
$$
This (finite) quantity is independent of the extension of $F$ to the neighborhood of $\Omega$. Notice that the $C^r$ size of a constant function is zero.
\medskip

A {\em $C^r$ structure} on a smooth manifold $N$  is defined by a maximal atlas $\mathfrak A$ with $C^r$ changes of coordinates. A smooth manifold equipped with a $C^r$ structure $\mathfrak A$ is called a $C^r$ manifold. A finite subset of $\mathfrak A$ which covers $N$ is called a $C^r$ atlas of $N$.

Let $N_1,N_2$ be two compact $C^r$ manifolds (later this will be $M$, $\hM$ or the circle $S^1$), and let $\mathcal A_i$ be  finite  $C^r$ atlases of $N_i$.
Let $\Omega$ be a compact subset of $N_1$ equal to the closure of its interior.
We say that $f\colon \Omega\to N_2$ is a {\em $C^r$ map} if
each map $\chi_2^{-1}\circ f\circ\chi_1$, where $\chi_i$ ranges over $\mathcal A_i$, is $C^r$. The \emph{$C^r$ size} of $f$ is:
$$
\|f\|_{C^r}:=\max_{\chi_1\in\mathcal A_1, \chi_2\in\mathcal A_2}\|\chi_2^{-1}\circ f\circ\chi_1\|_{C^r}<\infty.
$$
Again, the constant map has size zero.

The quantity $\|f\|_{C^r}$ depends on the choice of atlases $\mathcal A_i$, but if $N_i$ are compact, then finite atlases
induce  equivalent $C^r$ sizes. In case $N_1=S^1$, we will always use the Euclidean atlas.
\medskip

Suppose $f_k,f\in\Diff^r(M)$ and $1\leq r<\infty$.
We will say that $f_k$ converges to $f$  {\em uniformly in a $C^r$-bounded way}, if
$f_k\to f$ uniformly, and
$\sup_{k\geq1} \|f_k\|_{C^r} < \infty.
$
We write in this case
$$f_k\rbd{r}f.$$
If $M$ is compact, $f_k,f\in\Diff^\infty(M)$, and $f_k\to f$ in $C^\infty$, then $f_k\rbd{r}f$ for all $r$ finite.

The Arzela-Ascoli theorem implies the following.
\begin{lemma}\label{l-ArzelaAscoli}
{Let $N_1,N_2$ be compact $C^r$ manifolds},  and $f,f_1,f_2,\dots\colon N_1\to N_2$ be a collection of $C^r$ maps such that $(f_k)$ converges to $f$ uniformly, and $\sup_k \|f\|_{C^r}<\infty$.
Then $(f_k)$ converges to $f$ in the $C^\ell$-topology for any $\ell<r$, $\ell\in\NN$.

{Thus, if for some real $r>1$ {s.t. $r\not\in\NN$}, $f,f_1,f_2,\dots\in \Diff^r(M)$ where $M$ is a compact manifold,  then  $f_k\rbd{r} f$ implies that $f_k\to f$ in the $C^{\lfloor r\rfloor}$-topology.}
\end{lemma}
}

\subsection{The projective tangent bundle}\label{ss.projective}  Let
$$
P_x M:=\{E: E\text{ is a one-dimensional linear subspace of }T_x M\}.
$$
$P_x M$ is the quotient of $T_x M\setminus\{0\}$ by the  equivalence relation
$v\sim w$ $\iff$  $\exists\lambda\ne0,\; v=\lambda w$.
It can also be viewed as the image of $\{v\in T_x M:\|v\|_x=1\}$ by the two-to-one map $v\mapsto \mathrm{Span}\{v\}$. These identifications allow us to endow $P_x M$ with a topology and with a smooth structure, and to identify the tangent spaces $T_E(P_xM)$ with $\{w\in T_xM: w\perp E\}$.  
We can also pull back the induced Riemannian inner product on $\{w\in T_xM: w\perp E\}$ to an inner product on $T_E(P_x M)$. This endows $P_x M$ with a Riemannian structure.
The resulting Riemannian distance on $P_xM$ is simply $\dist(E_1,E_2)=|\measuredangle(E_1,E_2)|$.
With this structure, $P_x M$ is isometric to the circle with perimeter $\pi$.

The {\em projective tangent bundle} (or just ``projective bundle") of $M$ is the bundle
$(\hM,{\hpi},M)$ where $\hpi:\hM\to M$ is the natural projection $\hpi(x,E)=x$, and
$$
\hM:=\{(x,E):x\in M,\text{$E$ is a one-dimensional linear subspace of $T_x M$}\} =\bigsqcup_{x\in M} P_xM.
$$
$\hM$ is a smooth compact three-dimensional  manifold. We endow it with the Riemannian metric $\sqrt{ds^2+d\theta^2}$, where $ds$ is the length element on $M$ and $d\theta$ is the length element on $P_x M$.
Points in $\hM$ will  be denoted by $\hx=(x,E)$.

Let  $f:M\to M$ be a $C^1$ diffeomorphism. The {\em canonical lift} of $f$ is the  homeomorphism $\hf:\hM\to\hM$ given by
\begin{equation}\label{canonical-lift}
\hf(x,E)=(f(x),Df_x(E)).
\end{equation}
If $f$ is of class $C^r$, then $\hf$ is of class $C^{r-1}$.
Notice that
 $\hpi\circ \hf=f\circ\hpi$, and $$\hf^n(x,E)=(f^n(x),(Df^n)_x(E)).$$

Every $\hf$-invariant probability measure $\hnu$ on $\hM$ projects to an $f$-invariant probability measure $\nu$ on $M$ given by
$$
\nu:=\hpi_\ast(\hnu):=\hnu\circ\pi^{-1}.
$$
We call $\nu$ the {\em projection} of $\hnu$, and $\hnu$ a {\em lift} of $\nu$. In what follows, when we ``lift", we always mean an $\hf$-invariant lift. The following lemma is a well-known consequence of the compactness of $\widehat M$.

\begin{lemma}\label{Lemma-lift}\begin{enumerate}[(1)]
\item Every $f$-invariant probability measure $\nu$ has at least one lift $\hnu$.
\item If $\nu$ is $f$-ergodic and $\hnu$ lifts $\nu$, then a.e. ergodic component of $\hnu$ is a lift of $\nu$.
\end{enumerate}
Hence every ergodic $f$-invariant probability measure has at least one ergodic lift.
\end{lemma}

\subsection{Review of Lyapunov exponents}\label{ss.exponent}  We review some facts on Lyapunov exponents {in dimension two} (see \cite[Theorems 3.12~and~3.14]{Viana-book}).
Suppose $f\in\Diff^1(M)$ and $\mu$ is an $f$-invariant Borel probability measure.
Oseledets' theorem asserts that for $\mu$-a.e. $x$, $\lim_{|n|\to\infty}\frac1n\log\|Df^n_x v\|$ exists for all $v\in T_x M\setminus\{0\}$. The possible values of the limit are called the {\em Lyapunov exponents} of $x$. There are at most two such values. We denote them by $\lambda^+(f,x)$, $\lambda^-(f,x)$, or
$\lambda^+_x$, $\lambda^-_x$, with the convention
$$
\lambda^+(f,x)\geq \lambda^-(f,x).
$$
If $\lambda^+(f,x)\neq \lambda^-(f,x)$ then  $T_x M=E^+(x)\oplus E^-(x)$ where
 $$
     E^\pm(x):=\left\{v\in T_xM\setminus \{0\}:\lim_{|n|\to\infty} \frac1n\log\|Df^n_x v\|=\lambda^\pm(f,x)\right\}\cup\{0\}.
 $$
{The decomposition} $T_x M=E^+(x)\oplus E^-(x)$ is called the {\em Oseledets splitting}.

If $\mu$ is {\em ergodic},
the functions $\lambda^+(f,x)\geq\lambda^-(f,x)$ are equal $\mu$-almost everywhere to constants called the
{\em Lyapunov exponents} of $\mu$ and denoted by $\lambda^+(f,\mu),\lambda^-(f,\mu)$.
If $\mu$ is not ergodic,  the Lyapunov exponents of $\mu$ are defined by
$$\lambda^{+}(f,\mu):=\int \lambda^{+}(f,x)d\mu(x),\;
\lambda^{-}(f,\mu):=\int \lambda^{-}(f,x)d\mu(x).$$
In both cases, $\lambda^+(f,\mu)\geq \lambda^-(f,\mu)$.

By the subadditive ergodic theorem, the {largest Lyapunov exponent also satisfies}
\begin{equation}\label{e.subbadditive}
    \lambda^+(f,\mu) = \lim_{n\to\infty} \frac1n\int \log \|Df^n_y\| \,d\mu(y)=\inf_{n} \frac1n\int \log \|Df^n_y\| \,d\mu(y).
 \end{equation}

Throughout this paper, an ergodic invariant probability measure $\mu$ is called \emph{hyperbolic} if one of its Lyapunov exponents is positive, and the other is negative (sometimes this is called hyperbolic of saddle-type).
If $\mu$ is hyperbolic, we sometimes write $\lambda^u=\lambda^+$, $\lambda^s=\lambda^-$, $E^u=E^+$ and $E^s=E^-$.

\subsection{Semi-continuity of Lyapunov exponents}\label{s.lyapunov-exp}
We will use the  dynamics of the projective tangent bundle to study the semi-continuity properties of $(f,\mu)\mapsto \lambda^{+}(f,\mu)$ (and by symmetry of $(f,\mu)\mapsto \lambda^{-}(f,\mu)$).
 The principal tool is the function
$$
\vf:\hM\to\RR\ , \vf(x,E)=\log\|Df_x|_E\|.
$$
Notice that if $f$ is a $C^1$ diffeomorphism, then $\vf$ is bounded and uniformly continuous.
We will make frequent use of the following  identity:
\begin{equation}\label{varphi-sums}
\log\|Df^n|_E\|=\sum_{k=0}^{n-1}(\vf\circ\hf^k)(x,E)\ \ \ ((x,E)\in \hM).
\end{equation}
This is because $E$ is a one-dimensional subspace of $T_x M$, and therefore by the chain rule
$
\|Df^n_x|_E\|=\prod_{k=0}^{n-1}\|Df_{f^{k}(x)}|_{Df^{k}_x(E)}\|
=\prod_{k=0}^{n-1}\exp[(\vf\circ\hf^k)(x,E)]$. Equation \eqref{varphi-sums} presents the  {subadditive} cocycle $\log\|Df^n|_E\|$ for $f$ as an {additive}  cocycle for $\hf$.
See \cite[Prop. 5.1 on p. 328]{Ledrappier1982} for a proof of a more general fact (and \cite[Lemma 8.7]{Furstenberg-1963} for the first use of a related idea).

\begin{lemma}\label{Lemma-Lyap}
Suppose $f\in\Diff^1(M)$, and $\mu$ is an ergodic $f$-invariant probability measure.
 Then:
  \begin{enumerate}[(1)]
   \item $\lambda$ is a Lyapunov exponent of $\mu$ iff $\mu$ has an $\hf$-ergodic lift $\hmu$ s.t. $\int\vf d\hmu=\lambda$;

\item If $\mu$ has two different Lyapunov exponents, then it has exactly two ergodic $\hf$--invariant  lifts:
 $$\hmu^+:=\int_{\hM}\delta_{(x,E^+(x))}d\mu(x)\text{, and } \hmu^-:=\int_{\hM}\delta_{(x,E^-(x))}d\mu(x).$$
{Moreover}  $\int \vf d\hmu^{\pm}=\lambda^{\pm}(f,\mu)$.
  \end{enumerate}
\end{lemma}

For any $\hf$-invariant probability measure, it will be convenient to denote
$$\hlambda(\hf,\hmu):=\int_{\hM} \varphi d\hmu .$$
When $\mu$ is hyperbolic, the lifts $\hmu^+, \hmu^-$ are called the unstable and stable lifts of $\mu$.

\begin{proof}
{If $\mu$ has equal Lyapunov exponents, (1) follows from eq. \eqref{varphi-sums} and the ergodic theorem. Otherwise, by Oseledets theorem, there are two a.e. defined sections $x\mapsto E^{\pm}_x$ s.t. $Df_x E^{\pm}_x=E^{\pm}_{f(x)}$.}
Every $\hf$-invariant probability measure carried by { the graph of an invariant section $x\mapsto E_x$} is ergodic, and coincides with $\int_M \delta_{(x,E_x)}\, d\mu(x)$.
So (1) and (2) follow from {Lemma \ref{Lemma-lift} and \eqref{varphi-sums}} (see \cite{Ledrappier1982}).
\end{proof}

\begin{corollary}\label{cor-nonergodiclift}
Suppose $f\in\Diff^1(M)$ and $\mu$ is an $f$-invariant probability measure (not necessarily ergodic) s.t. $\lambda^+(f,x)>\lambda^-(f,x)$ for $\mu$-a.e. $x\in M$.
Then any $\hf$-invariant lift $\hmu$ of $\mu$ is carried by $\mathrm{graph}(E^+)\cup \mathrm{graph}(E^-)$,
and there are unique $\hf$-invariant lifts $\hmu^+,\hmu^-$ s.t. $\hmu^+(\mathrm{graph}(E^+))=1$,
$\hmu^-(\mathrm{graph}(E^-))=1$.
 \end{corollary}
\begin{proof}
By a general Borel construction, {the graphs of $E^+, E^-$ are measurable}. There are  unique lifts $\hmu^+,\hmu^-$ of $\mu$ to  $\mathrm{graph}(E^+), \mathrm{graph}(E^-)$, and  it is easy to check using the identity $Df_x(E^{\pm}_x)=E^{\pm}_{f(x)}$ that $\hmu^{\pm}$  are $\hf$-invariant.

Now let $\hmu$ be an arbitrary lift of $\mu$ and
consider its ergodic decomposition $\hmu=\int\hmu_{\xi}\,dm$. For almost every $\xi$, the ergodic measure $\mu_{\xi}=\hmu_{\xi}\circ\hpi^{-1}$ has two different exponents, hence, by Lemma~\ref{Lemma-Lyap}, its lift $\hmu_{\xi}$ is some combination $a(\xi)\hmu_{\xi}^+ + (1-a(\xi))\hmu_{\xi}^-$ of its lifts $\hmu^\pm_\xi$ where $0\leq a(\xi)\leq 1$. Observe that $a(\xi)=\hmu_\xi(\operatorname{graph}(E^+))$, hence the function $\xi\mapsto a(\xi)$ is measurable.
It follows that
 $$
    \hmu = \int a(\xi)\hmu_{\xi}^+ + (1-a(\xi))\hmu_{\xi}^-\, dm.
 $$
In particular, any lift $\hmu$ is carried by the union of the graphs of $E^+$ and $E^-$.
\end{proof}

For each $n\geq1$, $(f,\mu)\mapsto \int \log \|Df^n_x\|\, d\mu(x)$  is continuous as a function on $\Diff^1(M)\times\{\text{probability measure on $M$}\}$.
Eq.~\eqref{e.subbadditive} now give us the following ``folklore" fact:

\begin{theorem}[Upper semicontinuity]\label{t.Lyap-Semi-Cont}
Let $f_k \in \Diff^1(M)$ be diffeomorphisms with ergodic invariant probability  measures $\nu_k$.
If $f_k\overset{\scriptscriptstyle C^1}\longrightarrow f$ and $\nu_k\overset{w^*}\longrightarrow \mu$, then
the largest Lyapunov exponents $\lambda^+(f_k,\nu_k)$ satisfy
$\limsup_k \lambda^+(f_k,\nu_k)\leq \lambda^+(f,\mu).$
\end{theorem}

The next result {computes} the defect in continuity $\lambda^+(f,\mu)-\limsup \lambda^+(f_k,\nu_k)$ in terms of the dynamics on the {projective}  bundle. It relates the defect in continuity  to the escape of some of the mass of the lifts to $\mathrm{graph}(E^+)$ to the vicinity of $\mathrm{graph}(E^-)$.

 Let $\hmu$ be an $\hf$-invariant probability measure with ergodic decomposition  $\hmu=\int_\Omega\hmu_\xi dm$. The projection $\mu$ of $\hmu$ to $M$ has ergodic decomposition $\mu=\int_\Omega \mu_\xi\, dm$, where $\mu_\xi=\hmu_\xi\circ\pi^{-1}$. {Note that in general, the map $\hmu_\xi\mapsto\mu_\xi$ is  not injective.}

We split the set of ergodic components $\hmu_\xi$ $(\xi\in\Omega)$ by considering whether they are carried by the invariant line bundles $E^+$ or $E^-$ or by a subset of $\hM$ where these bundles are not defined:
  $$\begin{aligned}
   &\Omega^-:=\{\xi\in\Omega: \lambda^-(f,\mu_\xi)<\lambda^+(f,\mu_\xi)\ , \ \hmu_\xi=\hmu_\xi^-\}\\
   &\Omega^+:=\{\xi\in\Omega: \lambda^-(f,\mu_\xi)<\lambda^+(f,\mu_\xi)\ , \ \hmu_\xi=\hmu_\xi^+\}\\
   &\Omega^0:=\{\xi\in\Omega: \lambda^-(f,\mu_\xi)=\lambda^+(f,\mu_\xi)\}.
  \end{aligned}$$

\begin{theorem}[Defect in continuity]\label{thm-defect-cont}
{Let $M$ be a compact smooth boundaryless surface.}
For each $k\geq1$,  fix $f_k\in \Diff^1(M)$ and  ergodic $f_k$-invariant measures $\nu_k$ with $\lambda^-(f_k,\nu_k)<\lambda^+(f_k,\nu_k)$. Let $\hnu_k^+$ be the ergodic lift to $\widehat M$ carried by the bundle $E^+$.
Suppose
$
f_k\overset{\scriptscriptstyle C^1}\longrightarrow f$, $\nu_k\overset{w^*}\longrightarrow \mu$, and $\hnu_k^+\overset{w^*}\longrightarrow \hmu.
$
Considering the ergodic decompositions $\hmu=\int_\Omega\hmu_\xi dm$, $\mu=\int_\Omega \mu_\xi\,dm$
and defining $\Omega=\Omega^+\cup\Omega^-\cup\Omega^0$ as above, we have
\begin{equation*}
\begin{aligned}
\lim_{k\to\infty}\lambda^+(f,\nu_k)
&=\lambda^+(f,\mu)-\int_{\Omega^-}[\lambda^+(f,\mu_\xi)-\lambda^-(f,\mu_\xi)]dm.
\end{aligned}\end{equation*}
\medskip
\noindent
In the special case when  $\mu$ is ergodic,  we have the following (see Lemma \ref{Lemma-Lyap}):
\begin{enumerate}[(1)]
\item If $\lambda^+(f,\mu)=\lambda^-(f,\mu)$ or  $\hmu= \hmu^+$ , 
then
    $\lim_{k} \lambda^+(f_k,\nu_k)=\lambda^+(f,\mu).$

\medskip
\item If $\lambda^+(f,\mu)<\lambda^-(f,\mu)$ and $\hmu\neq\hmu^+$, then $\lim_k \lambda^+(f_k,\nu_k)<\lambda^+(f,\mu)$.\\
More precisely,  there is a unique  $0< a\leq  1$ such that
$\hmu=a\hmu^-+(1-a)\hmu^+$ and
    $$
    \lim_{k\to\infty} \lambda^+(f_k,\nu_k)=\lambda^+(f,\mu)-a(\lambda^+(f,\mu)-\lambda^-(f,\mu)).
    $$
    If $\nu_k$ and $\mu$ are hyperbolic (of saddle type), then $a\neq 1$.
\end{enumerate}
\end{theorem}

\begin{proof}
Using Lemma~\ref{Lemma-Lyap}, we see that
\begin{equation}\label{Omega-pm}
\int\vf d\hmu_\xi=\begin{cases}
\int\vf d\hmu_\xi^+ = \lambda^+(f,\mu_\xi) & \text{ if }\xi\in\Omega^+\\
\int\vf d\hmu_\xi^- = \lambda^-(f,\mu_\xi) & \text{ if } \xi\in\Omega^-\\
\lambda^+(f,\mu_\xi) & \text{ if } \xi\in\Omega^0,
\end{cases}
\end{equation}
\begin{equation}\label{eq-limlambda-hat}
\lim_{k\to\infty}\lambda^+(f_k,\nu_k)=\lim_{k\to\infty}  \int\vf d\hnu_k^+=
\int\vf d\hmu=\iint \vf d\hmu_\xi dm.
\end{equation}
Substituting \eqref{Omega-pm} in \eqref{eq-limlambda-hat}, we obtain
\begin{align*}
\lim_{k\to\infty}\lambda^+(f_k,\nu_k)&=\int_{\Omega^-}\lambda^-(f,\mu_\xi)dm+
\int_{\Omega^+\cup\Omega^0}\lambda^+(f,\mu_\xi)dm\\
&=
\int_{\Omega^-}\lambda^-(f,\mu_\xi)dm+\int_{\Omega}\lambda^+(f,\mu_\xi)dm-\int_{\Omega^-}\lambda^+(f,\mu_\xi)dm\\
&=\lambda^+(f,\mu)-\int_{\Omega^-}[\lambda^+(f,\mu_\xi)-\lambda^-(f,\mu_\xi)]dm,
\end{align*}
which proves the theorem for general, possibly non-ergodic, limits $\mu$.

When $\mu$ is ergodic, we apply equation~\eqref{eq-limlambda-hat} and Lemma~\ref{Lemma-Lyap}:
\smallskip

\noindent
(1) If $\lambda^+(f,\mu)=\lambda^-(f,\mu)$ or $\hmu=\hmu^+$, then $\int\vf\,d\hmu=\lambda^+(f,\mu)$ and item~(1) follows.
\smallskip

\noindent (2) Otherwise, $\mu$ has two different exponents and $m(\Omega^+)<1$. Necessarily,  $m(\Omega^0)=0$, $a:=m(\Omega^-)\in(0,1]$ and the lift $\hmu$ can be written as $(1-a)\hmu^+ + a\hmu^-$. So
 $$\begin{aligned}
  \lim_{k\to\infty}\lambda^+(f_k,\nu_k)=\int\vf\, d\hmu &= (1-a)\lambda^+(f,\mu)+a\lambda^-(f,\mu)\\
  &=\lambda^+(f,\mu)-a(\lambda^+(f,\mu)-\lambda^-(f,\mu))\\
  &<\lambda^+(f,\mu).
 \end{aligned}$$
 Finally, note that, if $a=1$ and $\mu$ is hyperbolic, then $\lim_{k}\lambda^+(f_k,\nu_k)=\lambda^-(f,\mu)<0$, and  $\nu_k$ are not hyperbolic {of saddle type} for $k$ large enough.
\end{proof}
 Notice that the defect in continuity originates at $\Omega^-$, the set of ergodic components of $\lim \hnu_k^+$ which are carried by $\mathrm{graph}(E^-)$. This confirms the heuristic that discontinuity in Lyapunov exponents is due to the asymptotic escape of mass from $\mathrm{graph}(E^+)$, which carries  $\hnu_k^+$, to  $\mathrm{graph}(E^-)$, which carries $\hmu_\xi$ for $\xi\in\Omega^-$.

\subsection{A bound for the asymptotic dilation of $\hf$}\label{ss.size-lift}

\begin{lemma}
For any $C^2$ diffeomorphism $f$ of a surface $M$, $\lambda(\hf)\leq \lambda(f)+\lambda(f^{-1})$.
\end{lemma}

\newcommand\hG{\widehat G}

\begin{proof}
Working locally in charts, we identify   the iterates $f^n$ locally with diffeomorphisms $F_i\colon U_i\to \RR^2$ defined on open subsets of $\RR^2$.  We choose the charts so that the change of coordinates distorts the metric by a factor of less than $2$.
The lift $\hf^n\in\Diff^1(\hM)$ is identified with:
 $$
    \hF_i(x,v) = \left(F_i(x),\frac{(DF_i)_x v}{\|(D F_i)_x v\|}\right).
 $$
In what follows, we omit the first factor.

The differential of $\wh F_i$ can be computed in a straightforward way (writing $a\ast b$ for the scalar product of two vectors in $\R^2$):
\begin{equation}\label{eq-Dhf}\begin{aligned}
    (D\hF_i)_{(x,v)}(y,w) = &\frac{(DF_i)_x w}{\|(DF_i)_x v\|}  + \frac{(D^2 F_i)_x(y).v}{\|(DF_i)_x v\|}\\
      &\quad - \left( \frac{(DF_i)_x v}{\|(DF_i)_x v\|} \ast \frac{(DF_i)_x w+(D^2F_i)_x(y).v}{\|(DF_i)_x v\|}\right) \frac{(DF_i)_x(v)}{\|(DF_i)_x v\|}.
  \end{aligned}
  \end{equation}
Thus,
 $$
   \| (D\hF_i)_{(x,v)}\| \leq 2\left( \|(DF_i)_x\| \, \|(DF_i^{-1})_{f(x)}\|
                          + \|(D^2F_i)_x\| \, \|(DF_i^{-1})_{f(x)}\|\right) .
 $$
{We apply this to some iterate of $f$ on $M$, remembering the distortion in the metric:
 $$
   \|D\hf^n\|_\supnorm \leq 32 \| Df^n\|_\supnorm \|Df^{-n}\|_\supnorm + 16 \|Df^{-n}\|_\supnorm \|D^2f^n\|_\supnorm
 $$
Let $0<\eps<1/4$. Fix $n$ an integer so large that $32\leq e^{\eps n/4}$ and
 $$
   \|Df^n\|_\supnorm \leq e^{n(\lambda(f)+\eps/8)} \text{ and }
   \|Df^{-n}\|_\supnorm \leq e^{n(\lambda(f^{-1})+\eps/8)}.
 $$
Therefore,
 $$
   \lambda(\hf) \leq \lambda(f)+\lambda(f^{-1}) + \eps/2 + \log\left(1+e^{-\lambda(f)}\|D^2f^n\|_\supnorm^{1/n}\right)
 $$
By dilating the metric on $M$, we can ensure that $\|D^2f^n\|_\supnorm^{1/n}\leq (\eps/4) e^{\lambda(f)}$ without changing the asymptotic dilations. Thus
  $
   \lambda(\hf) \leq \lambda(f)+\lambda(f^{-1})+\eps
  $
As $\eps$ is arbitrarily small, the claim follows.
}
\end{proof}

{These computations allow the following control of the $C^{r-1}$ size of the lift of a $C^r$ diffeomorphism.

\begin{lemma}\label{l-Ds-hg}
For every real $2<r<\infty$, there is a constant $A=A(r)$ with the following property. For any $g\in\Diff^r(M)$ with lift $\hg\in\Diff^{r-1}(\hM)$,
 $$
    \| \hg \|_{C^{r-1}} \leq A \left(\|g\|_{C^r}\cdot \|Dg^{-1}\|_\supnorm\right)^A.
 $$
\end{lemma}

\begin{proof}
We first consider the $k$th derivative of $f$ in charts  for the maximal integer $k\leq r$. A straightforward induction on the integer $k\geq2$ based on eq.~\eqref{eq-Dhf} shows that the $(k-1)$th  differential of $\hF$ at some point $(x,v)\in\hM$ can be written as a linear combination of terms:
 $$
     \frac{1}{\|D_xF.v\|^p} (D^{\alpha_1} F * D^{\beta_1} F) \dots (D^{\alpha_j} F * D^{\beta_j} F) \cdot D^\gamma F
 $$
where {$p,\alpha_1,\beta_1\dots,\alpha_j,\beta_j,\gamma$} are integers, and $\alpha_1,\beta_1,\dots,\gamma\leq k$. The coefficients of this linear combination depend only on $k$.

If $k=r$, the claim is immediate. If $\alpha:=r-k>0$, recall that the $C^r$ size is the sum of the $C^k$ size and $\alpha$-H\"older size of the $k$-th derivative. A further computation using the above expression gives the required bound for the H\"older constant of order $\alpha$ of $D^k\hg$.
\end{proof}
}

\section{Entropy formulas and reparametrizations}\label{sec-entropy}
We saw in last section that  the defect  in continuity of $(f,\mu)\mapsto \lambda^+(f,\mu)$ can be described in terms of  the canonical lift $\hf:\hM\to\hM$.  In this section we develop tools for studying the entropy map $(f,\mu)\mapsto h(f,\mu)$ in terms of $\hf:\hM\to\hM$.

Specifically, we will show that the entropy of hyperbolic measures on $M$ can be studied in terms of the exponential rate of growth in $C^r$--complexity of  $\hf^n\circ\hsigma$, where $\hsigma:[0,1]\to\hM$ is the curve $\hsigma(t)=(\sigma(t),\R .\sigma'(t))$ and $\sigma:[0,1]\to M$ is a smooth parameterization of a local unstable manifold.

\subsection{Review of entropy and the ergodic decomposition} This section collects several classical facts on the entropy theory of non-ergodic measures. For proofs and details, see \cite[chap. 13]{ETCS}.

Consider  a compact metric space $X$ together with a continuous map $T$ preserving an invariant Borel probability measure $\mu$ and the $\sigma$-algebra $\mathcal X$ of Borel subsets of $X$. The ergodic decomposition of $\mu$ with respect to $T$ is:
 $$
    \mu = \int_X \mu_x \, d\mu(x)
 $$
where $\mu_x:=\lim_{n\to\infty}\frac1n\sum_{0\leq k<n}\delta_{T^kx}$ (the \weakstar\ limit exists almost everywhere by the ergodic theorem).

The map $x\mapsto\mu_x$ is $\mu$-measurable with respect to the $\sigma$-algebra $\mathcal I$ of invariant measurable subsets; for $\mu$-a.e. $x\in X$, $\mu_x$ is a $T$-invariant and ergodic Borel probability measure and for every Borel $\mu$-absolutely integrable $u$, $x\mapsto\mu_x(u)$ belongs to $L^1(\mu)$ and $\mu(u)=\int_X \mu_x(u)\,d\mu(x)$.
{The metric entropy of $\mu$ and $\mu_x$ are related by
$$h(T,\mu)=\int_Xh(T,\mu_x)d\mu(x).$$

Suppose $\xi$ is  a countable measurable partition of $X$ with finite mean entropy $H_\mu(\xi):=-\sum_{A\in\xi}\mu(A)\log\mu(A)$.
Let $\xi_n:=\{\bigcap_{i=0}^{n-1} T^{-i}A_i:A_i\in\xi\}$.
It is a classical fact that
$
h(T,\mu,\xi):=\lim_{n\to\infty}\frac{1}{n}H_\mu(\xi_n)
$
exists.

Similarly, one defines $h(T,\mu_x,\xi)$.
The function $x\mapsto h(T,\mu_x,\xi)$ is defined $\mu$-a.e., is measurable with respect to the $\mu$-completion of $\mathcal I$.
Let $\xi_n(x)$ denote the atom of $\xi_n$ which contains $x$. The Shannon-McMillan-Breiman theorem, in its version for non-ergodic measures \cite[(13.4)]{ETCS}, states that,
 \begin{equation}\label{eq-SMB-thm}
    \lim_{n\to\infty} -\frac1n\log \mu\left(\xi_n(x)\right) = h(T,\mu_x,\xi) \quad \text{both $\mu$-a.e. and in }L^1(\mu).
 \end{equation}
In addition,  we have the following identity: \cite[(13.3)]{ETCS}:
$    h(T,\mu,\xi) = \int_X h(T,\mu_x,\xi)\, d\mu(x).$
In particular,  $h(T,\mu,\xi)\leq \overline{h}(T,\mu,\xi)$,
where
\begin{equation}\label{eq-ergdec-h}
\overline{h}(T,\mu) := \operatorname{ess-sup}_{x\in X} h(T,\mu_x).
 \end{equation}
We call $\overline{h}(T,\mu)$ the
 {\em essential supremum entropy} of $\mu$:

 }

\subsection{Bowen and Katok entropy formulas}\label{sec-Bowen-Katok}
Let $T:X\to X$ be a continuous map on a compact metrix space $X$.
An  {\em $(n,\epsilon)$-Bowen ball} is a set of the form
  $$
      B_T(x,n,\eps) := \{ y\in X: \forall 0\leq k<n,\; d(T^ky,T^kx)<\eps\}.
  $$
The {\em $(n,\eps)$-covering number} of a subset $Z\subset X$, is
 $$
   r_T(n,\eps,Z) := \min\{{|C|}:\bigcup_{x\in C} B_T(x,n,\eps)\supset Z\}.
 $$

 Bowen \cite{Bowen} defined the topological entropy of a (possibly non-invariant) set $Z\subset X$ for $T$ to be
 \begin{equation}\label{eq-Bowen-htop}
   h_\top(T,Z) = \lim_{\eps\to 0} h_\top(T,Z,\eps) \text{ with }
   h_\top(T,Z,\eps) = \limsup_{n\to\infty}\frac1n\log r_T(n,\eps,Z),
  \end{equation}
 and showed that the topological entropy of $T$ is $h_\top(T)=h_\top(T,X)$.

Katok gave a similar formula for the metric entropy of an invariant measure.
{Let $\mu$ be an invariant probability measure. For every $\gamma\in (0,1)$, let
$$
r_T(n,\varepsilon,\mu,\gamma):=\inf\{r_T(n,\eps,Z):Z\subset X\text{ measurable s.t. }\mu(Z)>\gamma\}.
$$
He showed that  if $\mu$ is ergodic, then
$\displaystyle   {h}(T,\mu) = \lim_{\lambda\to 1}\lim_{\eps\to 0}  \limsup_{n\to\infty}\frac1n\log r_T(n,\eps,\mu,\gamma).$
Katok's  proof in \cite{KatokIHES} also works in the non-ergodic case, if we replace the usual Shannon-McMillan-Breiman Theorem by \eqref{eq-SMB-thm}. The result is that for a general (possily non-ergodic)  invariant probability measure $\mu$,
 \begin{equation}\label{eq-Katok-hKS}
   \overline{h}(T,\mu) = \lim_{\lambda\to 1}\lim_{\eps\to 0}  \limsup_{n\to\infty}\frac1n\log r_T(n,\eps,\mu,\gamma).
  \end{equation}
Here $\overline{h}(T,\mu)$ is the essential supremum  entropy from \eqref{eq-ergdec-h}.}

\subsection{The lift to the projective tangent bundle preserves entropy}
This is a standard consequence of the following theorem~\cite{Ledrappier-Walters}.
\begin{theorem}[Ledrappier-Walters]\label{t.Led-Wal}
Let $(X,T)$, $(Y,S)$ be continuous self-maps of compact metric spaces and let $\pi:(Y,S)\to(X,T)$ be a topological factor map s.t.
 \begin{equation}\label{eq-zero-fiber-htop}
 \forall x\in X \qquad h_\top(S,\pi^{-1}(x))=0.
 \end{equation}
Then, for any $S$-invariant Borel probability measure $\nu$ on $Y$,
 $
  h(S,\nu) = h(T,\pi_*\nu).
 $
\end{theorem}
\noindent
Thus by the variational principle, in the setup of the theorem,  $h_\top(T)=h_\top(S)$.

\begin{corollary}\label{Lemma-Entropy-Extension}
Suppose $f\in\Diff^1(M)$ and $\hf$ is the canonical extension of $f$ to the {projective} bundle $\hM$. Then
$  h_\top(\hf)=h_\top(f)$, and for every $\hf$-invariant probability measure $\hnu$ with projection $\nu$,
 $
   h(\hf,\hnu) = h(f,\nu)$ {and $\oh(\hf,\hnu) = \oh(f,\nu)$.}
\end{corollary}

\begin{proof}
We check condition~\eqref{eq-zero-fiber-htop} for $T=f$, $S=\hf$, and apply the previous theorem.
$\hM$ is a topological bundle over $M$, and its fibers $P_x M$ are homeomorphic to circles.  The map $\hf: P_x M\to P_{f(x)} M$ is a homeomorphism. For every $\eps>0$, one can find partitions $\xi_x$ of $P_x M$ into a bounded number of arcs with diameter at most $\eps$. It is easy to see that $\xi_x\vee\hf^{-1}\xi_{\hf\hx}\vee\dots\vee \hf^{-n+1}\xi_{\hf^{n-1}\hx}$ has cardinality at most $\sum_{k=0}^{n-1} \Card(\xi_{\hf^k\hx})=O(n)$. It follows that $h_{top}(\hf,P_x M)=0$ for all $x$.
\end{proof}

\subsection{Bowen and Katok entropy formulas on the bundle $\hM$.}
We need a variant of the Bowen and Katok entropy formulas which uses a different type of Bowen balls, which are  better adapted to the bundle structure of $\hM$.

Recall the natural projection $\hpi:\hM\to M$,  $\hpi(x,E)=x$.
The {\em fibered ball} with {\em center} $\hx\in\hM$ and {\em scales}  $\eps,\heps>0$ is the set
$$B(\hx,\eps,\heps):=\{\hy\in\hM:\; d(\hx,\hy)<\heps \text{ and }
    d(\hpi(\hx),\hpi(\hy))<\eps\}.$$
A set $S\subset\hM$ has {\em fibered size} $\leq (\eps,\heps)$ if $S\subset B(\hx,\eps,\heps)$ for some $\hx$.

Suppose $f\in\Diff^1(M)$ and $\hf$ is the canonical extension of $f$ to $\hM$.
The {\em fibered $(n,\eps,\heps)$--Bowen ball}  with {\em center} $\hx\in\hM$, {\em size} $(\eps,\heps)$ and {\em length} $n$ is the set
\begin{equation}
\begin{split}
    B_{\hf}(\hx,n,\eps,\heps):=\{\hy\in\hM:\;\forall 0\leq k<n,\; d(&\hf^k(\hx),\hf^k(\hy))<\heps\\ &\text{ and }
    d(f^k(\hpi(\hx)),f^k(\hpi(\hy)))<\eps\}.
\end{split}
\end{equation}

The {\em $(n,\eps,\heps)$-covering number} of a subset $Z\subset\hM$ is the minimal number of fibered  $(n,\eps,\heps)$-Bowen balls whose union contains $Z$. It is denoted by $$r_{\hf}(n,\eps,\heps,Z).$$
Clearly if $\eps_1\leq \eps_2$ and $\heps_1\leq \heps_2$, then  $r_{\hf}(n,\eps_1,\heps_1,Z)\geq r_{\hf}(n,\eps_2,\heps_2,Z)$.

Similarly, given an ergodic measure $\hmu$ of $\hf$ and a number $0<\gamma<1$, the {\em $(n,\eps,\heps,\gamma)$-covering number of $\hmu$} is the minimal number of fibered $(n,\eps,\heps)$-Bowen balls whose union has $\hmu$-measure at least $\gamma$. It is denoted by
$$r_{\hf}(n,\eps,\heps,\hmu,\gamma).$$
If $\eps_1\leq \eps_2, \heps_1\leq \heps_2$, and $\gamma_1\geq\gamma_2$,  then  $r_{\hf}(n,\eps_1,\heps_1,\hmu,\gamma_1)\geq r_{\hf}(n,\eps_2,\heps_2,\hmu,\gamma_2)$.

\begin{proposition}\label{fibered Bowen-Katok}
Fix $f\in\Diff^1(M)$ with canonical lift $\hf$. We have:
\begin{enumerate}[(1)]
\item {\em Bowen's formula:} For every $\heps>0$,
$$
h_\top(f)=\lim_{\eps\to0}\limsup_{n\to\infty} \frac1n\log r_{\hf}(n,\eps,\heps,\hM)=
\lim_{\eps\to0}\liminf_{n\to\infty} \frac1n\log r_{\hf}(n,\eps,\heps,\hM)
$$
\item\label{i.Katok} {\em Katok's formula:} Suppose $\mu$ is an $f$-ergodic invariant measure, and let $\hmu$ be an $\hf$-ergodic lift of $\mu$. Then for every $\heps>0$ and $0<\gamma<1$,
 $$
    h(f,\mu)=\lim_{\eps\to0}\limsup_{n\to\infty} \frac1n\log r_{\hf}(n,\eps,\heps,\hmu,\gamma),
 $$
If $\mu$ is $f$-invariant, but possibly not ergodic, then for every lift $\hmu$,
$$
    \overline h(f,\mu)=\lim_{\gamma\to 1} \lim_{\eps\to0}\limsup_{n\to\infty} \frac1n\log r_{\hf}(n,\eps,\heps,\hmu,\gamma),
$$
\end{enumerate}
\end{proposition}
\begin{proof}
We
 follow the proof of \cite[Thm 17]{Bowen}. We need the following claim.

\begin{claim}
For every $\heps,\alpha>0$, there are $C,\varepsilon_*>0$ such that
 $$
   \forall \eps\in(0,\varepsilon_*),\; \forall n\geq1,\;\forall x\in M,\quad r_{\hf}(n,\eps,\heps,\hpi^{-1}B_f(x,n,\eps)) \leq C e^{\alpha n}.
 $$
\end{claim}

\medskip
\noindent
{\em Proof of the Claim.}
Let $\heps,\alpha>0$.
Note first that
$$
 r_{\hf}(n,\eps,\heps,\hpi^{-1}B_f(x,n,\eps)){\leq} r_{\hf}(n,\heps,\hpi^{-1}B_f(x,n,\eps)),
$$
hence it is enough to bound the latter. Since $h_\top(\hf,\hpi^{-1}(x))=0$,
for each $x\in M$ there is a smallest integer $n_x\geq1$ such that $r_{\hf}(n_x,\heps/2,\hpi^{-1}(x))\leq e^{\alpha n_x}$.

Recall that a set-valued function $F$ from a topological space $X$ to the set of subsets of a topological set $Y$ is called {\em upper semi-continuous}, if for every $E\subset Y$ closed, $\{x\in X: F(x)\cap E\neq \emptyset\}$ is closed.
The continuity of $\hpi$ and the compactness of $\hM$ implies that $x\mapsto \hpi^{-1}(x)$ is upper semi-continuous.

It follows that if $\hpi^{-1}(x)$ is contained in some open set $U$ (say the union of a minimal cover by   fibered Bowen balls), then  $\hpi^{-1}(y)$ is contained in $U$ for all $y$ sufficiently close to $x$.
Hence there is an $r_x>0$ such that
 \begin{equation}\label{eq-fiber-cover}
    r_{\hf}(n_x,\heps/2,\hpi^{-1}B(x,r_x))\leq r_{\hf}(n_x,\heps/2,\hpi^{-1}(x))\leq e^{\alpha n_x}.
  \end{equation}
Using a compactness argument, we see that $n_*:=\sup\{n_x:x\in M\}$ is finite and that one can arrange for ${\eps_*}:=\inf\{r_x:x\in M\}$ to be positive.

{Let $\eps\in(0,\eps_*)$, $n\geq1$ and $x\in M$.} Define $t_0(x):=0$ and $t_{i+1}(x):=t_i(x)+n_{f^{t^i(x)}x}$. Choose $i\geq0$  maximal such that $t_i(x)\leq n$. Note that $0\leq n-t_i(x)<n_*$. Thus, setting $x_j:=f^{t_j(x)}x$ and $n_j:=n_{x_j}$,
 $$
    B_f(x,n,\eps) = \bigcap_{j=0}^{i-1} f^{-t_j(x)}B_f(x_j,n_j,{\eps}) \cap  f^{-t_i(x)}B(x_i,n-t_i(x),{\eps}).
 $$

Let $\{z_1,\dots,z_C\}$ be a $(n_*-1,\heps/2)$-cover of $\hM$ with 
cardinality $C$.
For each $0\leq j<i$, let $\{y_{j1},\dots,y_{jm}\}$ with $m\leq e^{\alpha n_{j}}$ be the $(n_{j},\heps/2)$-cover of $\hpi^{-1}B(x_j,n_{j},\eps)$ implied by eq.~\eqref{eq-fiber-cover}. Then $ \hpi^{-1}B_f(x,n,\eps)$ is contained in the following union:
 $$
  \bigcup_{\tiny\begin{array}{c} k,k_0,\dots,k_{i-1}\\ k_j\leq \exp \alpha n_j\\ k\leq C\end{array}}  \bigcap_{j=0}^{i-1} \hf^{-t_j(x)}B_{\hf}(y_{jk_j},n_{j},\heps/2) \cap  \hf^{-t_i(x)}B_{\hf}(z_k,n-t^i(x),\heps/2).
 $$

Any two points contained in the same term of this union are not $(n,\heps)$-separated.
So  the cardinality of {any} $(n,\heps)$-separated subset of $\hpi^{-1}B_f(x,n,\eps)$ is bounded by
 $$
   C \times \prod_{j=0}^{i-1} e^{\alpha n_j}  \leq C e^{\alpha n}.
 $$
Since a maximal $(n,\heps)$--separated subset is the set of centers of a cover by $(n,\heps)$--Bowen balls, $r_{\hf}(n,\heps,\hpi^{-1} B_f(x,n,\eps))\leq C e^{\alpha n}$, which  proves the claim.

\medskip
The claim implies that for any $\heps,\alpha>0$, there are $C,\eps_*>0$ such that, for any $0<\eps\leq\eps_*$ and any ${Z}\subset M$,
 $$
    r_f(n,\eps,{Z}) \leq r_{\hf}(n,\eps,\heps,\hpi^{-1}{Z}) \leq C e^{\alpha n} r_f(n,\eps,{Z}).
 $$
The proposition now follows from the classical identities \eqref{eq-Bowen-htop}~and~\eqref{eq-Katok-hKS}.
\end{proof}

\subsection{Curves and $C^r$ reparametrizations}\label{ss.curves}
The entropy of a diffeomorphism can  be related to the exponential rate of growth in $C^r$ complexity of the iterates of  a local unstable manifold. To do this we need to control the curvature, and for this purpose it is useful to lift the curve  to {projective} bundle $\hM$ and study its iterations there. Here we develop the tools needed for doing this.

\begin{definition}
A $C^r$ curve $\sigma:[0,1]\to M$ is \emph{regular} if its derivative $\sigma'(t)$ never vanishes.
In this case, it has a \emph{canonical lift} $\hsigma:[0,1]\to\hM$ defined by $$\hsigma(t):=(\sigma(t),\RR.\sigma'(t)).$$
\end{definition}
\noindent
Here and throughout,  $\R.\sigma'(t)\equiv \mathrm{span}\{\sigma'(t)\}$ is a linear subspace of $T_x M$. By regularity, $\R.\sigma'(t)$ is  one-dimensional.

\medskip
\begin{definition}
Fix $r\geq 2$ and let $\eps,\heps$ be two positive numbers.
A regular $C^r$ curve $\sigma:[0,1]\to M$ has \emph{\Crr\ size} (or just size) {less than} $(\eps,\heps)$ if
 $$
    \|\sigma\|_{C^r} {<} \eps \text{ and }
    \|\hsigma \,\|_{C^{r-1}} {<} \heps.
 $$
The curve has \emph{diameter} less than $(\eps,\heps)$ if
 $$
    \diam_M(\sigma([0,1]))<\eps \text{ and } \diam_{\hM}(\hsigma([0,1]))<\heps.
 $$
To say that a curve has finite \Crr\ size implies that it is regular and $C^r$.
\end{definition}

\begin{remark}
If a curve $\sigma$ is parametrized by length (i.e., $\|\sigma'(t)\|=1$ for all $t$), then it has size at most $(\|\sigma\|_{C^r},C\|\sigma'\|_{C^{r-1}})$ for some constant $C>0$ which depends on the choice of $C^{r-1}$ atlas of  $\hM$ used to define $C^{r-1}$ size (see section \ref{ss.size}).
\end{remark}

Suppose $\sigma:[0,1]\to M$ is a curve. If we cut $[0,1]$ into small intervals $[a_i,b_i]$ and reparametrize $\sigma|_{[a_i,b_i]}$ by $\sigma\circ\psi_i$ where  $\psi_i[0,1]\to [a_i,b_i]$ is an affine bijection, then $\|\sigma\circ\psi_i\|_{C^r}\leq \kappa \|\sigma\|_{C^r}, \|\hsigma\circ\psi_i\|_{C^r}\leq \kappa\|\hsigma\|_{C^r}$, where $\kappa=|a_i-b_i|<1$. Cutting sufficiently finely, we can obtain covering by pieces with affine reparametrizations with $C^r$ size as small as we wish.

Yomdin measured the  $C^r$ complexity of a curve (more generally a set) by counting how many reparametrized pieces with  $C^r$ size less than $1$ are needed to cover it. We {adapt this to the projective dynamics}:

\begin{definition}\label{d.reparametrization-5.8}
{Let $\sigma:[0,1]\to M$ be a $C^r$ curve. A \emph{reparametrization of $\sigma$} is a non-constant affine map $\psi\colon [0,1]\to [0,1]$.}\\
A \emph{family of reparametrizations of $\sigma$ over {a subset} $T\subset[0,1]$} is a collection $\mathcal R$ of reparametrizations such that
$T\subset \bigcup_{\psi\in\mathcal R} \psi([0,1])$.
\end{definition}

Let $f\in\Diff^r(M)$ and let $\sigma$ be a regular $C^r$ curve. We will be interested in families of reparametrizations which remain bounded in $C^r$ size after application of $f^{n}$ for  certain  $n$.
Specifically, fix numbers $\eps,\heps>0$, an integer $N\geq1$, and $T\subset [0,1]$.

\begin{definition}\label{d.reparametrization}
{A reparametrization  $\psi$ of $\sigma$ is {\em $(C^r,f,N,\eps,\heps)$-admissible  up to time $n$},  if
there exists an increasing sequence $(n_0,n_1,\dots,n_\ell)$ such that
\begin{itemize}
\item $n_0=0$, $n_\ell=n$, and $n_{j}-n_{j-1}\leq N$ for each $1\leq j\leq \ell$,
\item for each $0\leq j\leq \ell$ the curve $f^{n_j}\circ\sigma\circ\psi$ has $C^r$ size less than $(\eps,\heps)$.
 \end{itemize}
We call the integers  $n_j$ the {\em admissible times}.

 A family $\cR$ of reparametrizations of $\sigma$ over $T$
 is \emph{$(C^r,f,N,\eps,\heps)$-ad\-mis\-sible up to time} $n$, if each $\psi\in \cR$ is $(C^r,f,N,\eps,\heps)$-admissible  up to time $n$.  }
\end{definition}

\begin{lemma}[Concatenation of reparametrizations]\label{l.concatenation}
Let $f\in\Diff^r(M)$, $r\geq 2$, and consider 
a regular $C^r$ curve $\sigma$ on $M$, and $T\subset [0,1]$. Suppose that
\begin{enumerate}[$\bullet$]
\item $\cR$ is a family of repa\-rame\-tri\-zations of $\sigma$ over $T$;
\item $\cR$ is
$(C^r,f,N,\eps,\heps)$-admissible up to time $n$;
\item for each $\psi\in\cR$, there is a family $\cR_\psi$ of reparametrizations of $f^{n}\circ\sigma\circ\psi$ 
over $\psi^{-1}(T)$,  which is $(C^r,f,N,\eps,\heps)$-admissible up to time $n'$.
\end{enumerate}
Then
  $
    \cR':=\bigcup_{\psi\in\mathcal R} \{\psi\circ\phi : \phi\in\mathcal R_\psi\}
  $
is a family of reparameterizations of $\sigma$ over $T$, and $\cR'$ is $(C^r,f,N,\eps,\heps)$-admissible up to time $n+n'$.
\end{lemma}
\begin{proof}
Since $\psi,\phi:[0,1]\to [0,1]$ are  non-constant and affine, $\psi\circ\phi:[0,1]\to [0,1]$ is non-constant and affine. Next,
$$
\bigcup_{\psi\in\cR}
\bigcup_{\phi\in\cR_\psi}(\psi\circ\phi)[0,1]=
\bigcup_{\psi\in\cR}\psi\bigl(\bigcup_{\phi\in\cR_\psi}\phi[0,1]\bigr)
\supseteq \bigcup_{\psi\in\cR}\psi(\psi^{-1}(T))=\bigcup_{\psi\in\cR}T\cap\psi[0,1]=T,
$$
So $\cR'$ is a family of reparametrizations of $(\sigma,\hsigma)$ over $T$.

To see that $\cR'$ is admissible, fix $\psi\in\cR, \phi\in\cR_\psi$, and  choose admissible times
$0=n_0<n_1<\cdots<n_\ell=n$ and $0=n_0'<n_1'<\cdots<n_m'=n'$ for $\psi$ and $\phi$. Let
$$
n_{\ell+k}:=n+n_k'\ \ (k=0,\ldots,m).
$$
We claim that $0=n_0<n_1<\cdots<n_{\ell+m}=n+n'$ are admissible times for $\psi\circ\phi$.
That the gaps are no larger than $N$ is clear. If $j=\ell+1,\ldots,\ell+m$, then $f^{n_j}\circ\psi\circ\phi$ is a regular curve with size 
$<(\eps,\heps),$ because $\phi\in\cR_\psi$ and $\cR_\psi$ is admissible. If $j\leq \ell$, then using the fact that  $\phi'=c$ with $c$ a constant s.t. $|c|\leq 1$, we find that
$
\|f^{n_j}\circ\sigma\circ\psi\circ\phi\|_{C^r}\leq
\|f^{n_j}\circ\sigma\circ\psi\|_{C^r}\cdot|c|\leq \|f^{n_j}\circ\sigma\circ\psi\|_{C^r}$. This is less than $\eps$,
because of the admissibility of $\cR$.
Similarly $\|\hf^{n_j}\circ\hsigma\circ\psi\circ\phi\|_{C^{r-1}}\leq
\|\hf^{n_j}\circ\hsigma\circ\psi\|_{C^{r-1}}<\heps$.
\end{proof}

\begin{lemma}[Length of reparametrizations]\label{l.length}
Let $f\in\Diff^r(M)$, $r\geq 2$.
For any $\eta>0$ and $N\geq 1$, there exist $\varepsilon_{*}>0$ and
a $C^2$-neighborhood $\cU_{*}$ of $f$ in $\Diff^r(M)$ with the following property.

Consider $g\in\cU_{*}$, $\eps,\heps\in(0,\eps_{*})$, a regular $C^r$ curve $\sigma$, and a reparametrization $\varphi$ of $\sigma$
which is $(C^r,g,N,\eps,\heps)$-admissible up to time $n$. Then for any $(x,E)\in \hsigma[0,1]$,
$$\operatorname{Length}(g^n\circ \sigma\circ \varphi)< e^{\eta n/ 10}\|Dg^{n}_x|_E\| \operatorname{Length}(\sigma\circ \varphi).$$
\end{lemma}
\begin{proof}
$\hM$ is compact, so one can find $\eps_{*}>0$ and a small  $C^2$-neighborhood $\cU_{*}$ of $f$ in $\Diff^r(M)$
such that for all $g\in\cU_{*}$ and $\hx,\hy\in\hM$ satisfying $d(\hx,\hy)<\eps_{*}$,
\begin{equation*}\label{ss.continuous0}
 \log \|Dg|_{\hg^k(\hx)}\|\leq \log \|Dg|_{\hg^k(\hy}\|)+\tfrac{\eta}{10}\qquad  (0\leq k\leq N).
\end{equation*}
Let us consider $g\in\cU_{*}$, a curve $\sigma$ and a reparametrization $\varphi$ as in the statement, with admissibility times
$n_0,\dots,n_\ell$. One gets for any $0\leq i<\ell$,
$$\forall (x,E),(y,F)\in \hg^{n_i}\circ\hsigma[0,1],\qquad \|Dg^{n_{i+1}-n_i}_y|_F\|< e^{(n_{i+1}-n_i)\eta / 10}\|Dg^{n_{i+1}-n_i}_x|_E\|,$$
which immediately implies the conclusion of the lemma.
\end{proof}

Admissible families of reparametrizations yield covers by fibered Bowen balls with size of the same order of magnitude.

\begin{lemma}[Bowen covers from admissible reparametrizations]\label{lemma covering}
Let $f\in\Diff^r(M)$, $r\geq2$, $T\subset [0,1]$, and $\sigma$ be a regular $C^r$ curve. Fix $\eps_*,\heps_*>0$ and $N\geq 1$.
Let $\cR$ be a family of repa\-rame\-tri\-zations of $\sigma$ over $T$ which is
$(C^r,f,N,\eps_*,\heps_*)$-admissible up to time $n$.
Then, for every $\eps,\heps>0$,
$$
     r_{\widehat f}(n,\eps,\heps,\hsigma(T)) \leq
    \frac{2\heps_\ast\|D\hf\|_\supnorm^N }{\min(\eps,\heps)} |\mathcal R|.
  $$
\end{lemma}
\begin{proof}
Let $L:=\|D\hf\|_\supnorm$.  $L\geq 1$ since $\hf$ is surjective.

Let $
{\displaystyle \rho:=\frac{\min(\eps,\heps)}{\heps_\ast L^N}}
$,  and construct a  $\rho$--dense set  $C_\rho\subset[0,1]$ s.t. $|C_\rho|\leq \frac{2}{\rho} $. {Set
$
C_\rho(\cR):=\{(\hsigma\circ\psi)(t'):t'\in C_\rho,\psi\in\cR\}$
}.

Fix $\psi\in\cR$ with admissible times $0=n_0<n_1<\cdots<n_\ell=n$. Then  $n_{j+1}-n_j\leq N$ and the regular curve $f^{n_j}\circ\sigma\circ\psi$ has $C^r$ size at most $(\eps_\ast,\heps_\ast)$, whence
$$
\|D\hf^{k}\circ\sigma\circ\psi\|_\supnorm\leq \heps_\ast L^{k-n_j}\ \text{ for }k\in [n_j, n_{j+1}).
$$
In particular,
$\|D\hf^{k}\circ\sigma\circ\psi\|_\supnorm \leq \heps_\ast L^N\text{ for all }0\leq k\leq n.
$

For every $\hx\in\hsigma(T\cap \psi[0,1])$ with $\hx=\hsigma\circ \psi(t)$, there is $\hy=\hsigma\circ \psi(t')$ with $t'\in C_\rho$
such that $|t-t'|<\rho$, whence
$$
d(\hf^k(\hx),\hf^k(\hy))<\rho\heps_\ast L^N \text{ and }d(f^k(\hpi(\hx)),f^k(\hpi(\hy)))<\rho\heps_\ast L^N\text{ for all }k=0,\ldots,n.
$$
(the second inequality follows from the first). Notice that $\hy\in C_\rho(\cR)$.

All this shows that
$
\displaystyle\hsigma(T)\subset \bigcup_{\hy\in C_\rho(\cR)}B_{\hf}(\hy,{ n,}\rho\heps_\ast L^N,\rho\heps_\ast L^N)\subset \bigcup_{\hy\in C_\rho(\cR)}B_{\hf}(\hy,n,\eps,\heps).
$
So
$
r_{\hf}(n,\eps,\heps,\hsigma(T))\leq |C_\rho(\cR)|\leq |\cR|\cdot |C_\rho|\leq \frac{ 2}{\rho}|\cR|\leq \frac{ 2 \heps_\ast}{\min(\eps,\heps)}\|D\hf\|_\supnorm^N |\cR|
$.
\end{proof}

\subsection{Yomdin estimates}\label{ss.yomdin}
In this section we discuss a converse to Lemma \ref{lemma covering}:  Covers by Bowen balls generate admissible repara\-metrizations with cardinality of the same order of magnitude. This result is much more delicate than Lemma \ref{lemma covering}, and  requires Yomdin's Theorem  \cite{Yomdin-Volume-Growth-Entropy}.
Here is the tool we need from Yomdin's work, in a form adapted to our setup. Let
\begin{equation}\label{Q-r-n}
    Q_r(g):=\max(\|g\|_{C^r},\|\hg\|_{C^{r-1}})\ ,\      Q_{r,N}(g):=\max_{n=1,2,\dots,N} Q_r(g^n).
\end{equation}

\begin{theorem}[Yomdin]\label{t-yomdin}
Given {real numbers} $2\leq r<\infty$ and $Q>0$, there {are}
$\Upsilon=\Upsilon(r)>0$ and $\eps_Y=\eps_Y(r,Q)>0$ with the following properties. For every
\begin{enumerate}[\quad$\bullet$]
\item $C^r$  {diffeomorphism} $g:M\to M$ such that $Q_r(g)\le Q$
\item regular $C^r$ curve $\sigma$ with \Crr\ size at most $(\eps,\heps)$ with $0<\eps,\heps\le\eps_Y$,
\item $\hx\in \hsigma[0,1]$, and $T:=\{t\in [0,1]: \hg(\hsigma(t))\in B(\hg(\hx),\eps,\heps)\}.
        $ \end{enumerate}
there exists a family $\cR$ of reparametrizations of $\sigma$ over $T$ such that
\begin{enumerate}[\quad(1)]
\item for every $\psi\in\cR$, $g\circ\sigma\circ\psi$ is a regular $C^r$ curve with \Crr\ size at most $(\eps,\heps)$,
\item $|\cR|\leq \Upsilon\|D\hg\|_\supnorm^{1/(r-1)}$.
\end{enumerate}
\end{theorem}

\begin{proof}
The reader who will compare this theorem to  Yomdin's original statement in \cite{Yomdin-Volume-Growth-Entropy} will find that the two results are nearly the same, except for the following differences: (1)  Yomdin considered the more general case of $\sigma:[0,1]^\ell\to M$ whereas we restrict to $\ell=1$; (2) Yomdin did not specify that all reparametrizations are affine as we do; (3) Our result allows $r$ to be real, not just an integer;  and (4) We use $(n,\eps,\heps)$ balls in $\hM$, whereas Yomdin used $(n,\eps)$ balls in $M$.

In the special case of curves $\ell=1$, Yomdin's proof works verbatim with affine reparametrizations, see, e.g., \cite[p. 297--298]{Yomdin-Volume-Growth-Entropy}. The extension to  non-integer smoothness is also simple and well-known, see, e.g., \cite[p.133]{BuzziSIM}.

To deal with (4), we  apply the Yomdin's original theorem twice, first for $(g,\sigma)$ {on $M$} and then for the lift $(\hg,\hsigma)$ {on $\wh M$}.
This yields numbers $\Upsilon_i=\Upsilon_i(r)$ $(i=1,2)$ and ${\eps_Y}=\eps_Y(r,Q)>0$ as follows. Suppose  $0<\eps,\heps\le\eps_Y$, $g\in\Diff^r(M)$,  $Q_r(g)\le Q$, and  $\sigma:[0,1]\to M$ has $C^r$ size at most $(\eps,\heps)$. Then

\smallbreak
\noindent
(1) 
There is a family $\cR_1$ of
reparametrizations of $g\circ \sigma$
over $T$ s.t. $|\cR_1|\leq \Upsilon_1 \|Dg\|_\supnorm^{1/r}$, and so that each $\psi\in\cR_1$ is affine, contracting, and  $\|g\circ\sigma\circ\psi\|_{C^r}\leq\eps$.

\smallbreak
\noindent
(2)
There is a family $\cR_2$ of
reparametrizations
of $\hg\circ\hsigma$
 over $T$ s.t. $|\cR_2|\leq \Upsilon_2 \|D\hg\|_\supnorm^{\frac{1}{r-1}}$ and so that each $\psi\in\cR_2$ is affine, contracting, and
  $\|\hg\circ\hsigma\circ\psi\|_{C^r}\leq\heps$.

\smallbreak
{Each {family of reparametrizations} $\cR_i$ generates a cover of $T$ by the intervals $\psi([0,1])$, $\psi\in\cR_i$. Without loss of generality, the interiors of these intervals are pairwise disjoint (otherwise discard some of them and shrink the rest  by composing the reparametrizations by affine contractions). We define
 $$
    \cR := \{ \phi_{\psi_1,\psi_2} \mid (\psi_1,\psi_2)\in\cR_1\times\cR_2 \text{ s.t. } \psi_1((0,1))\cap\psi_2((0,1))\ne\emptyset \}
 $$
where $\phi_{\psi_1,\psi_2}:[0,1]\to\psi_1([0,1])\cap\psi_2([0,1])$ is an affine diffeomorphism.
The image of this new family of reparametrizations contains the intersection of the images of $\cR_1$ and $\cR_2$, hence $\cR$ is a family of reparametrizations of $\sigma$ over $T$.

As the reparametrizations are affine and contracting, we see that
 $$
    \|g\circ\sigma\circ\phi_{\psi_1,\psi_2}\|_{C^r} \leq \|g\circ\sigma\circ\psi_1\|_{C^r}
    {<}\eps
 $$
Likewise for $\hg\circ\hsigma\circ\phi_{\psi_1,\psi_2}$. Since $\hg\circ\hsigma\circ\phi_{\psi_1,\psi_2}$  coincides with the lift of $g\circ\sigma\circ\phi_{\psi_1,\psi_2}$,  item (1) of the theorem holds.

Next, using the order structure on the interval, it is not difficult to show that
 $
   |\cR| \leq |\cR_1|+|\cR_2|-1 \leq \Upsilon_1 \|Dg\|_\supnorm^{1/r} +  \Upsilon_2\|D\hg\|_\supnorm^{\frac{1}{r-1}}.
$
Obviously $ \|D g\|_\supnorm^{1/r}\leq  \|D\hg\|_\supnorm^{1/(r-1)}$, so item (2) holds with  $\Upsilon:=\Upsilon_1+\Upsilon_2$.
}
\end{proof}

\begin{corollary}[Existence of admissible reparametrizations]\label{coro.reparametrization}
For $2\leq r<\infty$, $Q>0$, let $\Upsilon(r)$, ${\eps_Y}(r,Q)$ be the constants from Theorem~\ref{t-yomdin}. Suppose
\begin{enumerate}[\quad$\bullet$]
\item $g\in\Diff^r(M)$ and  $Q_{r,N}(g)< Q$,
\item $\sigma:[0,1]\to M$ is a regular curve with \Crr\ size $\leq (\eps,\heps)$, where  $0<\eps,\heps<\eps_Y$,
\item $N,n\geq 1$, and  $T\subset [0,1]$.
\end{enumerate}
Then there exists a family $\cR$ of reparametrizations of $\sigma$ over $T$, which is
$(C^r,g,N,\eps,\heps)$-admissible up to the time $n$, and with cardinality
$$|\cR|\leq C(r,\hg) r_{\hg}(n,\eps,\heps,\hg\circ\hsigma(T)),$$ where
$
 C(r,\hg)=\Upsilon(r)^{\lceil \frac{n}{N}\rceil}
   \|D\hg^N\|_\supnorm^{\frac{\lfloor \frac{n}{N}\rfloor}{r-1}}
  \|D\hg^{n-N\lfloor \frac{n}{N}\rfloor}\|_\supnorm^{\frac{1}{r-1}}.
 $
 \end{corollary}

\begin{proof}
Fix $n\geq 1$ and  divide with remainder
$
n=qN+p
$,
 $q\geq 0$, $p=0,\ldots,N-1$.

Let  ${\ell}:=r_{\hg}(n,\eps,\heps,\hg\circ\hsigma(T))$, then there exists a cover of ${\hg\circ}\hsigma(T)$ by ${\ell}$ fibered $(n,\eps,\heps)$--Bowen balls
$B_i:=B(\hg(\hx_i),{n,}\eps,\heps)$.
For each $B_i$, we will construct an admissible family of reparametrizations of $\sigma$ over
$
T_i:=(\hg\circ\hsigma)^{-1}(B_i),
$ and then take the union over $i$.

\medskip
\noindent
{{\em Step 0.\/}  If $p=0$ move to step 1. Otherwise proceed as follows.}

Fix $1\leq i\leq {\ell}$. Yomdin's theorem for  $g^p$, ${\sigma}$ and  $\hx_i$ gives  a family of reparametrizations $\cR_{0}$ of $\sigma$ over
$$
T_i^0:=\{t\in [0,1]: \hg^p(\hsigma(t))\in B(\hg^p(\hx_i),\eps,\heps)\},
$$
which is $(C^r,{g,}N,\eps,\heps)$-admissible up to time $p$ (the admissible times are $0,p$), and such that
$
|\cR_0|\leq \Upsilon\|D\hg^p\|^{1/(r-1)}_{\supnorm}
.
$
Notice that $T_i^0\supset T_i$.
If $q=0$, we have a reparametrization up to time $n$, and we stop.

\medskip
\noindent
{{\em Step 1.}\/} Fix $\psi\in\cR_0$ and apply Yomdin's theorem to $g^N$,  $g^p\circ\sigma\circ\psi$, and  $\hg^p(\hx_i)$. The result is a family $\cR_\psi$ of reparametrizations  of $g^p\circ\sigma\circ\psi$ over
$$
T_i^1(\psi):=\{t\in [0,1]: \hg^{N+p}[(\hsigma\circ \psi)(t)]\in B(\hg^{N+p}(\hx_i),\eps,\heps)\}
$$
which is $(C^r,{g,}N,\eps,\heps)$-admissible up to time $N$ (the admissible times are $0,N$), and with cardinality
$
|\cR_{\psi}|\leq \Upsilon\|D\hg^N\|^{1/(r-1)}_{\supnorm}
$.

 Notice that $T_i^1(\psi)\supset\psi^{-1}(T_i)$, therefore the concatenation
$$
\cR_1:=\{\psi\circ\phi:\psi\in\cR_0, \phi\in\cR_\psi\}
$$
is an admissible family of reparametrizations of $\sigma$ over $T_i$ up to time $N+p$, with admissible times $0,p,N+p$ and cardinality
$$
|\cR_1|\leq \Upsilon^2\|D\hg^N\|^{1/(r-1)}_{\supnorm}\|D\hg^p\|^{1/(r-1)}_{\supnorm}.
$$
If $q=1$, we have a reparametrization up to time $n$, and we stop.

\medskip
 Otherwise we continue as before to a ``step 2" which applies Yomdin's theorem to $g^N$, $g^{N+p}\circ\sigma\circ\psi$ and $\hg^{N+p}(\hx_i)$.

Eventually, at step $q$,  we arrive to a family of re\-para\-metrizations $\cR_q$ over $T_i$ which is admissible up to time $qN+p=n$, and which has cardinality
$$
{|\cR_q|\leq \Upsilon^{q+1}\|D\hg^N\|^{q/(r-1)}_{\supnorm}\|D\hg^p\|^{1/(r-1)}_{\supnorm}\equiv \Upsilon^{\lceil n/N\rceil} \|D\hg^N\|_\supnorm^{\frac{\lfloor n/N\rfloor}{r-1}}
  \|D\hg^{n-N\lfloor n/N\rfloor}\|_\supnorm^{1/(r-1)}.}
$$
Taking the union over $i=1,\ldots,{\ell}$, we obtain the family of reparametrizations over $\bigcup T_i\supset T$, as required.
\end{proof}

\subsection{Entropy and growth of $C^r$ complexity of unstable manifolds}\label{ss.unstable-partitions}
{In this section $f$ is a $C^r$ diffeomorphism, $r>1$, of a surface $M$ and $\nu$ is an ergodic hyperbolic probability measure of saddle type, ie}
$\lambda^s:=\lambda^-(f,\nu)$ is strictly negative, and $\lambda^u:=\lambda^+(f,\nu)$ is strictly positive.
{Pesin's Unstable Manifold Theorem says that $\nu$--a.e. $x$ belongs to an unstable manifold $W^u(x)$, which is an injectively immersed $C^r$ curve
and is characterized as:
$$W^u(x)=\{y\::\: \limsup_{n\to +\infty} \frac 1 n \log d(f^{-n}(x),f^{-n}(y))<0\}.$$}

\mbox{}
{A measurable partition $\xi$ is \emph{subordinated} to the unstable lamination $W^u$ of $\nu$ if
for $\nu$-almost every $x\in M$, the atom $\xi(x)$ is a neighborhood of $x$ inside the curve $W^u(x)$ and
$\xi$ is increasing: Every atom of $f(\xi)$ is a union of atoms of $\xi$.
By~\cite{Ledrappier-Strelcyn}, such measurable partitions exist.}

Since $\xi$ is measurable, Rokhlin's disintegration theorem applies, and for $\nu$--a.e. $x$ there exists a probability measure $\nu^u_x$  on $\xi(x)$ so that
$$
        \nu=\int \nu^u_{x} d\nu(x).
        $$
The family $\{\nu^u_x\}$ is not unique, but {given $\xi$,} any two families like that are equal outside a set of $x$ of measure zero. Therefore it is not a serious abuse of terminology to call $\nu^u_x$ {\em the} conditional measure on $\xi(x)$.

In this section we use the entropy theory of Ledrappier and Young \cite{Ledrappier-Young-II} and especially the following corollary established by Zang (see \cite[Remark 1.8]{Zang-EntropyVolume}) to show that the entropy of $\nu$ can be bounded by the exponential rate of growth of the $C^r$ complexity of the curve $f^n(W^u_{loc}(x))$, as quantified in the previous section using admissible $C^r$ reparametrizations up to time $n$.

\begin{theorem}[Y. Zang]\label{t.zang}
Let us consider $f\in\Diff^r(M)$ with $r>1$,
an ergodic hyperbolic probability measure $\nu$, and a system of conditional measures $\{\nu^u_x\}$ on local unstable manifolds.
Then for $\nu$-a.e. $x\in M$, the measure $\nu^u_x$ satisfies:
 $$
    h(f,\nu) =\inf_{\gamma>0} \lim_{\eps\to0} \liminf_{n\to\infty} \frac1n\log r_f(n,\eps,\nu^u_x,\gamma).
 $$
\end{theorem}
\noindent
The difference between this result and Ledrappier-Young theory is that Zang assumes  $C^r$ smoothness for some $r>1$ and hyperbolicity, whereas  Ledrappier and Young assume $C^2$ smoothness, but no hyperbolicity.

\begin{corollary}\label{lemma entropy reparametrizations}
Let $f\in\Diff^2(M)$, and let
$\nu$ be an ergodic hyperbolic measure with a system of conditional measures $\{\nu^u_x\}$ on local unstable manifolds.
For any $F\subset M$ with positive {$\nu$-}measure, for $\nu$-a.e. $x_0\in F$, and for any choice of
 \begin{itemize}
 
  \item $\sigma\colon [0,1]\to W^u(x_0)$,  a { regular $C^r$ curve,}
  \item $T\subset [0,1]$, a set such that $\nu^u_{x_0}(\sigma(T)\cap F)>0$,
\end{itemize}
the following holds.
If  $\cR_n$ $(n\geq 1)$ are families of reparametrizations
of $\sigma$ over $T$ which are $(C^r,f,N,\eps_*,\heps_*)$-admissible up to time $n$ for some (any)
$\eps_\ast,\heps_\ast>0, N\geq 1$ independent of $n$,  then
  $\displaystyle
      h(f,\nu) \leq \liminf_{n\to\infty} \frac1{n}\log |\mathcal R_{n}|.
   $
\end{corollary}
\begin{proof}
Fix $\eps_\ast,\heps_\ast,N>0$ and let $\cR_n$ be families of admissible reparametrizations as in the statement.
Let $\hsigma:[0,1]\to\hM$ be the canonical lift of the regular curve $\sigma$.
Fix $\eps,\heps>0$ arbitrarily small.
By Lemma \ref{lemma covering}, $$r_{\hf}(n,\eps,\heps,\hsigma(T))\leq C|\cR_n|,$$ where
$
C:=C(f,\eps,\heps,\heps_\ast,N):=
{ \frac{2\heps_\ast\|D\hf\|_\supnorm^N }{\min(\eps,\heps)}}
$  is independent of $n$.

By the definition of $r_{\hf}(n,\eps,\heps,\hsigma(T))$, there exist $\hx_1,\ldots,\hx_\ell\in\hM$ with $\ell\leq C|\cR_n|$ such that $\bigcup_{i=1}^\ell {B_{\hf}}(\hx_i,n,\eps,\heps)\supset\hsigma(T)$. Necessarily
$$
\bigcup_{i=1}^\ell {B_f}(\pi(\hx_i),n,\eps)\supset\sigma(T).
$$
In particular,
$
r_f(n,\eps,\sigma(T))
\leq r_{\hf}(n,\eps,\heps,\hsigma(T))\leq C|\cR_n|,
$
whence by Zang's Theorem, $
h(f,\nu)\leq \lim\limits_{\eps\to 0}
\liminf\limits_{n\to\infty}\frac{1}{n}\log r_f(n,\eps,{\sigma(T)\cap F})\leq \liminf\limits_{n\to\infty} \frac{1}{n}\log|\cR_n|.$
\end{proof}

\section{Main reparametrization lemmas}\label{s.reparametrization}

This section collects our main technical results on the existence of admissible families of reparametrizations of pieces of unstable manifolds.

The point is to produce families with cardinality as small as possible. The first result provides admissible families of reparametrizations of local unstable manifolds, with cardinality controlled in terms of the entropy.
The second result, which is much more subtle, produces much smaller families of reparametrizations for the subset of the local unstable stable where there is little expansion up to some iterate, see Definition~\ref{d.neutralblock}.

\subsection{Statements}
Throghout this section $M$ is a compact smooth surface without boundary, $f:M\to M$ is a diffeomorphism, $\hf:\hM\to\hM$ is the canonical lift \eqref{canonical-lift}, and
$\hmu$ is a (possibly non-ergodic!) $\hf$-invariant probability measure, which projects to an $f$-invariant
measure $\mu$.
 $Q_{r,N}(f)$ is given by \eqref{Q-r-n} {in the previous section,} $\oh$ is the essential entropy \eqref{eq-ergdec-h}, and $\lambda(\hf)$ is the
 asymptotic dilation of $\hf$, see \eqref{asymp-dil} and \S\ref{ss.size-lift}.

The statements of the following two propositions should be formally understood as stating the existence of functions  $N_1,
\overline{n}_1,$
{$\gamma_0$,} and $N_0$ with values in $(0,\infty)$ such that the following {stated} properties hold.

\begin{proposition}\label{p.estimate1'}
Let  us consider $f\in\Diff^r(M)$ with $2\leq r<\infty$, an $\hf$-invariant probability $\hmu$,
some real numbers ${Q},\eta,\gamma,\eps,\heps>0$, and integers $N,n$. Assume that:
 \begin{itemize}
  \item[--] $N\ge N_1(r,f,{\eta})$,
  \item[--]  $0<\eps,\heps \le {\eps_Y(r,{Q})}$ {(the constant in Yomdin's theorem~\ref{t-yomdin})}, and
  \item[--] $n\geq \overline{n}_1(f,\hmu,\eta,\gamma,N,\varepsilon,\widehat\varepsilon)$.
 \end{itemize}
Then there are:
 \begin{itemize}
  \item[--] a {$C^2$} neighborhood $\mathcal U_1(f,\eta,\gamma,\eps,\heps,N,n)$ of $f$ in $\Diff^r(M)$,
  \item[--] an open set $\hU_1(f,\hmu,\eta,\gamma,\eps,\heps,n)\subset\hM$ with $\hmu(\hU_1)>1-\gamma^2$ and $\hmu(\partial\hU_1)=0$,
\end{itemize}
such that the following property holds:

\begin{enumerate}[(*)]
\item For any $g\in\mathcal U_1$ with $Q_{r,N}(g)<{Q}$ and for any regular curve $\sigma$ with $C^r$ size at most $(\eps,\heps)$, there is a family $\mathcal R$ of reparametrizations over $\widehat \sigma^{-1}(\widehat U_1)$ s.t.
 \begin{itemize}
  \item[--] $\mathcal R$ is $(C^r,g,N,\eps,\heps)$-admissible up to time $n$,
  \item[--] $|\mathcal R|\leq \exp \left[n\left(\oh({f},\mu)+\frac{\lambda({\hf})}{r-1}+\eta\right)\right]$,
 \end{itemize}
 \end{enumerate}
\end{proposition}

The following  and key estimate applies to the part of the local unstable manifold which
{does not (initially) see much expansion}.
More precisely, suppose $g$ is a diffeomorphism with canonical lift $\hg$, and let  $\alpha>0$. An \emph{orbit segment} with length $n$ is a string $\vartheta:=(\hx,\hg(\hx),\ldots,\hg^{n-1}(\hx))$.
\begin{definition}
An orbit segment $\vartheta=(\hx,\hg(\hx),\ldots,\hg^{n-1}(\hx))$ is {\em $\alpha$-neutral}, if, denoting $\hx=(x,E)$, we have
$
\|Dg_x^m|_E\|\leq e^{\alpha m}\text{ for every }1\leq m\leq  n.
$
\end{definition}
\begin{proposition}\label{p.estimate2'}
Let  us consider $f\in\Diff^r(M)$ with $2\leq r<\infty$, an $\hf$-invariant probability $\hmu$,
some real numbers ${Q},\eta,\gamma,\delta>0$, and an integer $N$. Assume that:
 \begin{itemize}
  \item[--] $0<\gamma\leq\gamma_0(r,f,\eta)$,
  \item[--] $N\ge N_0(r,f,\eta,\gamma)$.
 \end{itemize}
Then there are:
 \begin{itemize}
  \item[--] $0<\eps,\heps\le\delta$.
  \item[--] a {$C^2$} neighborhood $\mathcal U_0(f,\hmu,{Q},\eta,\gamma,\delta,N)$ of $f$ in $\Diff^r(M)$,
  \item[--] an open set $\hU_0(f,\hmu,{Q},\eta,\gamma,\delta,N)\subset\hM$ with $\hmu(\hU_0)>1-\gamma^2$ and $\hmu(\partial\hU_0)=0$,
  \item[--] an integer $\overline{n}_0:=\overline{n}_0(f,\hmu,{Q},\eta,\gamma,N,\delta)\geq1$,
\end{itemize}
such that the following property holds:

\begin{enumerate}[(**)]
\item For any $g\in\mathcal U_0$ with $Q_{r,N}(g)<{Q}$, any regular curve $\sigma$ with \Crr size at most $(\eps,\heps)$, and any $n\geq\overline{n}_0$ there is a family $\mathcal R_n$ of reparametrizations over
\begin{equation}\label{e.T}
T:=\widehat \sigma^{-1}\bigg\{\hx\colon
\text{ $(\hx,\dots,\widehat g^{n-1}(\hx))$  is  $\textstyle\frac \eta{10}$-neutral and
$\frac{1}{n}\sum_{j=0}^{n-1}\delta_{\hg^j(\hx)}(\hU_0)> 1-\gamma$}\bigg\},
\end{equation}
such that
 \begin{itemize}
  \item[--] $\mathcal R_n$ is $(C^r,g,N,\eps,\heps)$-admissible up to time $n$,

  \item[--] $|\mathcal R_n|\leq \exp \left[n\left(\frac{\lambda({\hf})}{r-1}+\eta\right)\right]$.
 \end{itemize}
\end{enumerate}
\end{proposition}

\noindent
Unlike Proposition \ref{p.estimate1'}, here the upper bound has  no entropic term. Indeed,
in the $C^\infty$ case,
{the exponential rate of growth} tends to zero with $\eta$. This low complexity is due to the neutrality of the piece of $\sigma$ we are parametrizing.

\medskip
The proofs of these two propositions may be skipped at the first reading.

\subsection{Proof of Proposition \ref{p.estimate1'}}
Let $r\in[2,\infty)$, $f\in\Diff^r(M)$, ${Q},\eta,\gamma>0$, and consider a $\hf$-invariant probability measure $\hmu$.
Note that the lift $\hf\in\Diff^{r-1}(\hM)$ is uniquely defined by $f$ and depends continuously on  $f\in\Diff^r(M)$.
Recall  the number $\Upsilon:=\Upsilon(r)$ given by Yomdin's Theorem~\ref{t-yomdin}.

\proofitem
Fix an integer $N_1=N_1(r,f,\eta)\geq 1$ such that for all $N\ge N_1$,
\begin{equation}\label{e.choice-N}
    \Upsilon<\exp(\tfrac{\eta N}{10}) \;\;\text{ and } \;\; \tfrac{1}{N}\log\|D\hf^{N}\|_\supnorm< \lambda(\hf)+\tfrac{\eta}{10}.
\end{equation}

This is possible since $\lambda(\hf)=\lim_{N\to +\infty}\frac{1}{N}\log\|D\hf^N\|_\supnorm$.

\proofitem
Fix $N\ge N_1$.

\proofitem
Set ${\eps_Y}:={\eps_Y}(r,{Q})$ as in Yomdin's Theorem ~\ref{t-yomdin}.

\proofitem
Let $\eps,\heps$ be arbitrary in $(0,\eps_Y)$.

\proofitem
Pick an integer $\overline{n}_1=\overline{n}_1(f,\hmu,\eta,\gamma,N,\eps,\heps) > N$ using Proposition~\ref{fibered Bowen-Katok}\eqref{i.Katok} such that for any $n\geq \overline{n}_1$,
 \begin{equation}\label{eq-K-oh}
\tfrac 1 n \log r_{\widehat f}(n,\tfrac{\varepsilon}{2},\tfrac{\widehat \varepsilon}{ 2},\hmu,1-\gamma^2)< \overline h({\hf},\hmu)+\tfrac \eta 4,
 \end{equation}
  \begin{equation}\label{eq-choicen1}
  \log\|D\hf\|_\supnorm^N< \tfrac {n\eta} {10}.
 \end{equation}

\proofitem
Let $n$ be some integer{larger than $\overline{n}_1$}.

\proofitem
Let $\mathcal U_1=\mathcal U_1(f,\eta,N,n,\eps,\heps)\subset\Diff^r(M)$ be a small enough {$C^2$} neighborhood of $f$ in $\Diff^r(M)$ such that the lift $\hg$ of any $g\in\mathcal U_1$ satisfies:
 \begin{equation}\label{eq-choice-g}
 \|D\hg^{N}\|_\supnorm\leq e^{\eta N/10}\|D\hf^{N}\|_\supnorm,\quad \|D\hg\|_\supnorm\leq e^{\eta /10}\|D\hf\|_\supnorm,
 \end{equation}
 and so that every fibered $(n,\frac{\eps}{2},\frac{\heps}{2})$-Bowen ball for $\hf$ is contained in a fibered  $(n,{\eps},{\heps})$-Bowen ball for $\hg$.

By the regularity of Borel measures, there is a compact set $\hK=\hK(f,\hmu,\eta,\gamma,n,\varepsilon,\widehat{\varepsilon})$ with $\hmu(\hK)>1-\gamma^2$. By eq.~\eqref{eq-K-oh}
 there is  a neighborhood $\hU_1$ of $\hK$, such that $\hg(\hU_1)$ is contained in the union of a collection $\mathcal C$ of fibered $(n,\frac{\eps}{2},\frac{\heps}{2})$-Bowen balls for $\hf$ with cardinality at most $\exp(n(\oh(f,\hmu)+\frac{\eta}{4}))$. Passing to a smaller open set  containing $\hK$, we can ensure that $\hmu(\partial\hU_1)=0$.

\medbreak

Suppose $g\in\mathcal U_1$, $Q_{r,N}(g)<{Q}$, and let $\sigma$ be a regular $C^r$ curve with $C^r$  size at most $(\eps,\heps)$.
By construction, we can cover  $\hg(\hsigma[0,1]\cap \hU_1)$ using only the $(n,\frac{\eps}{ 2},\frac{\heps}{ 2})$-Bowen balls for $\hf$ from the collection $\mathcal C$. Every  ball in $\mathcal C$  is contained in some ${(n,\eps,\heps)}$--Bowen ball for $\hg$. Thus,
 $$r_{\hg}(n,\eps,\heps,{\hg(\hsigma[0,1]\cap \hU_1)})\le \exp(n(\overline{h}( {f},\hmu)+{\textstyle \frac \eta 4})).$$

Since $\varepsilon,\widehat{\varepsilon}<{\eps_Y}$, we can apply Yomdin theory in the form of
{Corollary~\ref{coro.reparametrization}. The result is  a family $\cR$ of reparametrizations of $\sigma$ over $\hsigma^{-1}\hU_1$
which is $(C^r,g,N,\eps,\heps)$--admissible up to time $n$, and such that
{
$$|\cR|\leq \Upsilon^{\lceil n/N \rceil}     \|D\hg^N\|_\supnorm^{\frac{\lfloor n/N\rfloor}{r-1}}
\|D\hg\|_\supnorm^\frac{N}{r-1}
\exp(n(\oh(f,\hmu)+{\textstyle \frac \eta 4})).$$
}
{Using~\eqref{e.choice-N}, \eqref{eq-choicen1} and~\eqref{eq-choice-g}, one gets
$|\cR|\leq \exp(n(\oh(f,\hmu)+\frac{\lambda(\hf)}{r-1}+\eta))$.}\hfill$\Box$

\subsection{Proof of Proposition \ref{p.estimate2'}}\label{proof-2'}
{The proof splits into two parts: In steps 0-6 we {select the parameters}
 $\gamma_0,N_0,\varepsilon,\heps,\cU_0,\hU_0,\overline{n}_0$;
In steps 7-11 we build admissible families of reparametrizations as in the statement,
 and estimate their cardinality.}

Fix $r\in[2,\infty)$, $f\in\Diff^r(M)$, ${Q},\eta,\delta>0$, and consider an $\hf$-invariant probability measure $\hmu$.
Let $\Upsilon:=\Upsilon(r)$ given by  Yomdin's Theorem \ref{t-yomdin}. Let
\begin{equation}\label{H-def}
H(t):=t\ln\tfrac{1}{t}+(1-t)\ln\tfrac{1}{1-t}.
\end{equation}

\subsection*{Step 0 (Preliminary choices)}

\proofitem
Choose $\gamma_0=\gamma_0(r,f,\eta)>0$ such that:
\begin{equation}\label{e.choice-gamma}
\forall0<\gamma\le\gamma_0,\quad
3\Upsilon\leq \exp(\tfrac{\eta }{10\gamma}), \ H(4\gamma)<\tfrac \eta {10}, \;  \|D\widehat f\|_\supnorm < \exp\big(\tfrac{\eta}{10\gamma}\big).
\end{equation}

\proofitem Let $0<\gamma<\gamma_0$ be arbitrary.

\proofitem Fix an integer $N_0=N_0(r,f,\gamma,\eta)>\max(\tfrac 1\gamma,\tfrac{10}\eta)$  such that (cf. \eqref{asymp-dil}, \S\ref{ss.size-lift}):
\begin{equation}\label{e.N-zero}
   \forall N\geq N_0,\quad
   \tfrac{1}{N}\log \|D\widehat f^N\|_\supnorm<
   {\lambda(\widehat f)}+\tfrac \eta{10} .
\end{equation}

\proofitem
Let $N\ge N_0$ be arbitrary.

\proofitem Fix ${\eps_Y}={\eps_Y}(r,{Q})$ as in
Yomdin's Theorem \ref{t-yomdin}.

\proofitem
Let $\cV=\cV(f,\eta,N)$ be a {$C^2$} neighborhood of $f$ in $\Diff^r(M)$
such that for $g\in\cV$,
\begin{equation}\label{delta-sharp}
\|D\hg^k\|_\supnorm\leq e^{\eta/10}\|D\hf^k\|_\supnorm\qquad  (k=1,\ldots,N).
\end{equation}

\subsection*{Step 1 (Decomposition of $\hmu$)}
$\hmu$ projects to an $f$-invariant probability measure. Let $\lambda_x^{\pm}:=\lambda^{\pm}(f,x)$ denote  the  Lyapunov exponents. These are well-defined $\mu$-a.e., but since we are assuming nothing on $\hmu$ they could be equal on a set of positive measure. Consider the invariant measurable subset
$$
   M_\#:=\{x\in M \mid \lambda_{x}^-\text{ and }\lambda_{ x}^+\text{ are defined and distinct}\}.
$$
We decompose $\hmu$ as a barycenter of two invariant probability measures:
$$\hmu=a\hmu_\# + (1-a)\hmu_0,$$
where $a:=\mu(M_\#)$,
$a \hmu_\#:={\hmu(\cdot\cap \hpi^{-1}(M_\#))}$ and
$(1-a)\hmu_0:={\hmu(\cdot\cap \hpi^{-1}(M\setminus M_\#))}$.
In the following we assume $a\in (0,1)$.
Indeed, in the special case where $a=0$ (resp. $a=1$), we simply write $\hmu=\hmu_0$ (resp. $\hmu=\hmu_\#$) and the proof of Proposition \ref{p.estimate2'}
becomes simpler and can be easily obtained by adapting the general case.

Let $\mu_0$ and $\mu_{\#}$ be the projections of $\hmu_0,\hmu_{\#}$ to $M$. These are  $f$-invariant measures;
 $\mu_0$--a.e. $x$ has two equal Lyapunov exponents; and  $\mu_\#$--a.e. $x$ has two different Lyapunov exponents.

\subsection*{Step 2 (Compact subsets $\hK_+,\hK_-$ approximating $\hmu_\#$)}
As in section~\ref{ss.exponent}, the Oseledets splitting induces two $\hf$-invariant measurable sections $x\mapsto (x,E^\pm(x))$, $M\to\hM$, defined $\mu_\#$-a.e.

\proofitem
By Lusin's theorem there exists a compact set $K_\#=K_\#(f,\hmu,\gamma,N)$ inside  $M_\#$ such that  $\mu_\#(K_\#)>1-\gamma^2$, and so that  the functions
  $$
     x\mapsto \bigl(x,E^{+}(f^k(x))\bigr)\ ,\ x\mapsto \bigl(x,E^-(f^k(x))\bigr)\qquad  (0\leq k\leq N)
  $$
 {are} continuous on $K_\#$.
By Lemma \ref{Lemma-Lyap}, $\hmu_{\#}$ is carried by
$
\mathrm{graph}(E^+)\cup\mathrm{graph}(E^-)
$. Therefore, the sets
 $$
   \hK^+:=\operatorname{graph}(E^+|_{K_\#}) \text{ and }
   \hK^-:=\operatorname{graph}(E^-|_{K_\#})
 $$
are compact and disjoint, and  $\hmu_\#(\hK^+\cup\hK^-)=\mu_\#(K_\#)>1-\gamma^2$.
We set
 $$
    \hK_\# := \hpi^{-1}(K_\#) = \hK_-\sqcup\hK_+.
 $$
Every $z\in K_\#$ has two lifts: $\hz^+\in\hK^+$, and $\hz^-\in\hK^+$. It is easy to see that  $\hf^k(\widehat{z}^{\pm})=(f^k(z),E^{\pm}(f^k(z)))$.

\subsection*{Step 3 (Control of $N$ iterates starting near $\hK_\#$)} $ $

\proofitem There exist  ${\eps_{*}}={\eps_{*}}(r,f,\eta,N)>0$ and a {$C^2$} neighborhood ${\cU_{*}}={\cU_{*}}(f,\eta,N)$ of $f$ in $\Diff^r(M)$
such that the conclusion of Lemma~\ref{l.length} holds.

\proofitem Let $\heps=\heps(r,f,{Q},\eta,\delta,N,K_\#):=\tfrac 1 {10} \min(\delta,{\eps_Y},
{\varepsilon_{*}},\dist(\hK_+,\hK_-)).$ Note that the latter distance is positive since
$\hK_+$ and $\hK_-$ are disjoint compact sets.

\proofitem
By construction of $K_\#$, there is $\eps=\eps(f,\heps,K_\#,N)\in (0,\heps)$  such that
 if $x,y\in K_\#$ satisfy $d(x,y)\leq 
 \varepsilon$,
then
$$
\forall 0\leq k\leq N,\quad d(\hf^k(\hx^+),\hf^k(\hy^+))<\tfrac \heps 2 \text{ and } \ d(\hf^k(\hx^-),\hf^k(\hy^-))<\tfrac \heps 2.
$$
 By the choice of $\heps$,  if $\hx,\hy\in\hK_\#$ satisfy  $d(\hx,\hy)\leq 
 \heps$, then
 either $\hx=\hx^+, \hy=\hy^+$ or
$\hx=\hx^-, \hy=\hy^-$.
 Consequently, for any $\hx,\hy\in\hM$,
$$
\left.
\begin{array}{c}
\hx,\hy\in \hK_\#\\
d(\hx,\hy)\leq \heps,\;\;
d(x,y)\leq \eps
\end{array}\right\}\Longrightarrow\;\;
\forall 0\leq k\leq N,\quad
d(\hf^k(\hx),\hf^k(\hy))< \tfrac \heps 2.
$$

Below, $B(S,\beta)$ denotes the $\beta$-neighborhood of a subset $S$. 

 \proofitem
Since $\hM$ is compact, there exist ${\delta_{\ast\ast}}={\delta_{\ast\ast}}(f,N,\eps,\heps)>0$ and a {$C^2$} neighborhood ${\cU_{**}}={\cU_{**}}(f,N,\eps,\heps)$ of $f$ in $\Diff^r(M)$ such that
\begin{equation}\label{ss.W-plus}
\left.
\begin{array}{c}
g\in{\cU_{**}}\\
\hx,\hy\in B(\hK_\#,{\delta_{\ast\ast}})\\
d(x,y)\leq \eps,\;\;
d(\hx,\hy)\leq\heps
\end{array}\right\}\Longrightarrow\;\;
\forall 0\leq k\leq N,\quad
d(\hg^k(\hx),\hg^k(\hy))<\heps.
\end{equation}

 \proofitem
Choose an open set $\hW_\#=\hW_\#(\hK_\#,{\delta_{\ast\ast}})$ such that  $\hK_\#\subset \hW_\#\subset B(\hK_\#,{\delta_{\ast\ast}})$ and $\hmu(\partial\hW_\#)=0$. By \eqref{ss.W-plus}, if $\sigma$  is a regular curve
with $C^r$ size at most $(\varepsilon,\widehat \varepsilon)$, then for
all $g\in{\cU_{**}}$,
\begin{equation}\label{e.control-diameter}
\forall 0\leq k\leq N\quad \diam(\widehat g^k(\widehat \sigma\cap \hW_\#))<\widehat \varepsilon.
\end{equation}

\subsection*{Step 4 (Control of $n_*$ iterates starting near $\hmu_0$)}
Recall that $\mu_0$ has equal Lyapunov exponents almost everywhere. By Ruelle's inequality $\mu_0$ must have zero entropy. It follows that $\oh(f,\mu_0)=0$.

\proofitem
By Proposition~\ref{fibered Bowen-Katok}, $\displaystyle\lim_{\overline{\gamma}\to 1}\lim_{\overline{\eps}\to 0}\limsup_{n\to\infty}\frac{1}{n}\log r_{\hf}(n,\overline{\eps},\tfrac{\heps}{2},\hmu_0,\overline{\gamma})=0$.
If $0<\overline{\eps}<\frac{\eps}{2}$ and $\gamma<\overline{\gamma}<1$, then
$r_{\hf}(n,\frac{\eps}{2},\frac{\heps}{2},\hmu,{\gamma})\leq r_{\hf}(n,\overline{\eps},\frac{\heps}{2},\hmu_0,\overline{\gamma})$. It follows that
$$
\limsup_{n\to\infty}\frac{1}{n}\log r_{\hf}(n,\tfrac{\eps}{2},\tfrac{\heps}{2},\hmu_0,\overline{\gamma})=0.
$$
So we can find a large integer $n_*=n_*(f,\hmu,\eta,\gamma,\varepsilon,\widehat \varepsilon,N)$ that is a multiple of $N$
and a compact set $\hK_0:=\hK_0(\hf,\hmu,\eta,\gamma,\eps,\heps,n_*)\subset \hM$ with $\hmu_0(\hK_0)>1-\gamma^2$, s.t.
\begin{equation*}
\tfrac 1 {n_*} \log r_{\widehat f}(n_*,\tfrac\varepsilon 2, \tfrac{\widehat \varepsilon}2,{\hf}(\widehat K_{0}))< \tfrac \eta{10}.
\end{equation*}

\proofitem
By continuity, we can choose a neighborhood $\hW_0=\hW_0(f,\eta,\eps,\heps,n_*,\hK_0)$ of $\hK_0$ with $\hmu(\partial\hW_0)=0$, and a $C^2$ neighborhood $\cU_{***}=\cU_{***}(f,\eta,\eps,\heps,n_*,\hK_0)$ of $f$ in $\Diff^r(M)$ so that:
 \begin{equation}\label{e.growth}
   \forall g\in\cU_{***},\quad
   \tfrac 1 {n_*} \log r_{\hg}(n_*,\varepsilon, \widehat \varepsilon,\hg(\widehat W_{0}))< \tfrac \eta{10}.
 \end{equation}

\subsection*{Step 5 ($\hU_0$ and a decomposition of long typical orbit segments)}
Let $$\hU_0:=\hW_0\cup\hW_\#.$$
Then  $\hU_0=\hU_0(r,f,\hmu,{Q},\eta,\gamma,\delta,N)$, $\hmu(\hU_0)> 1-\gamma^2$  and $\hmu(\partial \hU_0)=0$.
\begin{claim}\label{c.decomp}
Suppose $n>{n_{*}/\gamma}$.
Any orbit segment $(\hx,\hg(\hx)\dots,\widehat g^{n-1}(\hx))$ which spends a proportion of time
larger than $1-\gamma$ in $\hU_0$
can be decomposed into:
\begin{enumerate}[(a)]
\item orbit segments of length $n_*$ with initial point in $\hW_0$,
\item orbit segments of length $N$ with initial point in $\hW_\#$,
\item orbit segments of length $1$, of total number less than $2\gamma n$.
\end{enumerate}
\end{claim}
\noindent
($\hW_0$ and $\hW_\#$ are not necessarily disjoint, so the  decomposition may not be unique.)

\begin{proof}
A decomposition as in the statement is completely characterized by the increasing sequence of times
$(n_0,n_1,\dots,n_\ell)$, where  $\hg^{n_i}(\hx)$ is the initial point of the $i$-th segment. (So  $n_0=0$, $n_\ell=n$.)

We set $n_0=0$ and define the sequence inductively.
Assuming that $n_i<n$ has already been defined, we set
\begin{enumerate}[(a)]
\item $n_{i+1}:=n_i+n_*$ if $n_i+n_*\leq n$ and $\hg^{n_i}(\hx)\in\hW_0\setminus \hW_\#$,
\item $n_{i+1}:=n_i+N$ if $n_i+N\leq n$ and $\hg^{n_i}(\hx)\in\hW_\#$,
\item $n_{i+1}:=n_i+1$ otherwise.
\end{enumerate}
We stop when $n_{i+1}=n$.

{Since $n_*\geq N$,} the times $n_i$ such that
$\widehat g^{n_i}(\hx)\in \hU_0$ but which are not associated to case (a) or (b) must
satisfy ${n_i}>n-n_*$. Since {$n>n_*/\gamma$} and since $(\hx,\dots,\widehat g^{n-1}(\hx))$ spends a proportion of time
larger than $1-\gamma$ in $\widehat U_0$, the set of times $n_i$ corresponding to case (c)
has size smaller than $2\gamma n$.
\end{proof}

The sequence of times $\theta:= (n_0,n_1,\dots,n_{\ell})$
obtained in the previous claim is called
\emph{type of a decomposition}. Recall that $H(t)=t\ln \frac{1}{t}+(1-t)\ln\frac{1}{1-t}$.
\begin{claim}\label{c.bound-types}
There exists $n_H:=n_H(\gamma)$ such that for all  $n>n_H$, the number of possible types $\theta$ is less than $\exp(H(4\gamma)n)$.
\end{claim}
\begin{proof}
By our choices of $N,N_0$ and $n_\ast$, we have $n_*,N>1/\gamma$. Hence there can be at most  $\gamma n$ times $n_i$ such that $n_{i+1}-n_i\in \{n_\ast,N\}$. Since there are also at most $2\gamma n$ times $n_i$ such that $n_{i+1}-n_i=1$, we must have
$
\ell\leq \lfloor 3\gamma n\rfloor.
$

The number of types is thus bounded by $\sum_{\ell=1}^{\lfloor 3\gamma n\rfloor}{{n} \choose \ell-1}$. Since $3\gamma<\frac{1}{2}$, this is less than
$3\gamma n {n\choose\lfloor{3\gamma n}\rfloor}$, which by de Moivre's approximation is less than $\exp[n H(3\gamma)+o(n)]$. Since $4\gamma<\frac{1}{2}$, we have $H(3\gamma)<H(4\gamma)$, and the claim follows.
\end{proof}

\subsection*{Step 6 (Definition of $\cU_0,\overline{n}_0$)}
We fix the last parameters of our construction.
Recall the  {$C^2$} neighborhoods of $f$ in $\Diff^r(M)$ introduced in Lemma~\ref{l.length} and eqs. \eqref{delta-sharp}, \eqref{ss.W-plus}, \eqref{e.growth}.

\proofitem
Let 
$\cU_0:=\cU_0(r,f,\hmu,{Q},\eta,\gamma,\eps,\heps,n_*,N):=\cV\cap\cU_{*}\cap\cU_{**}\cap\cU_{***}$.

\proofitem
Define $\overline{n}_0=\overline{n}_0(f,\hmu,\eta,\gamma,\delta,N):=\max\{{n_{*}/\gamma},n_H\}$.

\bigskip

\subsection*{Step 7 (An inductive scheme)}
{Now} we fix some $g\in\cU_0$ with $Q_{r,N}(g)<{Q}$,
a regular $C^r$ curve  $\sigma$ with $C^r$ size at most $(\varepsilon,\widehat \varepsilon)$ and $n>\overline{n}_0$.
We need to {bound} the minimal cardinality of a family $\cR_n$ of reparametrizations of $\sigma$
which are $(C^r,g, N,\varepsilon, \widehat \varepsilon)$-admissible up to time $n$, over the set
$$T=\widehat \sigma^{-1}\bigg\{\hx\colon
\text{ $(\hx,\hg(\hx),\dots,\widehat g^{n-1}(\hx))$  is  $\textstyle\frac \eta{10}$-neutral and
$\frac{1}{n}\sum_{j=0}^{n-1}\delta_{\hg^j(\hx)}(\hU_0)\geq 1-\gamma$}\bigg\}.$$

For each type $\theta=(n_0,\ldots,n_\ell)$, we introduce the corresponding subset
$$T_\theta:=T\cap \widehat \sigma^{-1}\bigg\{\hx\colon
\text{ $(\hx,\hg(\hx),\dots,\widehat g^{n-1}(\hx))$ has type $\theta$}\bigg\}.$$
Then $T$ is the union of $T_\theta$ over all possible type $\theta$.

Fixed some type $\theta=(n_0,\ldots,n_\ell)$. We will build by  induction a family
$\cR_{n_i}^\theta$ of reparametrizations $\psi$ of  $\sigma$ over $T_\theta$ satisfying the following properties:
 \begin{enumerate}[\quad(i)]
  \item admissibility: $\cR_{n_i}^\theta\text{ is $(C^r, g,N,\varepsilon, \widehat \varepsilon)$-admissible up to time $n_i$};$
  \item small cardinality: if $i\geq 1$,
\begin{equation*}
\begin{aligned}
|\cR^\theta_{n_i}|
    &\le \exp \left((\tfrac{\lambda(\hf)}{r-1} + \tfrac {7\eta}{10})n_i\right)|\cR^\theta_{n_{i-1}}| \qquad \text{when }n_{i}-n_{i-1}\geq N,\\
|\cR^\theta_{n_i}|    & \leq \exp(\eta/\gamma) \; |\cR^\theta_{n_{i-1}}| \qquad\;\quad\text{otherwise;}
    \end{aligned}
\end{equation*}
  \item small length: for each $\psi\in \cR_{n_i}^\theta$ and
any $(x,E)\in\hsigma\circ\psi([0,1])$,
\begin{equation*}
\operatorname{Length}(g^{n_i}\circ\sigma\circ \psi)< \varepsilon e^{-\frac {\eta} {10} n_i}\|Dg^{n_i}_x|_E\|.
\end{equation*}
 \end{enumerate}
At the end of the construction,
one obtains a family $\cR_n^\theta:=\cR_{n_\ell}^\theta$ over $T_\theta$ which is admissible up to time $n$. Then one can take the union over all $\theta$ and finish the construction.

We begin the construction by defining $\cR_{0}^\theta:=\{\operatorname{Id}\}$: This meets our requirements because $n_0=0$ and  $\sigma$ has $C^r$ size at most
$(\varepsilon,\widehat \varepsilon)$.

\smallskip

 Now we assume by induction that  $\cR_{n_i}^\theta$ has been constructed, and  we build $\cR_{n_{i+1}}^\theta$.  The construction uses the concatenation procedure described in Lemma \ref{l.concatenation}.
For each $\psi\in\cR^\theta_{n_i}$ we have to build a family $\cR_\psi$ of reparametrizations of the curve
$g^{n_i}\circ \sigma \circ \psi$ over $\psi^{-1}(T_\theta)$ with the following properties:
 \begin{enumerate}[\quad(i')]
   \item $\cR_\psi$ is  $(C^r,g,N,\eps,\heps)$-admissible up to time $n_{i+1}-n_i$,
  \item $\log |\cR_\psi|$ is bounded by $(\tfrac{\lambda(\hf)}{r-1} + \tfrac {7\eta}{10} )(n_{i+1}-n_i)$ if $n_{i+1}-n_i\geq N$ and by  $\tfrac\eta\gamma$ otherwise,
  \item for each $\varphi\in \cR_\psi$ and $(x,E)\in\hsigma\circ\psi\circ \varphi([0,1])$,
$$
\operatorname{Length}(g^{n_{i+1}}\circ\sigma\circ \psi\circ \varphi)< \varepsilon e^{-\frac {\eta} {10} n_{i+1}}\|Dg^{n_{i+1}}_x|_E\|.$$
\end{enumerate}
The family $\cR^\theta_{n_{i+1}}:=\{\psi\circ \varphi,\; \psi\in \cR^\theta_{n_i}\; \varphi\in \cR_\psi\}$ then satisfies (i--iii) above.

\proofitem Given a type $\theta=(n_0,\ldots,n_\ell)$, an integer $i\in\{0,\dots,\ell-1\}$ and a reparametrization
$\psi\in \cR^\theta_{n_i}$, the construction of the families $\cR_\psi$ depends on which of the following cases from Claim \ref{c.decomp} holds for $n_i$:
\begin{enumerate}[\quad C{a}se (a):]
\item $\hg^{n_i}(\hx)\in \hW_0$ and $n_{i+1}-n_i=n_\ast$.
\item $\hg^{n_i}(\hx)\in\hW_\#$ and $n_{i+1}-n_i=N$.
\item $n_{i+1}-n_i=1$.
\end{enumerate}
The three cases are discussed in  steps 8-10 below.

In order to simplify the notations,
we set $$\sigma':=g^{n_i}\circ\sigma\circ\psi \text{ and } T':=\psi^{-1}(T_\theta).$$ Note that $\sigma'$ has $C^r$ size at most $(\eps,\heps)$;
moreover the induction assumption (iii) gives for each $(x,E)\in \hg^{-n_i}\circ \hsigma'[0,1]$,
\begin{equation}\label{e.length}
\operatorname{Length}(\sigma')< \varepsilon e^{-\frac {\eta} {10} n_i}\|Dg^{n_i}_x|_E\|.
\end{equation}

\subsection*{Step 8 (Case (a)):} In this case $n_{i+1}-n_i=n_*$ and $\hsigma'(T')\subset\hW_0$.  By eq.~\eqref{e.growth},
 \begin{equation}\label{eq-rhW0}
    r_{\hg}(n_*,\eps,\heps,\hg\circ\hsigma'(T')) \leq e^{\eta n_*/10}.
 \end{equation}
The integer $n_*$ is a multiple of $N$.
Corollary~\ref{coro.reparametrization} of Yomdin's theorem yields a family $\cR^0_\psi$ of reparametrizations of $\sigma'$ over $T'$ which is $(C^r,g,N,\eps,\heps)$-admissible up to time $n_*$ and with cardinality:
 $$\begin{aligned}
    |\cR^0_\psi| &\leq \Upsilon^{n_*/N} \|D\hg^N\|_\supnorm^{n_*/(r-1)N} r_{\hg}(n_*,\eps,\heps,\hg\circ\hsigma'(T')). \end{aligned}$$
Combining with~\eqref{e.choice-gamma}, \eqref{e.N-zero}, \eqref{delta-sharp}, \eqref{eq-rhW0}
and $N>\frac 1 \gamma$, we get $\log |\cR^0_\psi| \leq (\tfrac{\lambda(\hf)}{r-1} + \tfrac {4\eta} {10})n_*$.

The conclusion of Lemma~\ref{l.length}, together with~\eqref{e.length} gives that, for each $\varphi\in\cR^0_\psi$
and $(x,E)\in \hg^{-n_i}\circ\hsigma'([0,1])$,
\begin{align}
&\mathrm{Length}(g^{n_*}\circ\sigma'\circ \varphi)<
  \varepsilon e^{-\frac{\eta}{10} n_{i+1}}\|D_xg^{n_{i+1}}|_E\|  \cdot e^{\frac{\eta}{5} n_*}.\label{length-est-a}
 \end{align}
In order to compensate for the factor $e^{\frac{\eta}{5} n_*}$, we  subdivide $[0,1]$ into
intervals $I_1,\dots I_m$ with length less or equal to $e^{-\frac{\eta}{5} n_*}$.
Since $n_\ast\geq N\geq N_0>\tfrac{10}\eta$, $e^{\frac{\eta}{10}n_\ast}>2$, whence
$$
  m\leq \lceil e^{\frac{\eta}{5} n_*}\rceil {< 2 e^{\frac{\eta}{5} n_*}} <e^{\frac{3\eta}{10} n_*}.
$$
Let  $\chi_j\colon [0,1]\to I_j$ be affine bijections, and let
$$\cR_\psi:=\{\varphi\circ \chi_j, \: \varphi\in {\cR^0_\psi}, j=1,\dots,m\}.$$
The cardinality of $\log |\cR_\psi|$ is thus bounded as required by $(\tfrac{\lambda(\hf)}{r-1} + \tfrac {7\eta} {10})n_*$, so (ii') holds.
Property (iii') follows from \eqref{length-est-a} and the choice of $m$ and $\chi_j$, and (i') follows from Lemma \ref{l.concatenation}.

\subsection*{Step 9 (Case (b)):} In this case $n_{i+1}-n_i=N$ and $\hsigma'(T')\subset\hW_\#$.
We combine Lemma~\ref{l.length} with~\eqref{e.length} and get that, for each $(x,E)\in \hg^{-n_i}\circ\hsigma'([0,1])$,
\begin{align*}
&\mathrm{Length}(g^{N}\circ\sigma')<
  \varepsilon e^{-\frac{\eta}{10} n_{i+1}}\|D_xg^{n_{i+1}}|_E\|  \cdot e^{\frac{\eta}{5} N}.
 \end{align*}
One can thus subdivide $[0,1]$ into intervals $I_1,\dots,I_m$ with $m\leq {\lceil e^{\frac{\eta}{5} N} \rceil \leq e^{\frac{3\eta}{10}N}}$,
such that:
\begin{equation}\label{e.check-length}
\mathrm{Length}(g^{N}\circ\sigma'(I_j))<
  \varepsilon e^{-\frac{\eta}{10} n_{i+1}}\|D_xg^{n_{i+1}}|_E\|\quad \text{for }(x,E)\in \hg^{-n_i}\circ\hsigma'([0,1]).
  \end{equation}
We can focus on the intervals $I_j$ such that $I_j\cap T'\ne\emptyset$.
Fixing such an $I_j$, there exists $\hx=(x,E)\in (\hg^{-n_i}\circ\hsigma')(I_j)$,
such that $(\hx,\hg(\hx),\dots,\widehat g^{{n_{i+1}}-1}(\hx))$ is $\frac \eta {10}$-neutral.  In particular, $\|D_xg^{n_{i+1}}|_E\|\leq e^{\frac{\eta}{10}n_{i+1}}$, whence
 $$
  \mathrm{Length}(g^{N}\circ\sigma'(I_j))< \eps e^{-\frac{\eta}{10}n_{i+1}}e^{\frac{\eta}{10} n_{i+1}}=\eps.
 $$
On the other hand since $\hsigma'(T')\subset \hW_\#$, eq.~\eqref{e.control-diameter} implies:
 $$
   \diam(\widehat g^{N}\circ\widehat \sigma' (T'\cap I_j))<\heps.
   $$
We have shown that the image $\hg^N\circ \hsigma'(T'\cap I_j)$ is contained in a $(\eps,\heps)$-ball.
We can apply Yomdin's Theorem~\ref{t-yomdin} and obtain a family $\cR_j$  of reparametrizations $\varphi$ of $\sigma'$ over $T'\cap I_j$
with cardinality at most $\Upsilon\|D\hg^N\|_\supnorm^{1/(r-1)}$ such that each curve $g^N\circ \sigma'\circ \varphi$ has
$C^r$ size at most $(\eps,\heps)$. Consequently, $\cR_j$ is $(C^r,g,N,\eps,\heps)$-admissible up to time $N$.

The union  $   \cR_\psi:=\bigcup_j \cR_j  $ is thus a family of reparametrizations $\varphi$ of $\sigma'$ over $T'$ which is
$(C^r,g,N,\eps,\heps)$-admissible up to time $N$.
By~\eqref{e.check-length} they satisfy the bound (iii') of the induction scheme (Step 7) on the length of $\sigma'\circ \varphi$.
Combining the bounds on $m$ and $|\cR_j|$, one bounds the cardinality of $\cR_\psi$ by
   $$
      |\cR_\psi|<  \Upsilon e^{{\frac{3\eta}{10}} N} \|D\hg^N\|_\supnorm^{1/(r-1)}.
   $$
   which by~\eqref{e.choice-gamma}, \eqref{delta-sharp} and $N>\frac 1 \gamma$, is bounded by $\exp(\frac{\lambda(\hf)}{r-1} N+{\frac{6\eta}{10}}N)$ as required.

\subsection*{Step 10 (Case (c)):}  In this case $n_{i+1}-n_i=1$. 
By our assumptions on $\sigma'$, we have $\text{Length}(\hsigma')<\heps$ and $\text{Length}(\sigma')<\varepsilon e^{-\tfrac \eta {10} n_i} \|Dg_x^{n_i}|_E\|$ for each $(x,E)\in \hg^{-n_i}\circ\hsigma' [0,1]$.
One can thus decompose $[0,1]$ into intervals $I_1,\dots, I_m$
with $m\leq \lceil \|D\hg\|_\supnorm \rceil+\lceil e^{\frac\eta{10}}\|D g\|_\supnorm \rceil-1$ such that:
\begin{itemize}
\item[(a)] $\text{Length}(g\circ \sigma'(I_j))<\varepsilon e^{-\tfrac \eta {10} n_{i+1}} \|Dg_x^{n_{i+1}}|_E\|$ for each $(x,E)\in \hg^{-n_i}\circ\hsigma' [0,1]$,
\item[(b)] $\text{Length}(\hg\circ \hsigma'(I_j))<\heps$.
\end{itemize}
In particular, for each $1\leq j\leq m$, the image $\hg\circ \hsigma'(I_j)$ is contained in a $(\eps,\heps)$-ball.
We can thus apply Yomdin's Theorem \ref{t-yomdin} and obtain a family $\cR_j$  of reparametrizations  of $\sigma'$ over $I_j$
which is $(C^r,g,1,\eps,\heps)$-admissible up to time $1$ and with cardinality at most $\Upsilon\|D\hg\|_\supnorm^{1/(r-1)}$.

The union  $   \cR_\psi:=\bigcup_j \cR_j  $ is thus a family of reparametrizations $\varphi$ of $\sigma'$ over $[0,1]$ which is
$(C^r,g,N,\eps,\heps)$-admissible up to time $1$.
By property (a) above they satisfy the bound (iii') of the induction scheme (Step 7) on the length of $\sigma'\circ \varphi$.
Note that $m\leq 3 e^{\frac{\eta}{10}}\|D\hg\|_\supnorm^{1/(r-1)}$, hence
$\cR_\psi$ has cardinality smaller than
   $$
      |\cR_\psi|< 3\Upsilon e^{\frac{\eta}{10}}\|D\hg\|_\supnorm^{1+1/(r-1)},
   $$
   which by~\eqref{e.choice-gamma} and \eqref{delta-sharp}, is bounded by $\exp(\eta/\gamma)$.

\subsection*{Step 11 (Completion of the proof)}
Steps 7--10 provide the construction of the family $\cR_{n}^\theta$ for each type $\theta$.
The inductive bounds (ii) for $|\cR_{n_{i}}^\theta|/|{\cR_{n_{i-1}}^\theta}|$ imply
$$|\cR_n^\theta|\leq \exp\bigg(\frac{\lambda(\widehat f)}{r-1}n+\frac {7\eta} {10} n +\frac{\eta}{\gamma}A_c(\theta) \bigg),$$
where $A_c(\theta)$ is the number of times $n_i$ belongs to case (c) for the type $\theta$.

Let $\cR_n$ denote the union of $\cR_n^\theta$ over all possible types $\theta$. This is a family of reparametrizations of $\sigma$ over $T=\bigcup_\theta T_\theta$, which is
$(C^r,g,N,\eps,\heps)$-admissible up to time $n$.

Since $A_c(\theta)\leq 2\gamma n$ for all $\theta$ (by Claim~\ref{c.decomp}), since  the number of types $\theta$ is bounded by $\exp(H(4\gamma)n)$
(by Claim~\ref{c.bound-types}), and since $H(4\gamma)<\frac{\eta}{10}$ (by our choice of $\gamma$, see~\eqref{e.choice-gamma}), this gives
\begin{equation*}
\begin{split}
|\cR_n|&\leq \exp\bigg(\frac{\lambda(\widehat f)}{r-1}n+\frac {7\eta} {10} n +\frac{\eta}{10\gamma}\cdot 2\gamma n + H(4\gamma)n \bigg)\\
&\leq \exp\bigg(\frac{\lambda(\widehat f)}{r-1}n+\frac {7\eta} {10} n +\frac{\eta}{5}n + \frac \eta {10} n \bigg)= \exp\bigg(\frac{\lambda(\widehat f)}{r-1}n+\eta n \bigg).
\end{split}
\end{equation*}
This concludes the proof of Proposition~\ref{p.estimate2'}. \qed

\section{The neutral decomposition}
\label{s.neutral}
Let $f$ be a homeomorphism on a compact metric space $X$.
We denote the point mass measure at $x\in X$ by $\delta_x$. Given   ${\mathfrak N}\subset\NN$, let  $$
      \mu^{\mathfrak N}_{x,n} := \frac1n \sum_{j\in [0,n)\cap {\mathfrak N}} \delta_{f^j(x)}.
 $$
The \weakstar\  limit points of $(\mu^{\mathfrak N}_{x,n})_{n\geq 1}$ are called  the \emph{${\mathfrak N}$-empirical measures} of $x$.

\begin{definition}\label{d.neutralblock}
Suppose   $\varphi\colon X\to \RR$ is continuous, $\alpha>0$ and $L\geq 1$.
An \emph{$(\alpha,L)$-neutral block} of $(x,f,\varphi)$ is an interval of integers $(n_0,n_0+1,\ldots,n_1-1)$ s.t.
\begin{itemize}
\item $n_1-n_0\geq L$, and
\item $\varphi(f^{n_0}(x))+\varphi(f^{n_0+1}(x))+\dots+\varphi(f^{n-1}(x))\leq \alpha\cdot(n-n_0)$ for all $n_0< n \leq n_1$.
\end{itemize}
We denote by $\Neutral_{\alpha,L}(x,f,\varphi)$ the union of the $(\alpha,L)$-neutral blocks of $(x,f,\varphi)$.
\end{definition}
Any interval of integers which is a union of two neutral blocks is still a neutral block.
Therefore, if $\liminf\limits_{n\to\infty}\frac1{n}\sum_{k=0}^{n-1}\varphi(f^k(x)), \liminf\limits_{n\to\infty}\frac{1}{n}\sum_{k=-n}^{-1}{\varphi}(f^{k}(x))>\alpha$, then  $\Neutral_{\alpha,L}(x,f,\varphi)$ is a disjoint union of (finite)  maximal neutral blocks.

\begin{proposition}\label{p.neutral-block}
Let  $f,f_1,f_2,\dots$ be homeomorphisms of a compact metric space $X$,
and let $\varphi,\varphi_1,\varphi_2,\dots$ be continuous functions on $X$ such that $f_k\to f$ and $\varphi_k\to \varphi$ uniformly.
For each $k$, let $\nu_k$ be an ergodic probability for $f_k$ such that $\int\varphi_k d\nu_k\geq 0$.
Then there exist a subsequence $(\nu_{k_i})$ and (positive) measures $m_0,m_1$ such that:
\begin{enumerate}[(i)]
\item Both $m_0$ and $m_1$ are $f$-invariant.
\item\label{item2} The subsequence $(\nu_{k_i})$ converges \weakstar\  to $m_0+m_1$.
\item\label{item3} For any neighborhoods $V=(V_0,V_1)$ of $m_0,m_1$, if $\alpha<\alpha_*(V)$, $L\geq L_*(V)$, $i\geq i_*(V,\alpha,L)$,  then for $\nu_{k_i}$-a.e. $x$,
the $\Neutral_{\alpha,L}(x,f_{k_i},\varphi_{k_i})$-empirical measures belong to $V_0$, and
the $(\NN\setminus \Neutral_{\alpha,L}(x,f_{k_i},\varphi_{k_i}))$-empirical measures belong to $V_1$.
\item $\int \varphi dm_0 = 0$.
\item For $m_1$-almost every point $x$, the limit of
$\frac 1 n \sum_{j=0}^{n-1}\varphi(f^j(x))$ is positive.
\end{enumerate}
\end{proposition}

\begin{remark}
The measures $m_0,m_1$ are not normalized, but there are $f$-invariant probabilities
$\mu_0,\mu_1$ such that $m_0=(1-\beta) \mu_0$ and $m_1=\beta\mu_1$, where $\beta=m_1(X)$.
\end{remark}

\noindent
{\bf Example 1.} The following {constructions} show that the decomposition $m_0+m_1$  depends on the sequence $(\nu_k)$ and not just on its limit.

 Let $X=\{-2,1,2\}^\ZZ$, $f_k=f=$ the left shift, and $\varphi_k(x):=\varphi(x)=x_0$. For each $k\geq1$, consider the periodic sequence $p^{(k)}$ with period
 $$
    \underbrace{-2,\dots,-2}_{k}, \underbrace{+2,\dots,+2}_{k}, \underbrace{+1,\dots,+1}_{k}.
 $$
Let $\nu_k$ be the unique shift invariant probability measure on the orbit of $p^{(k)}$. Let $\delta_s:=$ the probability measure concentrated on $(\cdots s,s,s\cdots)$. It is easy to see that $\nu_k$ converges to $\mu:=\frac13(\delta_{-2}+\delta_{+1}+\delta_{+2})$

If $0<\alpha<1$, $L\geq 1$, and $k\geq L/2$, then the maximal $(\alpha,L)$-neutral blocks of $(p^{(k)},f_k,\vf_k)$ are $[n_0,n_1)\cap\ZZ$ where
$$
(p^{(k)}_{n_0},\ldots,p^{(k)}_{n_1-1})=(    \underbrace{-2,\dots,-2}_{k}, \underbrace{+2,\dots,+2}_{k}, {\underbrace{+1,\dots,+1}_{\ell}}),
$$
{with $\ell=\lfloor \tfrac{2\alpha}{1-\alpha}k\rfloor$.}
So $
   m_0=\frac13(\delta_{-2}+\delta_{+2})\text{ and }m_1=\frac13 \delta_{+1}.
 $

Now consider the measures $\nu_k'$ obtained from the periodic sequence $q^{(k)}$ with period
$
    \underbrace{-2,\dots,-2}_{k},  \underbrace{+1,\dots,+1}_{k}, \underbrace{+2,\dots,+2}_{k}.
 $
These measures also converge to $\mu$. But now, if $0<\alpha<1$, $L\geq 1$ and $k>L/2$, then the maximal neutral blocks of $(q^{(k)},f_k,\vf_k)$
 are $[n_0,n_1)\cap\ZZ$ where
$$(q^{(k)}_{n_0},\ldots,q^{(k)}_{n_1-1})=
    (\underbrace{-2,\dots,-2}_{k},  \underbrace{+1,\dots,+1}_{k}, \underbrace{+2,\dots,+2}_{\ell}),
$$
with {$\ell=\lfloor \frac{2\alpha+1}{2-\alpha} k\rfloor$.}
So
 $
   m_0=\frac13\delta_{-2}+\frac13\delta_{+1}+\frac16\delta_{+2}\text{ and }m_1=\frac16 \delta_{+2}
 $.

\medskip
\noindent
{\bf Example 2:} Suppose $(f_k)$ converges to $f$ in $\Diff^r(M)$, and $\nu_k$ are $f_k$-invariant measures which converge to an $f$-invariant measure $\mu$.
Assume the limiting measure $\mu:=\lim\nu_k$ is  ergodic and hyperbolic of saddle type. Let $\lambda^-_\mu<0, \lambda^+_\mu$ be the Lyapunov exponents of $\mu$.

Consider the unstable lifts $\hnu_k^u$ to the fibered bundle. Passing to a subsequence, we may assume that $\hnu_k^u$ converge \weakstar\  to a limit $\hmu$  (a lift of $\mu$).
Let $\vf_k:\hM\to\R$ be
$
\vf_k(x,E):=\log\|(Df_k)_x|_E\|.
$
We apply Proposition \ref{p.neutral-block} to $\hf_k,\vf_k,\hnu_k^u$, {obtaining} a decomposition
$\hmu=m_0+m_1=(1-\beta)\hmu_0+\beta\hmu_1$. 
On the other hand, since $\mu$ is ergodic and hyperbolic,
$$\hmu=a\hmu^-+(1-a)\hmu^+$$  where $\hmu^{\pm}$ are the unique lifts of $\mu$ to $\mathrm{graph}(E^{\pm})$. So $\hmu_0,\hmu_1\ll \hmu^++\hmu^-$, whence by the ergodicity of $\hmu^{\pm}$,  $\hmu_0,\hmu_1$ are convex combinations of $\hmu^+,\hmu^-$. By (v), $\hmu_1$ has no $\hmu^-$ component, so $\hmu_1=\hmu^+$. By (iv), if $\hmu_0=b\hmu^-+(1-b)\hmu^+$ then necessarily
$
0=\int\vf d\hmu_0=b\lambda^-_\mu+(1-b)\lambda^+_\mu,
$
whence $b=\lambda^+_\mu/(\lambda^+_\mu+|\lambda^-_\mu|)$. It follows that
  $$\begin{aligned}
   &{\widehat{\mu}_1}=\hmu^+,\quad
   {\widehat{\mu}_0}=\frac{\lambda^+_\mu\hmu^- + |\lambda^-_\mu|\hmu^+}{\lambda^+_\mu + |\lambda^-_\mu|}.
  \end{aligned}$$

To finish the calculation of  $m_0=(1-\beta)\hmu_0,m_1=\beta\hmu_1$ it remains to determine $\beta$.  To do this, we substitute the formulas for $\hmu_0,\hmu_1$ in the identity $a\hmu^-+(1-a)\hmu^+=\hmu=(1-\beta)\hmu_0+\beta\hmu_1$.
Since $\hmu^{\pm}$ are ergodic, the coefficient of $\hmu^{-}$ on both sides must be equal, which leads to
$
\beta=1-a\frac{\lambda^+_\mu+|\lambda^-_\mu|}{\lambda^+_\mu}.
$
Looking at case (2) of Theorem~\ref{thm-defect-cont}, we recognize  that
$\displaystyle\beta=\frac{1}{\lambda^+(f,\mu)} \lim_{k\to\infty}\lambda^+(f_k,\nu_k).$

\begin{proof}[Proof of Proposition~\ref{p.neutral-block}]
Without loss of generality,  $(\nu_k)$ converges \weakstar\ to an $f$-invariant probability measure $\mu$ (otherwise pass to a suitable subsequence).
We abbreviate  $\Neutral^k_{\alpha,L}(x):=\Neutral_{\alpha,L}(x,f_k,\varphi_k)\subset\ZZ$ and define
 \begin{equation}\label{N_alpha_L}
     N_{\alpha,L}^k := \{x\in X: 0\in\Neutral^k_{\alpha,L}(x)\}.
 \end{equation}
This is a measurable set. We call it  the {\em $(\alpha,L)$-neutral set}
of $(\varphi_k,f_k)$.
Let $\chi^k_{\alpha,L}$ denote the indicator function of  $N^k_{\alpha,L}$.
Since $\nu_k$ is ergodic, for $\nu_k$-a.e. point $x$,  $$
    \lim_{n\to\infty} \mu_{x,n}^{\Neutral^k_{\alpha,L}(x)} = \chi^k_{\alpha,L}\nu_k\text{ in the \weakstar\  topology.}$$

\begin{claim}\label{claim-subseq}
There  exists an increasing sequence of integers $(k_i)$ such that
 $$
    \forall(\alpha,L)\in(0,1]\times\NN\quad  \lim_{ i\to\infty} \chi^{k_i}_{\alpha,L}\nu_{k_i} \text{ exists}
 $$
in the \weakstar\  topology. We write these limits as $m_{\alpha,L}$.
\end{claim}

\begin{proof}[Proof of the claim]
Fix some  countable dense set $ E\subset(0,1]$. By compactness and a diagonal argument, there is an increasing sequence $k_i\to\infty$ such that the following limits exist in the \weakstar\  topology:
 $$
    \forall (\alpha,L)\in E\times\NN\quad  \lim_{p\to\infty} \chi^{k_i}_{\alpha,L}\nu_{k_i} = m_{\alpha,L}.
 $$

Let us check that this can be extended to all $(\alpha,L)\in(0,1]\times\NN$, maybe after passing to a subsequence. Indeed,  select a  countable family $(u_j)$ of nonnegative continuous functions which generate a countable dense algebra over $\mathbb Q$ in $C^0(M)$ (the space of continuous real-valued functions on $M$ with the supremum norm). Fix $L$ and $u_j$. The function $\alpha\in E\mapsto m_{\alpha,L}(u_j)$ is non-decreasing on $E$, and therefore extends uniquely to a left-continuous function $\alpha\in[0,1]\to\widetilde m_{\alpha,L}(u_j)$. The discontinuity points form a countable set $D_{L,j}$. Again by monotonicity with respect to $\alpha$,
 $$
   m_{\alpha,L}(u_j)=\widetilde m_{\alpha,L}(u_j)=\lim_i \chi^{k_i}_{\alpha,L}\nu_{k_i}(u_j)
 $$
at every $\alpha\in[0,1)\setminus D_{L,j}$. By a further extraction of a subsequence, we ensure that $\chi_{\alpha,L}\nu_{k_i}$ converge for all $(\alpha,L)$ in the countable set $\bigcup_{L\in\NN,j\geq1} D_{L,j}$.
\end{proof}

We return to the proof of Proposition 5.2. To simplify notation, from now on $(\nu_k)$ will denote the subsequence $(\nu_{k_i})$.

The following \weakstar\  limit exists by monotonicity:
 $$
    m_0:= \lim_{\tiny\begin{array}{c} \alpha\to0 \\ L\to\infty
     \end{array}} m_{\alpha,L} = \inf_{\alpha>0,L\geq1} m_{\alpha,L}.
 $$
We set $m_1:=\mu-m_0$. Since $0\leq \chi^k_{\alpha,L}\nu_k\leq \nu_k$, it follows that $0\leq m_0\leq\mu$ so that both $m_0$ and $m_1$ are positive measures.

\medskip

Neutral blocks have length {at least} $L$, therefore
for every continuous function $u$, we have $(\chi^k_{\alpha,L}\nu_k)(u-u\circ f_k)\leq (2/L)\|u\|_{\sup}$. Since $u\circ f-u\circ f_k\to 0$ uniformly, $m_{\alpha,L}(u-u\circ f)\leq (2/L)\|u\|_{\sup}$. It follows that $m_0$ is $f$-invariant. So is  $m_1=\mu-m_0$. This proves items (i) and (ii).

Item (iii) is a simple consequence of the construction.
\medskip

We turn to (iv). For any function $\psi\colon X\to \RR$, we define
 $$
    S_n^k\psi:= \sum_{j=0}^{n-1} \psi\circ f_k^j \text{ and } S_n^\infty\psi:= \sum_{j=0}^{n-1} \psi\circ f^j.
 $$
For every $x$, we decompose $\Neutral^k_{\alpha,L}(x)\cap [0,\infty)$ into maximal disjoint intervals:
$$
   \Neutral^k_{\alpha,L}(x)\cap[0,\infty)=\bigsqcup_{{i\geq1}}[{a}_i,{a}_i+{b}_i).
 $$
Since $\nu_k$ is ergodic, for $\nu_k$-a.e. $x$,
$$\begin{aligned}
&(\chi^k_{\alpha,L}\nu_k)(\varphi_k) = \int \chi^k_{\alpha,L}\varphi_k \, d\nu_k
    =\lim_{n\to\infty}\frac1nS^k_n(\chi^k_{\alpha,L}\varphi_k)(x)\\
& = \lim_{{j}\to\infty}\frac1{{a}_{j}+{b}_{j}}\sum_{{i\leq j}} \left(S^k_{{b}_i}\varphi_k\right)(f_k^{{a}_i}(x)).\\
\end{aligned}$$
Each interval $[{a}_i,{a}_i+{b}_i)$ is a maximal $(\alpha,L)$-neutral block except possibly the initial one, if it contains $0$. The first block contributes {${C_0}(x)/n\to0$} to the limit. The other blocks are all  maximal neutral blocks, and satisfy the bounds
 $$
      \alpha ({b}_i+1) - \varphi_k(f_k^{{a}_i+{b}_i}(x)) <  (S^k_{{b}_i}\varphi_k)( f_k^{{a}_i}(x)) \leq \alpha {b}_i.
 $$
The first inequality comes from the maximality of the block, the second is the definition of neutrality.
{Summing over $i=1,\dots,j$, we obtain the bounds
 $$
   {C_0(x)+}\alpha \big(\underset{{1<i}\leq j}{\textstyle \sum} {b}_i + j\big) - j \sup_{x,k}\varphi_k(x) <(S^k_{{b}_j+{a}_j}\chi_{\alpha,L}{\varphi_k})(x) \leq {C_0(x)+}\alpha \underset{{1<}i\leq j}{\textstyle \sum} {b}_i.
  $$
  }
 Since each such complete neutral block  has length at least $L$, there are {$j\leq n/L$ maximal $(\alpha,L)$-blocks in $[0,{a}_j+{b}_j)$. Dividing by ${a}_j+{b}_j\geq\sum_{i\leq j}{b}_i$  and discarding some nonnegative terms from the lower bound, we obtain in the limit $j\to\infty$,}
 $$
{-\sup_{x,k}\varphi_k(x)/L }< (\chi_{\alpha,L}\nu_k)(\varphi_k) \leq \alpha.
 $$
Passing to the limit $\alpha\to 0, L\to\infty$ {and recalling that ${\varphi_k}\to\varphi$ uniformly,} we obtain item (iv): $m_0(\varphi)=0$.
\medskip

We prove item (v) by contradiction, assuming that
 $$
    \gamma:= \frac12 m_1(\{x:\lim_{n\to\infty} (1/n)(S_n^\infty\varphi)(x)\leq 0\})>0.
 $$
There are $\alpha_0>0,L_0<\infty$ such that, for $0<\alpha\leq\alpha_0$, $L\geq L_0$,
 \begin{equation}\label{eq-m1-conv}
       \left| m_{\alpha,L}(X)-m_0(X) \right| < \frac\gamma{100}.
 \end{equation}

Given $K\geq 0$, let
 $$\begin{aligned}
   &V_0^k(K) := \{x\in M | \exists 0\leq a\leq K \text{ s.t. }
   [-a,0]\cap\ZZ \text{ is $(\alpha_0,L_0)$-neutral for $(x,f_k,\vf_k)$}
   \},\\
   &W^\infty_0(K) := \{x\in M | \exists 0\leq a\leq K \text{ s.t. }
   [-a,0]\cap\ZZ \text{ is $(\alpha_0/2,L_0)$-neutral for $(x,f,\vf)$}
   \}.
 \end{aligned}$$
These are closed sets, and $W^\infty_0(K)\subset V^k_0(K)$ for all large $k$. We can ensure that $\alpha_0/2$ is not a member of  the countable set of $\alpha$'s such that
 $$
  \mu\left(\left\{x\in X:\exists n\geq1\;(1/n)S^\infty_n\varphi(x)=\alpha\right\}\right)>0.
 $$
In particular, $\mu(\partial W^\infty_0(K))=0$ for each integer $K\geq1$. Therefore, for any $K\geq1$,
 \begin{equation}\label{eq-m1-VK}
  \lim_{\alpha,L}\lim_k \nu_k(W^\infty_0(K)\setminus N^k_{\alpha,L})=
   \lim_{\alpha,L}\lim_k \left[(1-\chi_{\alpha,L}^k)\nu_k\right](W^\infty_0(K))=
  m_1(W^\infty_0(K)).
 \end{equation}

By the ergodic theorem, for $\mu$-a.e.\ $x\in M$ such that $\lim_{n\to\infty}(1/n)S_n^\infty\varphi(x)\leq0$, it is also the case that
$
\lim_{n\to\infty}\frac{1}{n}\sum_{j={-n{+1}}}^0 \vf(f^j(x))\leq 0.
$
For such $x$, the Pliss lemma {\cite{Pliss}, \cite[Ch. IV.11]{Mane-Book}} yields arbitrarily large integers $a\geq0$ such that $[-a,0]$ is $\alpha_0/2$-neutral. In particular, fixing $K_0\geq L_0$ large enough,
 $$m_1( W^\infty_0(K_0)) > \gamma.
  $$
  Hence there exist $0<\alpha_1\leq\alpha_0$ and $L_1>\max(200 K_0/\gamma,L_0)$ such that
for all $k$ large enough,
 \begin{equation}\label{eq-m1-VK1}
   \nu_k( V_0^k(K_0)\setminus N^k_{\alpha_1,L_1}) \geq \nu_k( W^\infty_0(K_0)\setminus  N^k_{\alpha_1,L_1})> \gamma.
 \end{equation}

Since $\nu_k$ is ergodic, for $\nu_k$-a.e. $x$, the set $\Neutral_1:=\Neutral^k_{\alpha_1,L_1}(x)$ of visits under iterations of $f_k$ to the $(\alpha_1,L_1)$-neutral set $N^k_{\alpha_1,L_1}$ has density
$$ d(\Neutral_1):=\lim_{n\to\infty} \frac{|\Neutral_1\cap[0,n-1]|}n
      = \nu_k(N^k_{\alpha_1,L_1})\xrightarrow[k\to\infty]{}m_{\alpha_1,L_1}(X).$$
So by eq. \eqref{eq-m1-conv},  for $k$ large enough and  $\nu_k$-a.e. $x$,
\begin{equation}\label{N_1}
     d(\Neutral_1)> m_0(X)-\gamma/100.
\end{equation}
Similarly, the set $\Neutral_0:=\Neutral^k_{\alpha_0,L_0}(x)$  has density
  \begin{equation}\label{eq-m0}
    d(\Neutral_0) =\nu_k(N^k_{\alpha_0,L_0})< m_0(X)+\gamma/100.
  \end{equation}
 for $k$ large enough and  $\nu_k$-a.e. $x$. Finally, let ${\frak V}$ denotes the set of $j$ such that  $f^j(x)\in V_0^k(K_0)\setminus N^k_{\alpha_1,L_1}$. Then  by eq.~\eqref{eq-m1-VK1},  for $k$ large enough and  $\nu_k$-a.e. $x$,
\begin{equation}\label{frak-V}
   d({\frak V}) = \nu_k(V_0^k(K_0)\setminus N_{\alpha_1,L_1}^k)>\gamma.
\end{equation}
We will show that \eqref{N_1}--\eqref{frak-V} lead to a contradiction.

By definition {of $V_0^k(K_0)$}, each $j\in \frak V$ is the last element of an $(\alpha_0,L_0)$-neutral block $I(j)$ with length  $\leq K_0$ (we do not claim that this block is maximal). Let
$$
\mathfrak I:=\bigcup\{I(j):j\in\mathfrak V, I(j)\cap\Neutral_1=\emptyset\}\ ,\ \mathfrak I':=\bigcup \{I(j):j\in\mathfrak V, I(j)\cap\Neutral_1\neq \emptyset\}.
$$

We claim that the {upper} asymptotic density {$\overline{d}(\mathfrak I'):=\limsup \frac{1}{n}|\mathfrak I'\cap [0,n)|$} is less than $\gamma/100$. To see this note that if $j\in\mathfrak V$ and  $I(j)\cap\Neutral_1\neq \emptyset$, then $j\not\in\Neutral_1$ (by definition of $\frak V$)
and since $L_0<L_1$, $I(j)$ contains the last element of a maximal sub-interval of $\Neutral_1$. The interval with length $2K_0$ centered at this last element must contain $I(j)$.
Since the number of maximal sub-intervals of $\Neutral_1\cap [0,n]$ is bounded by $n/{L_1}$,
the {upper} asymptotic  density of $\mathfrak I'$ is no more than $2K_0/L_1<\gamma/100$.

It follows that the {upper} asymptotic density of $\mathfrak I$ is at least
 $$
   \overline{d}(\frak I) \geq {d}(\frak V) - \frac\gamma{100} > \frac{99}{100}\gamma,
 $$
and since  $\Neutral_1$ and $\frak I$ are disjoint
 $
   \overline{d}(\Neutral_1\cup \frak I) = d(\Neutral_1)+\overline{d}(\frak I) > m_0(X)+\frac{98}{100}\gamma.
 $

But  $\mathfrak N_1$ and $\mathfrak I$  are a union of $(\alpha_0,L_0)$-neutral blocks, so $\mathfrak N_1\cup\mathfrak I\subset \mathfrak N_0$, whence by  eq.~\eqref{eq-m0},
$\overline{d}(\mathfrak N_1\cup\mathfrak I)<m_0(X)+\gamma/100$. This contradiction proves item (v).
\end{proof}

\section{Proof of the main theorem }\label{s.conclude}

We recall the notation $\hlambda(\hf,\hmu):=\int_{\hM} \log\|Df_x|_E\| d\hmu(x,E)$.
In this section, we prove  the following stronger version of Theorem~\ref{t.ergodic3}.

\begin{maintheorem}\label{t.ergodic4}
Fix a real number $r>2$.
For every $k\geq 1$, let  $f_k\in\Diff^r(M)$ and
let $\nu_k$ be an $f_k$-ergodic measure. Let $\hnu_k$ be an $\hf_k$-ergodic lift satisfying
$\hlambda(\hf_k,\hnu_k)=\lambda^+(f_k,\nu_k)$ such that:
\begin{itemize}
  \item[--] the limits $\lim_k\lambda^+(f_k,\nu_k)$ and $\lim_k h(f_k,\nu_k)$ exist and $\lim_k\lambda^+(f_k,\nu_k)\geq 0$,
  \item[--] $f_k\overset{\scriptscriptstyle r-\text{bd}}\longrightarrow f$ for some $f\in\Diff^r(M)$ (i.e. $f_k\to f$ uniformly and $\sup_k \|f_k\|_{C^r}<\infty$),
  \item[--] $\widehat \nu_k\overset{w^*}\to \widehat \mu$ for some $\widehat f$-invariant probability measure $\widehat \mu$ on $\widehat M$, perhaps non-ergodic.
 \end{itemize}
Then there exist $\beta\in [0,1]$, two
$f$-invariant measures $\mu_0,\mu_1$
with $\widehat f$-invariant lifts $\widehat \mu_0,\widehat \mu_1$ s.t.

\begin{equation}\label{e.eq0}
\widehat \mu=(1-\beta)\widehat \mu_0+\beta\widehat \mu_1,
\end{equation}
\begin{equation}\label{e.eq1}\lim_{k\to\infty} \lambda^+(f_k,\nu_k) =  \beta\lambda^+(f,\mu_1),
\end{equation}
\begin{equation}\label{e.eq2}\lim_{k\to\infty} h(f_k,\nu_k)-\tfrac1{r-1}\lambda(\hf)\leq \beta h(f,\mu_1).
\end{equation}
Moreover:
\begin{itemize}
  \item[--] If $\beta>0$, then $\hlambda(\hf,\hmu_1)=\lambda^+(f,\mu_1)$ and $\lambda^+(f,x)>0$ for $\mu_1$-a.e. $x$.
  \item[--] If $\beta<1$, then $\hlambda(\hf,\hmu_0)=0$.
\end{itemize}
\end{maintheorem}
Note that when $\nu_k$ is hyperbolic, the measure $\hnu_k$ above is simply the unstable lift $\hnu_k^+$.

\subsection{Reductions}\label{s.Reduction-Ptv-Ent}
We assume the setting of Theorem~\ref{t.ergodic4}.
There is no loss of generality in assuming  that $r$ is finite, since the $C^\infty$ case follows from the $C^r$ case by letting $r\to\infty$.
By Lemma~\ref{l-ArzelaAscoli} and since   $f_k\rbd{r} f$ with $r>2$,
 \begin{equation}\label{eq-conv}
   f_k\to f \text{ in the $C^2$ topology and } \hf_k\to\hf \text{ in the $C^1$ topology}.
 \end{equation}

Let $h:=\lim\limits_{k\to\infty} h(f_k,\nu_k)$.
By Ruelle's inequality and \eqref{eq-conv}, $h$ is bounded by $\sup_k\|Df_k\|_\supnorm<\infty$.
It is clearly non-negative. The theorem has a simple proof when $h=0$:

\begin{proof}[Proof of Theorem~\ref{t.ergodic4} when $h=0$] In this case eq.~\eqref{e.eq2} is trivial.
Since $\log\|Dg(x)|_E\|$ depends continuously on $(x,E)$ and $g$, one gets
$
0\leq \lim_k \lambda^+(f_k,\nu_k)=\lim_k \lambda(f_k,\hnu_k)=\lambda(f,\hmu)
$.

If all ergodic components $\hmu'$ of $\hmu$ satisfy $\hlambda(\hf,\hmu')>0$,
it is enough to take $\hmu_0=\hmu_1=\hmu$ and fix $\beta=1$.
If some ergodic components of $\hmu$ satisfy $\hlambda(\hf,\hmu)\leq 0$ and since $\hlambda(\hf,\hmu)>0$,
one can decomposes $\hmu=(1-\beta)\hmu_1+\beta\hmu_0$ where $\hmu_0,\hmu_1$ are two $\hf$-invariant measures
such that $\hlambda(\hf,\hmu_0)=0$ and all ergodic components $\hmu'$ of $\hmu_1$ satisfy $\hlambda(\hf,\hmu')>0$.
\end{proof}

\noindent
Henceforth, we assume that
\medskip
\begin{center}
{\em $h:=\lim\limits_{k\to\infty}h(f_k,\nu_k)>0$, and $h(f_k,\nu_k)>0$ for all $k$.}
\end{center}
In particular, each measure $\nu_k$ is hyperbolic, i.e. has one positive and one negative Lyapunov exponent.
Note that it is enough to prove the theorem for any convenient further subsequence.

\subsection{\bf The decomposition of the limiting measure}\label{s.proofA''.subsequence}
Theorem~\ref{t.ergodic4} is stated in terms of the properties of a special decomposition $\mu=(1-\beta)\mu_0+\beta\mu_1$ of the $\mu=\lim\nu_k$. In this section we construct $\beta,\mu_0$ and $\mu_1$.

The idea is to apply Proposition~\ref{p.neutral-block} to a suitable sequence of measures.
By Ruelle's inequality and the reduction to the case $h(f_k,\nu_k)>0$, $\nu_k$ must be
$f_k$-hyperbolic of saddle type.
Let $\hnu_k^+$ denote the unstable lift of $\nu_k$ to $\hM$, and let $\hf_k, \hf$ be the lifts of $f_k, f$ to $\hM$. Define $\vf_k,\vf:\hM\to\RR$ by
$$
\vf_k(x,E):=\log\|Df_k|_E\|\ ,\vf(x,E):=\log\|Df|_E\|.
$$
We apply Proposition \ref{p.neutral-block} to $\hM$, $\hf_k$, $\vf_k$, $\hnu_k^+$.
(The proposition is applicable,  because  {by eq.~\eqref{eq-conv}}, $\hf_k\to\hf$ in $\Diff^1(\hM)$  and $\vf_k\to\vf$ uniformly on $\hM$,    and because  by Lemma~\ref{Lemma-Lyap} and Ruelle's inequality,
$\int\vf_k d\hnu_k=\lambda^+(f_k,\nu_k)\geq h(f_k,\nu_k)>0.$)

Proposition \ref{p.neutral-block} gives us a subsequence $\{k_i\}$ and two finite positive measures $\hm_0,\hm_1$ with the following properties.
\begin{enumerate}[(i)]
\item $\hm_i\circ\hf^{-1}=\hm_i$.
\item $\hnu_{k_i}^+\xrightarrow[i\to\infty]{w^\ast}\hm_0+\hm_1=:\hmu$. The limit $\hmu$ is $\hf$--invariant and lifts $\mu$.
\item Suppose $\hV_0,\hV_1$ are \weakstar\  open sets of measures such that  $\hV_i\owns\hm_i$, then  there are $\alpha_\ast(\hV_0,\hV_1)\in (0,1)$ and $L_\ast(\hV_0,\hV_1)\geq 1$ as follows. If $0<\alpha<\alpha_\ast$ and  $L>L_\ast$, then for all $k_i>k_\ast(\hV_0,\hV_1,\alpha,L)$, for $\hnu_{k_i}$--a.e. $\hx\in\hM$,
    \begin{enumerate}[--]
    \item the $\Neutral_{\alpha,L}(\hx,\hf_{k_i},\vf_{k_i})$--empirical measures of $\hx$ belong to $\hV_0$,
    \item the $\NN\setminus \Neutral_{\alpha,L}(\hx,\hf_{k_i},\vf_{k_i})$--empirical measures of $\hx$ belong to $\hV_1$.
    \end{enumerate}
\item\label{i-hmu0} $\int\vf d\hm_0=0$.
\item For $\hm_1$--a.e. point $\hx$, $\lim_{n\to\infty}\frac{1}{n}\sum_{j=0}^{n-1}\vf(\hf^j(\hx))>0$.
\item $h(f_{k_i},\nu_{k_i})\xrightarrow[i\to\infty]{} h>0$ and $h(f_{k_i},\nu_{k_i})>0$ for all $i$.
\item $\lambda^+(f_{k_i},\nu_{k_i})\xrightarrow[i\to\infty]{}\int \vf d\hmu$.
\end{enumerate}
Parts (i)--(v) are in Proposition \ref{p.neutral-block}; Part (vi) is the reduction in \S\ref{s.Reduction-Ptv-Ent}; and Part (vii) is because  $\lambda^+(f_{k_i},\nu_{k_i})=\int \vf_{k_i} d\hnu_{k_i}^+$ (by Lemma~\ref{Lemma-Lyap}), $\vf_{k_i}\to\vf$ uniformly, and $\hnu_{k_i}^+\to\hmu$ \weakstar\ .
To keep the notation as simple as possible we will henceforth  assume without loss of generality that
$
\{\nu_k\}=\{\nu_{k_i}\}
$.

\medskip
Let
\begin{align*}
\beta&:=1-\hm_0(\hM)=\hm_1(\hM).\\
\hmu_i&:=\frac{1}{\hm_i(\hM)}\hm_i,\text{ or {any invariant probability measure} if $\hm_i(\hM)=0$.}\\
\mu_i&:=\text{the projections of $\hmu_i$ to the corresponding $f$-invariant measures on $M$}.
\end{align*}
Notice that $\hmu=(1-\beta)\hmu_0+\beta\hmu_1$,  $\mu=(1-\beta)\mu_0+\beta\mu_1$, and $0\leq \beta\leq 1$.

\begin{claim}\label{c.non-one}
If $\beta<1$, then $\hmu_0(\vf) = \hm_0(\vf)/(1-\beta)=0$.
\end{claim}
\begin{proof} This is property \eqref{i-hmu0}.\end{proof}

\begin{claim}\label{c.non-zero}
$\beta\neq 0$. Consequently, $\hm_1(\hM)\neq 0$, and  $\mu_1$ is a probability measure.
\end{claim}
\begin{proof}
Assume by contradiction that $\beta=0$. Then  $\hmu=\hm_0$, and
\begin{align*}
0&=\int\vf d\hm_0=
\int\vf d\hmu,\ \ \ \text{ by (iv)}\\
&=\lim_{k\to\infty}\lambda^+(f_k,\nu_k),\ \ \ \text{ by (vii) and the assumption that $\{k_i\}=\{k\}$}\\
&\geq \lim_{k\to\infty}h(f_k,\nu_k)>0,\ \ \ \text{ by Ruelle's inequality and (vi)}.
\end{align*}
This contradiction shows that  $\beta\neq 0$.
\end{proof}

\begin{claim}\label{c.sign}
 $\mu_1$--a.e. $x$ has one positive and one non-positive Lyapunov exponent.
\end{claim}
\begin{proof}

By (v), the definition of $\vf$, and Fubini's Theorem,  for $\mu_1$--a.e. $x\in M$, there exists a one-dimensional subspace $E\subset T_x M$ such that
$$
\lim_{n\to+\infty}\frac{1}{n}\log\|Df^n|_E\|=\lim_{n\to+\infty}\frac{1}{n}\sum_{j=0}^{n-1}
(\vf\circ\hf^j)(x,E)>0.
$$
It follows that $\mu_1$--a.e. $x$ has  at least one positive Lyapunov exponent.

Assume by contradiction that the claim is false, then  there is an $f$-invariant set $\Omega$ of positive $\mu_1$-measure such that every $x\in \Omega$ has two (possibly equal) positive Lyapunov exponents. Recall the following well-known fact:

\medskip
\noindent
{\sc Fact:}
{\em If $\dim(M)=2$, then $\mu_1$-almost every $x$ with two positive Lyapunov exponents has an open neighborhood $U_x$ such that $\underset{\scriptscriptstyle n\to+\infty}\lim\frac{1}{n}\log\|Df^{-n}_y\|<0$ for all $y\in U_x$.}
{\begin{proof}
Pesin's local stable manifold theorem~\cite[Thm 2.2.1]{Pesin-Izvestia-1976} implies that $\mu_1$-almost every point $x$ admits a neighborhood $U_x$ and constants $C>0$ and $\kappa\in (0,1)$
such that for any $y\in U_x$ and $n\geq 0$,
$$
  d(f^{-n}(x),f^{-n}(y))\leq C\kappa^n.
  $$
Since the orbit of $x$ is  recurrent, this implies that the forward orbit of $x$ converges towards a periodic orbit $O$ {and in fact must coincide with that periodic orbit, again by recurrence.}
As a consequence, $x$ is a hyperbolic sink so
$\|Df^{-N}_x\|<\tfrac 1 2$ for some $N\geq1$ such that $f^N(x)=x$.  \end{proof}}

\medskip

\noindent

This fact enables us to build an open set $U$  such that
$$\mu_1(U)>0\ , \mu_1(\partial U)=\mu_0(\partial U)=0\ , \
\lim_{n\to+\infty}\frac{1}{n}\log\|Df^{-n}_y\|<0\text{ for all }y\in U.
$$
By  (ii), $\nu_k\xrightarrow[k\to\infty]{w^\ast}(1-\beta)\mu_0+\beta\mu_1$, and $\beta\neq 0$ by claim~\ref{c.non-zero}. So $\nu_k(U)>0$ for all $k$ large enough. But this is a contradiction, because the $\nu_k$ have one negative Lyapunov exponent, so that
$
\underset{n\to+\infty}\lim\frac{1}{n}\log\|Df^{-n}_y\|>0 \text{ $\nu_k$-almost everywhere.}
$
\end{proof}

\begin{claim}\label{c.lift-measure} $\hmu_1$ is the unstable lift $\hmu_1^+$ of $\mu_1$.
\end{claim}
\begin{proof}
By Claim~\ref{c.sign} {and the Oseledets theorem,}  for $\mu_1$-a.e. $x$, $T_x M=E^u(x)\oplus E^c(x)$, where $Df_x E^\ast (x)=E^\ast(f(x))$, $(\ast=u,c)$.
By Corollary~\ref{cor-nonergodiclift}, $\mu_1$ has a unique lift $\hmu_1^+$ to $\mathrm{graph}(E^u)$, and all other lifts charge some part of $\mathrm{graph}(E^c)$. Since $\lim (1/n)\sum_{j=0}^{n-1}\vf(\hf^j(\hx))\leq 0$ on $\mathrm{graph}(E^c)$, property (v) forces  $\hmu_1=\hmu_1^+$ a.e.
\end{proof}

\subsection{\bf Proof of Theorem~\ref{t.ergodic4} part (1)}
We compare the exponents of $\nu_k$ and $\mu$.
\begin{align*}
\beta\lambda^+(f,\mu_1)&=\beta\int\vf d\hmu^+_1 \text{ by Lemma~\ref{Lemma-Lyap}}\\
&=\beta\int\vf d\hmu_1=\int\vf d\hm_1 \text{ by  claim~\ref{c.lift-measure} and definition of $\hm_1$}\\
&=\int\vf d\hm_1+\int\vf d\hm_0=\int\vf d\hmu \text{ by  (iv) and definition of $\hmu$}\\
&=\lim_{k\to\infty}\lambda^+(f_k,\nu_k) \text{ by  (vii) and the convention $\{\nu_{k_i}\}=\{\nu_k\}$.}
\end{align*}
Thus $\lim_{k\to\infty}\lambda^+(f_k,\nu_k)=\beta\lambda^+(f,\mu_1)$, as required.

\subsection{Proof of  Theorem~\ref{t.ergodic4} part (2).}
{We now come to the heart of the proof of Theorem~\ref{t.ergodic4}.}
\subsection*{Step 1 (The decomposition  $\mu_1=\sum a_c\mu_{1,c}$)} We decompose $\mu_1$ into invariant measures $\mu_{1,c}$ all of whose ergodic components have  nearly the same entropy.

Let $\Upsilon:=\Upsilon(r)$ be as in Yomdin's Theorem~\ref{t-yomdin}, and let $\gamma_0(r,f,\eta)>0$ be as in Proposition~\ref{p.estimate2'}.
{We fix  $\eta,\gamma>0$ arbitrarily small, and $\ell\in\NN$ arbitrarily large as follows.
First we choose $\eta>0$; then we take an integer $\ell$ such that
\begin{equation}\label{e.ell}
\ell\eta>\overline{h}(f,\mu_1);
\end{equation}
and then we then fix $0<\gamma<\min(\gamma_0(r,f,r),1/20)$ such that for each $k\ge1$,
\begin{equation}\label{e.gamma}
10\gamma\bigg(\log(2\Upsilon)+\frac{r}{r-1}\log\|D\hf_k\|_\supnorm\bigg)<\eta,
\end{equation}
\begin{equation}\label{e.gamma2}
(\ell+10)\gamma<1, \; \; \text{ and } \;\;  \overline{h}(f,\mu_1)\ell\gamma<\eta.
\end{equation}

{By \eqref{e.ell}}, we can decompose
$
\displaystyle{\mu_1=\sum_{c=1}^{\ell'} a_c \mu_{1,c}\ ,\ a_c\in {(}0,1]\ ,\ \sum_{c=1}^{\ell'} a_c=1}
$
where $\ell'\leq \ell$ and   $\mu_{1,c}$ are $f$-invariant probability measures such that:
\begin{itemize}
\item[--] For $c\neq c'$, the measures $\mu_{1,c}$ and $\mu_{1,c'}$ are mutually singular;
\item[--] For each $c$ there is a number $h_c$ such that all the ergodic components of $\mu_{1,c}$ have entropy in $[h_c,h_c+\eta)$;
\item[--] $h(f,\mu_1)-\eta<\sum_c a_ch_c \leq h(f,\mu_1)$.
\end{itemize}

Since $\mu_{1,c}\ll\mu_1$,
Claim \ref{c.sign} implies that $\mu_{1,c}$-a.e. $x$ has one positive and one non-negative Lyapunov exponent.
By Corollary~\ref{cor-nonergodiclift}, $\mu_{1,c}$ has an unstable lift $\hmu_{1,c}^+$ carried by $\mathrm{graph}(E^u)$.
By Claim~\ref{c.lift-measure}, this gives
$$
\hmu_1=\hmu_1^+=\sum_{c=1}^{\ell'} a_c\hmu_{1,c}^+.
$$

\subsection*{Step 2 (The neutral segment parameters $\hU_0,\overline{n}_0$)}\label{ss.choice-n0}
\begin{enumerate}[$-$]
\item Let $N_1(r,f,\eta)$ and $N_0(r,f,\eta,\gamma)$ be as in Propositions~\ref{p.estimate1'} and~\ref{p.estimate2'}.

\item Fix $N$ larger than $N_1(r,f,\eta)$ and $N_0(r,f,\eta,\gamma)$.

\item Let $Q:={\sup_{k\geq1} Q_{r,N}(f_k)+1}$, with $Q_{r,N}(\cdot)$ defined as in \eqref{Q-r-n}. This supremum is finite because, for any $n=1,\dots,N$, any $k\geq1$,
 $$
     \|\hf_k^n\|_{C^{r-1}}\leq A(\|f_k^n\|_{C^r}\cdot\|Df_k^{-n}\|_\supnorm)^A
      \leq B(\|f_k\|_{C^r} \cdot \|Df_k^{-1}\|_\supnorm^n)^B
 $$
for some $A=A(r)$ by Lemma~\ref{l-Ds-hg}, and $B=B(r,n)$ by the formulas for the differential of a composition. Since $f_k\rbd{r} f$, the factors $\|f_k\|_{C^r}$, $k\geq1$, are bounded. Since $f_k\to f$ in $\Diff^1(M)$, $\|Df_k^{-1}\|_\supnorm$ converges to $\|Df^{-1}\|_\supnorm$ and is therefore bounded too.

\item Let  ${\eps_Y:=\eps_Y}(r,Q)$ be as in Yomdin's Theorem~\ref{t-yomdin}. We also set $\delta:=\varepsilon_Y$.
\end{enumerate}
With these choices of $\eta,\gamma,\delta,N$, we apply
Proposition~\ref{p.estimate2'}  to $\widehat \mu_0$ and obtain:
\begin{enumerate}[--]
    \item some numbers $0<\varepsilon,\widehat \varepsilon<\delta$,
    \item  an integer $\overline{n}_0$,
    \item a $C^2$ neighborhood $\mathcal U_0$ of $f$ in $\Diff^r(M)$,
    \item and an open set $\widehat U_0$ satisfying $\widehat \mu_0(\widehat U_0)>1-\gamma^2$ and $\widehat \mu_0(\partial \widehat U_0)=0$,
\end{enumerate}
such that property (**) holds.

{By further reducing $\hU_0$, we can also ensure} that
\begin{equation}\label{l.boundary-measure}
\widehat \mu(\partial \widehat U_0)=0.
\end{equation}

\subsection*{Step 3 (Expanding segment parameters $\hU_{1,c},n_{1,c}$)}
Having fixed $N,\varepsilon,\widehat \varepsilon$ as above,
Pro\-po\-sition~\ref{p.estimate1'} now associates to each $\widehat \mu_{1,c}$, with $1\leq c\leq \ell'$, an integer $\bar n_{1,c}:=\bar n_{1}(f,\hmu_{1,c},\eta,\gamma,N,\eps,\heps)$. We then introduce the integers
$$
n_1:=\max(\{\bar{n}_{1,c}:1\leq c\leq \ell'\}\cup\{1/\gamma\}),$$
$$
n_{1,c}:=n_1+c.
$$
We construct open sets $\hU_{1,c}$ and an integer $k_0$ with the following properties:
\begin{enumerate}[(a)]
\item $\widehat \mu_{1,c}(\widehat U_{1,c})>1-\gamma^2$
and $\widehat \mu(\partial \widehat U_{1,c})=0$.
\item For all $f_k$ with  $k>k_0$, for any regular curve $\sigma$ with $C^r$ size at most $(\eps,\heps)$, there exists a family of reparametrizations $\cR$ of $\sigma$  over $\hsigma^{-1}(\hU_{1,c})$ such that
    \begin{enumerate}[\;\;(b1)]
    \item $\cR$ is $(C^r,f_k,N,\eps,\heps)$--admissible up to time $n_{1,c}:=n_1+c$,
    \item $\frac{1}{n_1+c}\log |\cR|\leq h_c+\frac{\lambda(\hf_k)}{r-1}+\eta$.
    \end{enumerate}
\item For all $f_k$ with $k>k_0$, for any different $1\leq c,c'\leq \ell$, and for any $0\leq j\leq n_{1,c}$,
    $$
    \hf^j_k(\text{closure}(\hU_{1,c}))\cap \text{closure}(\hU_{1,c'})=\emptyset.
    $$
\item If $c\neq c'$, then $\hmu_{1,c}(\hU_{1,c'})<\gamma^2$.
\end{enumerate}

\medskip
\noindent
{\em Construction\/.}
For each $c$, we apply
Proposition \ref{p.estimate1'} to $f$, $\mu_{1,c}$ and to the parameters  $\eta,\gamma, \eps,\heps, N$ and $n=n_{1,c}$. This gives an open set $\hU_{1,c}$ s.t.
\begin{enumerate}[(a')]
\item $\hmu_{1,c}(\hU_{1,c})>1-\gamma^2$ and $\hmu_{1,c}(\partial\hU_{1,c})=0$.
\item For all $g$ sufficiently close to $f$ in $C^2$-topology such that $Q_{r,N}(g)<{Q}$
and for any regular $C^r$ curve $\sigma$ with $C^r$ size at most $(\eps,\heps)$, there exists a family of reparametrizations $\cR$ of $\sigma$ over $\hsigma^{-1}(\hU_{1,c})$
satisfying (b1) and (b2).
\end{enumerate}
Choose $k_0'$ so that (b') holds for all $g=f_k$ with  $k>k_0'$, for all $1\leq c\leq \ell$.

By assumption, the measures $\hmu_{1,c}$ (for $1\leq c\leq \ell'$) are mutually singular and there exist
pairwise disjoint $f$-invariant measurable sets $X_{c}$ such that $\hmu_{1,c}(X_{c'})$ equals one when $c=c'$, and zero otherwise.
Using (a'), one constructs compact sets $\hK_{1,c}\subset X_c\cap\hU_{1,c}$ such that $\hmu_{1,c}(\hK_{1,c})>1-\gamma^2$. Necessarily
$\hf^j(\hK_{1,c})\cap \hK_{1,c'}\subset X_{c}\cap X_{c'}=\emptyset$
for all different $1\leq c,c'\leq \ell'$ and every $0\leq j\leq n_{1,c}$. So
$$
\min\bigg\{\dist(\hf^j(\hK_{1,c}),\hK_{1,c'}):1\leq c,c'\leq \ell',\; c\neq c',\; j=1,\ldots,n_{1,c}\bigg\}>0.
$$
{This inequality remains true if one replaces $\hf$ by $\hf_k$ with $k$ large enough  and  the compact sets $\hK_{1,c}$ by small enough neighborhoods $\hU'_{1,c}$.
We may choose those neighborhoods so that $\hmu(\partial\hU'_{1,c})=0$.  Replacing each $\hU_{1,c}$ by its intersection with $\hU_{1,c}'$, we obtain sets satisfying  both the conclusion~(*) of Proposition~\ref{p.estimate1'} and:
$$
 \hf^j_k(\hU''_{1,c})\cap \hU''_{1,c'}=\emptyset
 \qquad ( 1\leq c\ne c'\leq\ell', \quad \forall 0\leq j \leq n_{1,c})
$$
for $k$ large enough. We replace the sets $\hU_{1,c}$ by these new $\hU''_{1,c}$}. 
Moreover (a') and (b') are preserved.

Notice that if $c\neq c'$, then $\hU_{1,c'}\cap \hU_{1,c}=\emptyset$, so $\hmu_{1,c}(\hU_{1,c'})<\hmu_{1,c}(\hM\setminus \hK_{1,c})<\gamma^2$.
Then, all the properties (a)-(d) hold.

\subsection*{Step 4 (\weakstar\  neighborhoods of $\hm_0$ and $\hm_1$)}\label{ss.choice-V} We construct \weakstar\  open neighborhoods $\hV_0,\hV_1$ of the measures $\hm_0,\hm_1$ s.t.
\begin{equation}\label{V-bounds}
\begin{aligned}
&\hm\in \hV_0\Rightarrow \begin{cases} |\hm(\hU_0)-\hm_0(\hU_0)|<\gamma^2,\\
|\hm(\hM)-\hm_0(\hM)|<\gamma^2,
\end{cases}\\
&\hm\in \hV_1\Rightarrow \begin{cases} |\hm(\hU_{1,c})-\hm_1(\hU_{1,c})|<\gamma^2, & 1\leq c\leq\ell',\\
|\hm(\hM)-\hm_1(\hM)|<\gamma^2.
\end{cases}
\end{aligned}
\end{equation}
Such neighborhoods exist since $\hm_i(\partial \hU_{1,c})=0$ and $\hm_i(\partial\hU_0)=0$ from the property (a) in step 3 and~\eqref{l.boundary-measure}.

\subsection*{Step 5 (Neutral block parameters $\alpha,L,k_\ast$)}\label{ss.choiceL}
Recall the integer $k_0$ obtained in step 3.
Using   property (iii) in \S\ref{s.proofA''.subsequence}, one finds {$\alpha\in(0,\eta/10)$, $L\geq 1$ satisfying
\begin{equation}\label{e.L}
L> {2\gamma^{-1}}\max\{\overline{n}_0, \; n_{1,1},\dots,n_{1,\ell'}\},
\end{equation}
and $k_\ast=k_\ast(\alpha,L)\geq k_0$ such that  for $k>k_\ast$ and $\hnu_k^+$-a.e. $\hx\in\hM$,
    \begin{enumerate}[--]
    \item the $\Neutral_{\alpha,L}(\hx,\hf_k,\vf_k)$-empirical measures are in a compact subset of $\hV_0$;
    \item the $\NN\setminus\Neutral_{\alpha,L}(\hx,\hf_k,\vf_k)$-empirical measures are in a compact subset of $\hV_1$.
    \end{enumerate}
These compact sets will give the extra margin necessary to deal with boundary terms
{(see the proof of Lemma~\ref{l.decomposition} below).}

\subsection*{Step 6 (Decomposition of orbits into orbit segments)}
Recall that the orbit segment of $\hf_k$ with \emph{length} $t$ and \emph{initial point} $\hx$ is the string $(\hx,\hf_k(\hx),\ldots,\hf_k^{t-1}(\hx))$.
It is associated with the measure $$
\hmu_{\hf_k,\hx}^t:=\frac{1}{t}\sum_{j=0}^{t-1}\delta_{\hf_k^j(\hx)}.
$$
An orbit segment will be called {\em neutral} if  $(0,1,\dots,t-1)$ is an $(\alpha,L)$-neutral block of $(\hx,\hf_k,\vf_k)$ as defined in Section~\ref{s.neutral}, i.e. if  $t\geq L$
and if  $\hx=(x,E)$ satisfies:
 \begin{equation}\label{eq-neutral}
   \|Df_k^m|_{E}\|\leq e^{\alpha m}\ \ \text{ for all $0< m\leq t$}.
 \end{equation}
Using the open sets $\hU_0,\hU_{1,c}$ and the integers $n_{1,c}$ defined at steps 2 and 3, we introduce $\ell'+2$ classes of orbit segments $(\hx,\hf_k(\hx),\ldots,\hf_k^{t-1}(\hx))$:
\begin{enumerate}[(a)]
\item {\bf Segments with color $1\leq c\leq  \ell'$:} orbit segments such that $\hx\in\hU_{1,c}$ and  $t=n_{1,c}$.
\item {\bf Blank segments:} neutral orbit segments  such that
    $\hmu^{t}_{\hf_k,\hx}(\hU_0)\geq 1-\gamma$.
\item {\bf Fillers:} orbit segments with length $t=1$.
\end{enumerate}
{The class of an orbit segment as above can be recognized from its length $t$: If $t=1$, it is a filler, if $t\in[n_1+1,n_1+\ell']$, it is colored with color $t-n_1$, and if $t$  is larger than $L$, then it is blank, see \eqref{e.L}.}
So these $\ell'+2$ classes are disjoint.

\begin{lemma}\label{l.decomposition}
For all $k>k_\ast$ and for $\widehat \nu_k^+$-a.e. $\widehat x$, there exists $n_k(\widehat x)\in\NN$ such that
all the orbit segments $(\widehat x, \widehat f_k(\widehat x),\dots,\widehat f^{n-1}_k(\widehat x))$ with $n\geq n_k(\widehat x)$ can be decomposed into:
\begin{enumerate}[(a)]
\item colored segments of total length at most  $\beta a_c n+\gamma n$, for each color $c$,
\item blank segments of total length at least $(1-\beta) n-4\gamma n$,
\item fillers of total length at most $6\gamma n$.
\end{enumerate}
\end{lemma}

\begin{proof}
By the reduction in section~\ref{s.Reduction-Ptv-Ent}, the ergodic measures $\nu_k$ have positive entropy, and therefore
the  $\hnu_k$ measure of $\hf_k$-periodic points is zero. Thus it is sufficient to consider non-periodic $\hx$ only.
Orbit segments of non-periodic points can be identified with the non-ordered sets of points they contain without any loss of information, because there is only one way to order them to get an orbit segment. We will therefore feel free to
abuse terminology and  treat orbit segments as  sets,  subject to the usual set-theoretic operations.

\medskip
{Given an orbit segment $\vartheta:=(\hx,\dots,\hf^{n-1}(\hx))$, we are going to build a decomposition
 \begin{equation}\label{eqDec}
    (\underbrace{\hf^{t_0}(\hx),\dots,\hf^{t_1-1}(\hx)}_{\vartheta_1};
   \dots;
   \underbrace{\hf^{t_{m-1}}(\hx),\dots,\hf^{t_m-1}(\hx)}_{\vartheta_m})
  \end{equation}
where $t_0=0<t_1<\dots<t_m=n$ and each segment $\vartheta_i:=(\hf_k^{t_i-1}(\hx),\dots, \hf_k^{t_{i}-1}(\hx))$ is either a colored segment, a blank segment, or a filler.

We call the sequence $(t_0,\dots,t_m)$ the \emph{type of the decomposition} since it determines not only how the orbit segment is divided but to which class each segment belong.
}

{By analogy with Section~\ref{s.neutral},} a neutral sub-segment of $\vartheta$ is called {\em maximal}, if it does not lie in a strictly longer neutral sub-segment of $\vartheta$. Let
$
{\mathcal{S}_{neut}}(\vartheta)
$
denote the collection of all maximal neutral sub-segments of $\vartheta$.
It is not difficult to see that every neutral sub-segment of $\vartheta$ is contained in some element of ${\mathcal{S}_{neut}}(\vartheta)$, and that the segments in ${\mathcal{S}_{neut}}(\vartheta)$ are pairwise disjoint.

\medskip
\noindent
{\sc Decomposition.}
We define $t_i$ inductively beginning with $t_0:=0$. Assuming that $0\leq t_{i-1}<n$ has been defined, we consider the following three possibilities:

\begin{enumerate}[--]
\item {\em Case (a)\/}. There exists $1\leq c\leq \ell'$ such that
 $\hf^{t_{i-1}}(\hx)\in\hU_{1,c}$, the orbit segment 
$(\widehat f_k^{t_{i-1}}(\widehat x),\dots,$ $\widehat f^{t_{i-1}+n_{1,c}-1}_k(\widehat x))$
does not intersect any segment in $\mathcal S_{neut}(\vartheta)$,
 and $t_{i-1}+n_{1,c}\leq n$.
 The color $c$ is uniquely defined because the $\hU_{1,c}$ are disjoint. We set $t_{i}:=t_{i-1}+n_{1,c}$. The resulting orbit segment $\vartheta_i:=(\widehat f_k^{t_{i-1}}(\widehat x),\dots,\widehat f^{t_{i}-1}_k(\widehat x))$ with length $n_{1,c}$ is a segment with color $c$.

 \medskip

\item {\em Case (b)\/}. There exists $T$ such that $\vartheta_i:=(\widehat f_k^{t_{i-1}}(\widehat x),\dots,\widehat f^{t_{i-1}+T-1}_k(\widehat x))\in \mathcal S_{neut}(\vartheta)$, and $|{\vartheta_i}\cap\hU_0|>(1-\gamma)|{\vartheta_i}|$.
The integer $T$ is unique {by maximality and}, by the definition of $\mathcal S_{neut}(\vartheta)$, it satisfies $T\geq L$ and $t_{i-1}+T\leq n$. 
We set $t_{i}:=t_{i-1}+T$.
Then $\vartheta_i=(\widehat f_k^{t_{i-1}}(\widehat x),\dots,\widehat f^{t_{i}-1}_k(\widehat x))$, and $\vartheta_i$ is a blank segment with length $T$.

\medskip
\item {\em Case (c)\/}. There are no such $T$ or $c$. In this case we set $t_{i}:=t_{i-1}+1$, and $\vartheta_i:=(\hf_k^{t_{i-1}}(\hx))$. This is a filler.
\end{enumerate}
These cases are mutually exclusive and at least one of them must happen (case (b) implies that $\vartheta_i\in\mathcal S_{neut}(\vartheta)$, excluding case (a), and case (c) happens iff case (a) and case (b) both fail), and   in all cases, $t_{i}\le n$ since $t_{i-1}<n$. Thus we have defined $t_{i+1}\in (t_i,n]$ unambiguously.

The inductive process stops with $t_m=n$.
The result is a decomposition of $\vartheta$ {as in eq.~\eqref{eqDec}} into blank segments, colored segments, and fillers.

\medskip

\noindent
{\sc Size estimates.} \medskip
 We now fix a $\hnu_k$-typical $\hx$, a large $n$, set $\vartheta:=(\hx,\ldots,\hf^{n-1}(\hx))$, and estimate the total size of the fillers, blank segments, and the segments of given color in $\vartheta$. ``Typical" means that our estimates apply to a set of full $\hnu_k$-measure, and the ``largeness" of $n$ is allowed to depend on $\hx$.

Let $\textsf{Neut}$ be the union of all   neutral sub-segments of $\vartheta$ and $\textsf{Neut}^c$ its complement:
$$\textsf{Neut}:=\bigcup\mathcal S_{neut}(\vartheta), \quad
\textsf{Neut}^c:=\vartheta\setminus \textsf{Neut}.$$
Clearly $\textsf{Neut}\subseteq\mathfrak N_{\alpha,L}(\hx,\hf_k,\vf_k)\cap [0,n)$, but the sets  could be  different, because the neutral segments in $\mathfrak N_{\alpha,L}(\hx,\hf_k,\vf_k)$ which contains $0$ or $n$ may have a non-$(\alpha,L)$-neutral intersection with $[0,n)$.   However, it is not difficult to see that
$
\frac{1}{n}\bigl|\textsf{Neut}\triangle \bigl(\mathfrak N_{\alpha,L}(\hx,\hf_k,\vf_k)\cap [0,n)\bigr)\bigr|\xrightarrow[n\to\infty]{}0.
$
Therefore, for $\hnu_k$-a.e. $\hx$ there exists $n_k(\hx)$ such that for all $n>n_k(\hx)$ and  $k>k_\ast$ (cf. step 5)
$$\widehat m_0':=\frac 1 n \sum_{\widehat y \in \textsf{Neut} } \delta_{\widehat y} \in \widehat V_0, \quad\quad
\widehat m'_1:=\frac 1 n \sum_{\widehat y \in \textsf{Neut}^c} \delta_{\widehat y} \in \widehat V_1.$$

Recall that $\hm_1=\beta\sum_c a_c \hmu_{1,c}$ with $\sum_c a_c=1$ and $0<\beta\leq 1$.
Since $\hm_1'\in \hV_1$, eq.~\eqref{V-bounds} and items~(a) and (c)
of Step~3 imply:
 \begin{align}
  \hm_1'(\hM\setminus\bigcup_{c'}\hU_{1,c'}) &= \hm_1'(\hM) - \sum_{c'} \hm_1'(\hU_{1,c'})
   < \hm_1(\hM)+\gamma^2 - \sum_{c'}  (\hm_1(\hU_{1,c'}) - \gamma^2)\notag \\
   &\leq  \beta\left( 1  - \sum_{c'} a_{c'}(1-\gamma^2)\right)+(1+\ell')\gamma^2\notag \\
   &\le (\ell+2)\gamma^2, \text{because the number of colors $\ell'$ is at most $\ell$.} \label{eq-m1compl}
 \end{align}

Let $\textsf{C}_c$ denote the union of all colored segments with color $c$;
let $\textsf{B}$ denote the union of all blank segments;  and let $\textsf{F}$ denote the union of all fillers.

\medskip
\noindent
{\it (a) Colored segments.}
By construction, if $c\neq c'$, then $\bigcup_{j=0}^{n_{1,c}}\hf_k^j(\hU_{1,c})\cap \hU_{1,c'}=\emptyset$. So if $\vartheta_i$ has color $c$, then
$\vartheta_i\subset\textsf{Neut}^c\setminus{\bigcup}_{c'\neq c}\hU_{1,c'}
$, and
\begin{align*}
|\textsf{C}_c|&\leq |\textsf{Neut}^c\setminus{\bigcup}_{c'\neq c}\hU_{1,c'}|{\leq}|\textsf{Neut}^c\setminus{\bigcup}_{c'}\hU_{1,c'}|
+|\textsf{Neut}^c\cap\hU_{1,c}|\\
&{\leq} n\cdot \hm_1'(\hM\setminus\bigcup_{c'}\hU_{1,c'})+n\cdot\hm_1'(\hU_{1,c})
<n(\ell+2)\gamma^2+n[\hm_1(\hU_{1,c})+\gamma^2],
\end{align*}
by eqs. \eqref{V-bounds}~and~\eqref{eq-m1compl}.
By step 3 (d), $\hm_1(\hU_{1,c})<\beta(a_c+\gamma^2)$.
Substituting this in the above and using~\eqref{e.gamma2} give
$$
|\textsf{C}_c|\leq n\beta a_c+n(\ell+ 4)\gamma^2\leq n\beta a_c+n\gamma.$$

\medskip
\noindent
{\it (b) Blank segments.} By definition, every blank segment is neutral, so $\textsf{B}\subset\textsf{Neut}$, and  $|\textsf{B}|=|\textsf{Neut}|-|\textsf{Neut}\setminus\textsf{B}|$.
By \eqref{V-bounds} and the definition of $\hm_0'$,
$$
|\textsf{Neut}|=n\cdot \hm_0'(\hM)\geq n(\hm_0(\hM)-\gamma^2)=n(1-\beta-\gamma^2).
$$
 The set $\textsf{Neut}\setminus \textsf{B}$ is the union of the maximal neutral orbit segments which visit $\hU_0$  with frequency less than {$1-\gamma$}. Thus
 $
    \gamma\cdot |\textsf{Neut}\setminus \textsf{B}| < \hm_0'(\hM\setminus\hU_0)n
 $. By \eqref{V-bounds} and the bound
$\widehat \mu_{0}(\widehat U_{0})>1-\gamma^2$ in step 2,
$$
\widehat m'_0(\widehat M\setminus \widehat U_0)<
\widehat m_0(\widehat M)-\widehat m_0(\widehat U_0)+2\gamma^2
<(1-\beta)-(1-\beta)(1-\gamma^2)+2\gamma^2=\gamma^2(3-\beta),
$$
so
$|\textsf{Neut}\setminus \textsf{B}|\leq \gamma^{-1}\widehat m'_0(\widehat M\setminus \widehat U_0)n<\gamma(3-\beta)n.$
It follows that $|\textsf{B}|>(1-\beta) n-4\gamma n.$

\medskip

\noindent
{\it (c) Fillers.}
By construction, a filler is a segment of length one $(\hy)$ such that one of the following holds:
 \begin{enumerate}[(i)]
   \item $\hy$ does not belong to a colored segment or to a segment in $\mathcal S_{neut}(\vartheta)$;
   \item $\hy$ belongs to a segment in $\mathcal S_{neut}(\vartheta)$, but this segment is not  a blank segment;
   \item $\hy$ belongs to a
   segment of length $n_{1,c}$ which begins at $\hU_{1,c}$, but it fails to be a colored segment because it extends beyond the right endpoint of $\vartheta$;
   \item $\hy$ belongs to a
   segment of length $n_{1,c}$ which begins at $\hU_{1,c}$, but it fails to be a colored segment because it intersects an element of $\mathcal S_{neut}(\vartheta)$.
 \end{enumerate}

{The fillers of type (i) belong to $\textsf{Neut}^c\setminus\bigcup_c\hU_{1,c}$, so their cardinality is bounded by eq.~\eqref{eq-m1compl}: 
$
|\textsf{Neut}^c\setminus\bigcup_c\hU_{1,c}|=n\cdot\hm_1'(\hM\setminus \bigcup_c\hU_{1,c})
 < (2+\ell)\gamma^2 n < \gamma n
 $

The fillers of type (ii) belong to $\textsf{Neut}\setminus \textsf{B}$. As we saw above this means that  their cardinality is less than  $ \gamma(3-\beta)n<3\gamma n.$

The  number of fillers of type (iii) is clearly bounded by the maximum length of a colored segment $\max_c n_{1,c} = n_1+\ell'\leq n_1+\ell$. This can be assumed to be  less than $\gamma n$ when $n$ is large enough.

It remains to  control the fillers of type (iv).
Fix $\vartheta_0\in\mathcal S_{neut}(\vartheta)$, and suppose  $\hy$ belongs to a  colored segment which intersects $\vartheta_0$.  All colored segments have lengths at most $n_1+\ell'$, therefore $\hy$ must belong to one of two segments of length $n_1+\ell'$ adjacent to the endpoints of $\vartheta_0$. This gives the following bound for the number of fillers of type (iv): $   2(n_1+\ell')\cdot|\mathcal S_{neut}(\vartheta)|$.
Recalling that $\mathcal S_{neut}(\vartheta)$ consists of disjoint sub-segments of $\vartheta$, each with length at least $L$, we find that
$$
|\mathcal S_{neut}(\vartheta)|\leq\frac{n}{L}.
$$
Thus by \eqref{e.L}, the number of fillers of type (iv) is at most
 $
\frac{2(n_1+\ell')}{L} n < \gamma n.
 $

It follows that  the total length of the fillers is $|\textsf{F}|<6\gamma n$.
}
\end{proof}

\subsection*{Step 7 (A bound on the number of decomposition types)}
In the previous step we decomposed  orbit segments of typical points with length $n$ large enough into colored segments, blank segments and fillers.

Let
$\theta=(t_0,t_1,\dots,t_m)$ be the type of the decomposition, see \eqref{eqDec} and the discussion which follows it. Here we bound the number of possible types. As always, let
$H(t):=t\log \tfrac 1 t+(1-t)\log\tfrac 1{1-t} \text{ for $0<t<1$}.$
\begin{claim}\label{c.number-type}
There exists $n_H:=n_H(\gamma)$ such that the number of types of decompositions of all  $\hf_k$--orbit segments {as in Lemma~\ref{l.decomposition}} with length $n>n_H$ and $k$ arbitrary is  at most $\exp[n H({10}\gamma)]$.
\end{claim}
\begin{proof}
A decomposition of  an orbit segment with length $n$ has
\begin{enumerate}[--]
\item  at most $\gamma n$ blank segments (because these have lengths $\geq L>1/\gamma$),
\item at most $\gamma n$ colored segments
(because these have lengths $\geq n_1+1>1/\gamma$),
\item and at most $6\gamma n$ fillers (by Lemma \ref{l.decomposition}).
\end{enumerate}
This gives a total of  at most $\lfloor 8\gamma n\rfloor$ segments.

So every type $\theta=(t_0,\ldots,t_m)$ has length $m< \lfloor 8\gamma n\rfloor+1$.
 Since $t_0=0, t_m=n$, there can be at most
$
\sum_{m=1}^{\lfloor 8\gamma n\rfloor}{n\choose m-1}
$ different types.
Since $8\gamma<1/2$,  the  sum is bounded by $8\gamma n {n\choose\lfloor 8\gamma n\rfloor}$.
By De Moivre's approximation, this is less than $\exp[n H(8\gamma)+o(n)]$ as $n\to\infty$.

The claim follows, because ${10}\gamma<1/2$ so
 $H({8}\gamma)<H({10}\gamma)$.
 \end{proof}

\subsection*{Step 8 (Conditional measures and choice of $N_k, F_k$)}
The measures $\nu_k$ are assumed to be $f_k$-ergodic, and  by  the reductions in section~\ref{s.Reduction-Ptv-Ent} they have positive entropy.   So by Ruelle's inequality,  each $\nu_k$ is a hyperbolic measure.

As explained in Section~\ref{ss.unstable-partitions}, one can introduce a measurable partition subordinated to the unstable lamination of $\nu_k$
and associate to it a system of conditionals probability measures $\nu^u_{k,x}$.

We fix $N_k\geq 1$ and a Borel set $F_k\subset M$ with $\nu_k(F_k)>\frac{1}{2}$ such that  for every point $x\in F_k$ and for the diffeomorphism $f_k$:
\begin{enumerate}[--]
\item $x$ has a well-defined unstable manifold, an immersed $C^r$ curve $W^u(x)\subset M$;
\item $\nu^u_{k,x}$ is well-defined and $x$ belongs to the support of the restriction of  $\nu^u_{k,x}$ to $F_k$;
\item $\hx:=(x,E^u(x))$ satisfies Lemma~\ref{l.decomposition} with $n_k(\hx)\leq N_k$. In particular for each $n\geq N_k$, the orbit segment  $(\hx,\hf_k(\hx),\ldots,\hf_k^{n-1}(\hx))$ has a decomposition {as in Lemma~\ref{l.decomposition}}.  Let $\theta=\theta(x,n)$ be the type of decomposition.
\end{enumerate}

\subsection*{Step 9 (Construction of reparametrizations)}
Choose a point $x\in F_k$ which satisfies Corollary~\ref{lemma entropy reparametrizations}. 

Let  $\sigma\colon [0,1]\to W^u(x)$ be a regular $C^r$-curve which  parametrizes a neighborhood of $x$ in $W^u(x)$ in the intrinsic topology, and which
has $C^r$ size at most $(\varepsilon,\heps)$.
{By the choice of $F_k$,}
$T:=\sigma^{-1}(F_k)$ has positive measure for $\nu^u_{k,x}$.

Fix $n\geq N_k$, and let $\eps,\heps$ and $N$ be as in step 2.
Our aim is to  construct a  particular family of reparametrizations $\mathcal R_n$ of $\sigma$ over $T$, which is $(C^r,f_k,N,\eps,\heps)$--admissible
up to time $n$. In later steps, we will estimate the cardinality of $\cR_n$ and use
Corollary~\ref{lemma entropy reparametrizations} to obtain the upper bound for $h(f_k,\nu_k)$ which completes the proof of the theorem.

We begin by fixing a type $\theta:= (t_0,t_1,\dots,t_m)$ with $t_m=n$, and  constructing  a family  of reparametrizations $\mathcal R^\theta_n$
of $\sigma$ admissible up to time $n$ over the set
$$T_\theta :=\sigma^{-1}\{y\in F_k \text{ with type $\theta$}\}.$$
Then we will take the union over all possible types and obtain the family $\cR_n$ of reparametrizations over $T$.
\medskip

$\mathcal R^\theta_n$  is obtained inductively by defining families
$\mathcal R^\theta_{t_i}$ of parametrizations of $\sigma$ over $T_\theta$, which are $(C^r,f_k,N,\eps,\heps)$--admissible up to time $t_i$.
The base of the induction is defined by taking $\mathcal R^\theta_{0}:=\{\operatorname{Id}\}$. This parametrization is admissible, because $\sigma$ has $C^r$ size at most
$(\varepsilon,\widehat \varepsilon)$.
After $m$ steps, we will obtain the family $\mathcal R^\theta_n=\mathcal R_{\theta,t_{m}}$, which is an admissible family of reparametrizations over $T_\theta$, up to time $n$.

\medskip
\noindent
{\sc Induction step:} We build $\mathcal R^\theta_{t_{i}}$,
 assuming $\mathcal R^\theta_{t_{i-1}}$ was already constructed.
 We proceed by concatenation (see Lemma~\ref{l.concatenation}). We will set
  $$
    \mathcal R^\theta_{t_{i-1}} := \{\psi\circ\vf:\psi\in\mathcal R^\theta_{t_{i-1}},\; \vf\in\mathcal R(\psi,t_{i-1})\}
 $$
for well-chosen families $\cR(\psi,t_{i-1})$ which parametrize of $f_k^{t_{i-1}}\circ\sigma\circ\psi$ over $\psi^{-1}(T_\theta)$ in a  $(C^r,f_k,N,\eps,\heps)$-admissible up to time $t_{i}-t_{i-1}$, and which we now construct.

Fix $\psi\in\cR_{t_{i-1}}^\theta$ and let $\sigma':=f^{t_{i-1}}_k\circ \sigma\circ \psi$ and $T':=\psi^{-1}({T_\theta})$. By the induction hypothesis, $\cR_{t_{i-1}}^\theta$ is admissible, therefore  $\sigma'$ has $C^r$ size at most $(\eps,\heps)$.

By the definition of $T_\theta$,  the orbit segments $(\hy,\hf_k(\hy),\ldots,\hf_k^{t_i-t_{i-1}}(\hy))$ have the same type for every $\widehat{y}\in \hsigma'(T')$: If $t_i-t_{i-1}=n_{1,c}$ they are all colored segments with color $c$; if $t_i-t_{i-1}\geq L$ they are all blank; and if $t_i-t_{i-1}=1$ they are all fillers. See step 6.
Our construction of $\cR(\psi,t_{i-1})$ depends on the case:
\medskip

\noindent
{\it Case (a): $t_{i}-t_{i-1}=n_{1,c}$.} In this case, $  (\hf_k^{t_{i-1}}(\hy),\ldots,\hf_k^{t_{i}-1}(\hy))$ are colored segments with the same color $c$ for all $\hy\in\hsigma(T_\theta)$.
 Thus  $\widehat f_k^{t_{i-1}}(\widehat \sigma (T_\theta))\subset\widehat U_{1,c}$.
Applying Proposition~\ref{p.estimate1'} to $\hmu_{1,c},n_{1,c}$ and
$\sigma'$, we obtain a family $\cR(\psi,t_{i-1})$ of reparametrizations $\varphi$ over the set $T'$ which is $(C^r,f_k,N,\eps,\heps)$--admissible up to time $t_{i}-t_{i-1}$, and which satisfies the  cardinality bound
$$\frac 1 {t_{i}-t_{i-1}} \log|\mathcal{R}(\psi,t_{i-1})|\leq \overline h(f,\mu_{1,c})+\frac{\lambda(\widehat f)}{r-1}+\eta
\leq h_{c}+\frac{\lambda(\hf)}{r-1}+ 2 \eta$$
(Recall that the entropy of every ergodic component of $\mu_{1,c}$ is in $[h_c,h_c+\eta)$.)

\medskip

\noindent
{\it Case (b): $t_{i}-t_{i-1}\geq L$.} {In this case $  (\hf_k^{t_{i-1}}(\hy),\ldots,\hf_k^{t_{i}-1}(\hy))$ are  blank for  all $\hy\in\hsigma(T_\theta)$. By the choice of $\alpha,L$ {in Step 5, these segments are $\tfrac\eta{10}$-neutral, and their lengths are larger or equal to $\overline{n}_0$.
Consequently, the set $T'$ is contained in the set {controlled by Proposition~\ref{p.estimate2'}, when  applied to the diffeomorphism $g=f_k$ and to the curve $\sigma'$  (see eq.~\eqref{e.T}).
Hence, this}
proposition gives us  a family $\cR(\psi,t_{i-1})$ of reparametrizations $\vf$
of $\sigma'$ over the set $T'$ which is $(C^r,f_k, N,\eps,\heps)$--admissible up to time $t_{i}-t_{i-1}$, and which satisfies the  cardinality bound
$$\frac 1 {t_{i}-t_{i-1}} \log|\mathcal{R}(\psi,t_{i-1})|\leq \frac{\lambda(\widehat f)}{r-1}+\eta.$$
\medskip

\noindent
{\it Case (c): $t_{i}-t_{i-1}=1$.} As $\sigma'$ has $C^r$ size at most $(\eps,\heps)$,
 $$
   r_{\hf_k}(\eps,\heps,1,\hf_k(\hsigma(T'))) \leq \| D\hf_k\|_\supnorm+1 \leq 2\| D\hf_k\|_\supnorm.
 $$
Since $\varepsilon,\widehat \varepsilon$ have been chosen smaller than $\varepsilon_Y$,
Corollary~\ref{coro.reparametrization} of Yomdin's Theorem applies and provides
a family $\cR(\psi,t_i)$ of reparametrizations $\vf$  over the set $T_\theta$
which are $(C^r,f_k, N,\eps,\heps)$--admissible up to time $t_{i}-t_{i-1}=1$, and which satisfy the  cardinality bound
$|\mathcal{R}(\psi,t_{i-1})| \leq \Upsilon\|D\hf_k\|_\supnorm^{1/(r-1)} \times 2 \|D\hf_k\|_\supnorm$, hence
$$\frac 1 {t_{i}-t_{i-1}} \log|\mathcal{R}(\psi,t_{i-1})|
   \leq \log(2\Upsilon)+
        \frac{r}{r-1}\log\|D\hf_k\|_\supnorm.
  $$

This completes the inductive step.

\subsection*{Step 10 (Cardinality of $\cR_n$) }
The families of reparametrizations obtained in step 9 satisfy
$|\mathcal{R}(\psi,t_i)|\leq \exp(\kappa_i(\theta)(t_{i}-t_{i-1})),$
where
\begin{align*}
\kappa_i(\theta)&:=\begin{cases}
h_{c}+(r-1)^{-1}\lambda(\hf)+2\eta & \text{if $t_{i}-t_{i-1}=n_{1,c}$},\quad \text{(case a)},\\
(r-1)^{-1}\lambda(\hf)+\eta & \text{if $t_{i}-t_{i-1}\geq L$},\quad \text{(case b)},\\
\log(2\Upsilon)+\tfrac{ r}{r-1}\log\|D\hf_k\|_\supnorm & \text{if $t_{i}-t_{i-1}=1$},\quad \text{(case c)}.
\end{cases}
\end{align*}
It follows that
$
|\mathcal R_n^\theta|\leq \exp\biggl(\sum_{i=1}^{m}\kappa_i(\theta)(t_{i}-t_{i-1})\biggr).
$

The total length of the blank segments is (trivially) less than $n$, and the total lengths of the segments with color $c$ and fillers is respectively, less than $\beta a_c n+\gamma n$ and $6\gamma n$, by Lemma \ref{l.decomposition}.
{Denoting the total length of colored segments with color $c$ (resp. blank segments, fillers) by $N_c$ (resp. $N_b$, $N_f$),} we find that
\begin{align*}
&\sum_{i=1}^{m}\kappa_i(\theta)(t_{i}-t_{i-1})\leq
\sum_{c} h_c N_c+ \frac{\lambda(\hf)}{r-1}\left(\sum_{c}N_c+N_b\right)+2\eta\sum_{c}N_c+\eta N_b\\
&+\left(\log(2\Upsilon)+\frac{ r}{r-1}\log\|D\hf_k\|_{\supnorm}\right) N_f.
\end{align*}
Using the trivial bounds $\sum_c N_c+N_b=n-N_f\leq n$, $N_b\leq n$, and the bounds $N_c\leq \beta a_c n+\gamma n$, $N_f\leq 6\gamma n$ from Lemma \ref{l.decomposition}, we find that
\begin{align*}
&\sum_{i=1}^{m}\kappa_i(\theta)(t_{i}-t_{i-1})\leq n\sum_c h_c[\beta a_c+\gamma]+n\frac{\lambda(\hf)}{r-1}+ 2\eta n\\
&+6\gamma n\left(\log(2\Upsilon)+\frac{r}{r-1}\log\|D\hf_k\|_{\supnorm}\right).
\end{align*}
Recall that we chose $a_c,h_c,\ell$ and $\gamma$ so that
 \begin{itemize}
  \item[--] $\sum a_c h_c  \leq h(f,\mu_1)$ by the choice of the decomposition of $\mu_1$ in Step 1;
  \item[--] $\gamma\sum_c h_c\leq \ell\gamma\max\{h_c\}<\eta$ by eq.~\eqref{e.gamma2};
  \item[--] $6\gamma(\log(2\Upsilon)+\frac{ r}{r-1}\log\|D\hf_k\|_\supnorm)<\eta$ by eq.~\eqref{e.gamma}.
\end{itemize}
Hence,
$
\displaystyle|\mathcal R_n^\theta|\leq \exp\biggl(\beta h(f,\mu_1)n+\frac{\lambda(\hf)}{r-1}n+ 4\eta n \biggr).
$

Recalling that   $\mathcal R_n:=\bigcup_{\theta}\mathcal R_n(\theta)$ and the number of types is bounded for all $n$ large enough by $\exp[n H(10\gamma)]$, we conclude that
$$
|\mathcal R_n|\leq \exp\biggl(\beta h(f,\mu_1)n+\frac{\lambda(\hf)}{r-1}n
+ 4\eta n+H({10}\gamma) n \biggr).
$$

\subsection*{Step 11 (Completion of the proof)} By  Corollary \ref{lemma entropy reparametrizations}, for all $k$ large enough
$$
h(f_k,\nu_k)\leq \limsup_{n\to\infty}\frac{1}{n}\log|\cR_n|\leq  \beta h(f,\mu_1)+\frac{\lambda(\hf)}{r-1} +(4\eta+H(10\gamma)).
$$
Passing to the limits in the order $k\to\infty$, $\gamma\to 0$, $\eta\to 0$ gives the second part of  Theorem~\ref{t.ergodic4}, and completes its proof. \qed

\section{Supplements}\label{s.supplements}
We prove here the additional properties mentioned in Section~\ref{s.comments}.

\subsection{Discontinuities: construction of the Example~\ref{e.example}}
Let us recall that two transitive hyperbolic sets $K_1,K_2$
are \emph{homoclinically related} if a stable manifold
of $K_1$ has a transverse intersection point with an unstable manifold of $K_2$
and a stable manifold
of $K_2$ has a transverse intersection point with an unstable manifold of $K_1$.
In this case there exists a transitive hyperbolic set that contains $K_1$ and $K_2$.

If $O$ is a periodic orbit, we will denote by $\mu_O$ the invariant probability measure
supported on $O$. We say that a sequence of periodic orbits $(O_k)$
converges \weakstar{} to a measure $\mu$, if the sequence of measures $(\mu_{O_k})$
converges \weakstar{} towards $\mu$.

We say that a $C^\infty$ diffeomorphism $f_0$ belongs to the \emph{Newhouse domain}
if there exist an attracting region $U$ where
$|\det Df_{0}|<1$, a transitive hyperbolic, locally maximal set $K\subset U$ (not reduced to a periodic orbit)
and a $C^\infty$ neighborhood $\cU$ of $f_0$ such that
for any diffeomorphism $f\in \cU$
the hyperbolic continuation of $K$ (still denoted by $K$) admits a stable manifold and an unstable manifold
with a non-transverse intersection.
The Newhouse domain is open by definition, and non-empty by  \cite{Ne79}.

We prove the following more precise version of Example~\ref{e.example}:

\begin{proposition}\label{propAlphaBeta}
The Newhouse domain in $\Diff^\infty(M)$ contains a dense G$_\delta$ subset of diffeomorphisms $f$ with the following property. For any pair of numbers $0<\alpha\leq\beta\leq1$, there is a sequence of ergodic measures $(\nu_k)$ converging \weakstar\ to a measure $\mu$ with $h(f,\mu)>0$ and such that:
 $$
     \lim h(f,\nu_k) = \alpha h(f,\mu) \text{ and } \lim \lambda^+(f,\nu_k) = \beta \lambda^+(f,\mu).
  $$
\begin{remark}
One can choose for $\mu$ any invariant probability measure with positive entropy, ergodic or not, and  carried by the hyperbolic set $K$ associated with the Newhouse domain of $f$.
\end{remark}

\end{proposition}

\begin{lemma}\label{lemmaWeakPeriodic}
There is a dense G$_\delta$ subset of the Newhouse domain in $\Diff^\infty(M)$,
made of diffeomorphisms $f$ with the following property. For any periodic orbit $P$ contained in the hyperbolic set $K$ associated to $f$,
there exists a sequence of hyperbolic periodic orbits
$O_k$ homoclinically related to $P$ which converge \weakstar{} towards $P$
and satisfy $\lambda^+(O_k)\to 0$.
\end{lemma}
\begin{proof}
Let $f_0\in \cU$.
By an application of Baire's argument, it is enough to find $f$ $C^\infty$ close to $f_0$
with a periodic orbit $O$ homoclinically related to $P$
which is \weakstar{} close to $P$ and has a top Lyapunov exponent close to $0$.
We sketch the proof which uses classical arguments on the behavior near homoclinic tangencies,
and we refer to~\cite{palis-takens} for further details. In order to simplify the presentation,
we assume that $P$ is fixed and the eigenvalues $0<\lambda<1<\mu$ of $Df(P)$ are positive.
By dissipation, $\lambda\cdot \mu<1$.

Since the stable (resp. unstable) manifold of $P$ is dense in the stable
(resp. unstable) lamination of $K$, and since $f_0$ belongs to the Newhouse domain,
one can perturb $f_0$ in such a way that $P$ exhibits a quadratic homoclinic tangency $z\in W^s_{loc}(P)$.
One can also assume that the eigenvalues $\lambda,\mu$ are non-resonant, so that by Sternberg's theorem,
there exists a smooth chart on a neighborhood $U\simeq [-1,1]^2$ of $P$, where the dynamics is
linear: On $[-1,1]\times [-\mu^{-1}, \mu^{-1}]$, $f$ coincides with the map $\cL\colon (x,y)\mapsto(\lambda \cdot x,\mu \cdot y)$.
The local manifolds $W^s_{loc}(P)$ and $W^u_{loc}(P)$ coincide with $\{y=0\}$ and $\{x=0\}$.
Moreover $z$ has a preimage $z'\in W^u_{loc}(P)$ by an iterate $f^N$
and one denotes by $\mathcal T$ the map induced by $f^N$ from
a neighborhood of $z'$ to $z$.
The unstable manifold at $z$ is locally a graph $\{(x,\varphi(x))\}$
and by a suitable rescaling of the axis of $U$, one can require that $D^2\varphi\simeq 1$ near $z$.

Let us fix $\delta>0$ small. When $n$ is large
one considers a rectangle    $$R=z+[-\delta,\delta]\times [a-\delta \mu^{-n},a+\delta \mu^{-n}],$$
where $a$ is chosen such that $z'=(0,a\cdot \mu^{n})$. Note that $C^{-1}\leq a\cdot\mu^n\leq C$ where $C$ depends on the Sternberg linearization domain $U$, but not on $n$.

The rectangle $R$ is mapped by $f^{n+N}$ to a thin curved rectangle
$\mathcal T\circ \cL^n(R)$ whose width is of the order of $\delta \lambda^n$, hence much smaller than
the width of $R$. One perturbs $f$ near $z'$ in such a way that
the transition map $\cT$ is composed with a vertical translation.
The tip of the image can thus be adjusted to be at distance $L\cdot a$
from the rectangle $R$ where $L$ is a large constant independent from $n$.
Therefore $f^{n+N}(R)$ crosses $\{y=0\}$ and also $R$ with a slope $s$
close to $L\cdot a$ (since $D^2\varphi\simeq 1$).
See Figure~\ref{fig-returnmap}.

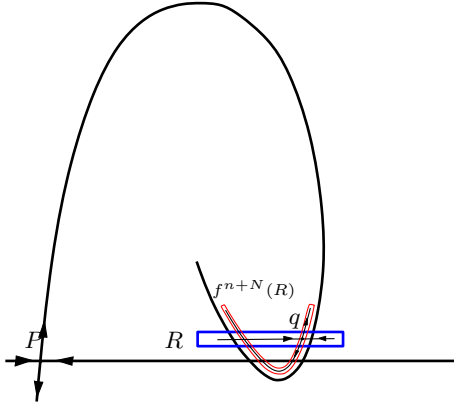
\begin{figure}[h]
\begin{center}
\includegraphics[width=6cm,angle=0]{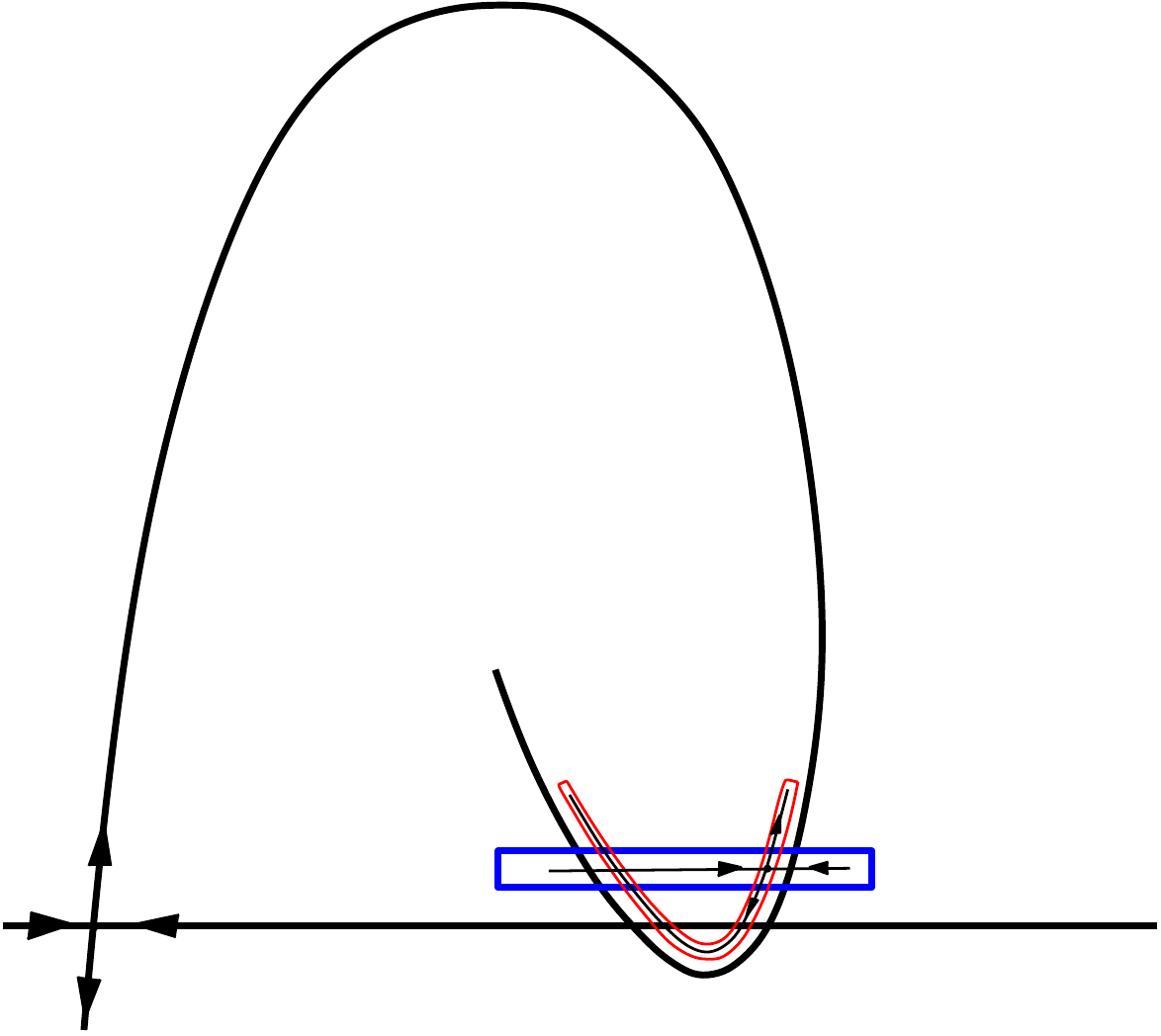}
\put(-63,30){\small $q$}
\put(-163,21){\small $P$}
\put(-110,21){\small $R$}
\put(-92,41){\tiny $f^{n+N}(R)$}
\end{center}
\caption{Return map near an homoclinic tangency.}\label{fig-returnmap}
\end{figure}

Moreover $R\cap f^{n+N}(R)$ contains a periodic point $q$ whose unstable
direction is dilated at the period by a factor of the order of
$s\exp(-\lambda^+n)\simeq L\cdot a\cdot \mu^n$,
which is close to a large constant (comparable to $L$).
As the period $n+N$  of $q$ can be chosen arbitrarily large, the unstable Lyapunov exponent of $q$
is close to $0$.

Note that the unstable manifold of $q$ crosses $f^n(R)$ along its largest dimension (see Figure~\ref{fig-returnmap}), hence crosses $W^s_{loc}(P)$.
The local stable manifold of $q$ is a graph which crosses $R$ horizontally.
The image $f^n(W^s_{loc}(P))$ if close to $f^n(R)$,
crosses $R$, and then the local stable manifold of $Q$.
Hence $P$ and the orbit of $q$ are homoclinically related.
As the $n$ first iterates of $q$ belong to the linearization domain $U$,
the orbit of $q$ spends an arbitrarily large proportion of time in any neighborhood of $P$,
as the period $n+N$ goes to infinity, proving that the invariant probability measure supported on the orbit of $q$
gets arbitrarily close to $P$ in the \weakstar{} topology.
\end{proof}

\newcommand\cE{\mathcal E}
\newcommand\cA{\mathcal A}
\newcommand\cB{\mathcal B}

We will also need the following fact:

\begin{lemma}\label{lemmaApproxInEntropy}
Let $\Lambda$ be a locally maximal, hyperbolic compact set carrying an invariant probability measure $m$, not necessarily ergodic. Then there exist ergodic invariant probability measures $m_k$ carried by $\Lambda$ which converges to $m$ in the \weakstar{} topology and in entropy: $\lim_k h(f,m_k)=h(f,m)$.
\end{lemma}

\begin{proof}[Sketch of proof]
This is routine, even if we could not locate an exact reference. Observe that it is enough to show this for a transitive subshift of finite type $\Sigma$. Given an invariant probability measure on $\Sigma$, approximate it by a Markov measure with finite memory $N$. Taking $N$ sufficiently large, we can make this approximation arbitrarily close, both  \weakstar\ and in entropy. By a small modification of the transition probabilities we can make the measure fully supported on $\Sigma$, and therefore ergodic.
\end{proof}

\begin{proof}[Proof of Proposition~\ref{propAlphaBeta}]
For convenience, we fix some distance $d$ on the space of Borel probability measures of $M$,
compatible with the \weakstar\ topology.
Let $f$ be a diffeomorphism with a locally maximal transitive hyperbolic set $K$ as given by Lemma~\ref{lemmaWeakPeriodic}.
Since $K$ is not reduced to a single periodic orbit, it carries invariant probability measures with positive entropy.
We choose any one of them.
Lemma~\ref{lemmaWeakPeriodic} yields a sequence $(O_k)_{k\ge1}$ of hyperbolic periodic orbits
homoclinically related to $K$ such that $d(\mu_{O_k},\mu)<1/k$ and $|\lambda^+(O_k)|<1/k$.

Fix $k\geq1$. Let $P$ be a periodic orbit in $K$ so close to $\mu$ that $d(\mu_P,\mu)<1/k$ and  $|\lambda^+(f,\mu_P)-\lambda^+(f,\mu)|<1/k$ (by continuity of the unstable bundle over $K$).

Let $\Lambda$ be a transitive, hyperbolic, locally maximal invariant set containing $K\cup O_k$.  Define $m=\alpha\mu+(\beta-\alpha)\mu_P+(1-\beta)\mu_{O_k}$ on $\Lambda$. Now Lemma~\ref{lemmaApproxInEntropy} yields $\nu_k$ such that:
 \begin{enumerate}[\quad --]
  \item $|h(f,\nu_k)-\alpha h(f,\mu)|=|h(f,\nu_k)-h(f,m)|<1/k$;
  \item $d(\nu_k,\mu)<d(\nu_k,m)+1/k<2/k$;
  \item $|\lambda^+(f,\nu_k)-\beta\lambda^+(f,\mu)|<|\lambda^+(f,\nu_k)-\lambda^+(f,m)|+1/k<2/k$.
 \end{enumerate}
The sequence $(\nu_k)_{k\geq1}$ is as claimed.
\end{proof}

\subsection{Variant inequality on Lyapunov exponents: proof of Corollary~\ref{c.variant}}
By the Oseledets theorem,
 $\lambda^+(f,\mu)+\lambda^-(f,\mu)=\int\log|\det Df|d\mu,$
which is continuous with respect to $(f,\mu)$ in the $C^1\times$\weakstar\ topology.  Since the sum is continuous, the discontinuities in the summands must cancel out, whence
\begin{equation}\label{sym-disc}
\lambda^{+}(f,\mu)-\lim_{k\to\infty} \lambda^{+}(f_k,\nu_k)=-\left(\lambda^{-}(f,\mu)-\lim_{k\to\infty} \lambda^{-}(f_k,\nu_k)\right).
\end{equation}
Now, Theorem \ref{t.ergodic} is equivalent to the statement
\begin{equation}\label{gold-var}
\lambda^+(f,\mu)-\lim_{k\to\infty}\lambda^+(f_k,\nu_k)\leq \lambda^+(f,\mu)\left(1-\frac{\lim_k h(f_k,\nu_k)}{h(f,\mu)}\right).
\end{equation}
Applying this to $f_k^{-1}$ and noting \eqref{sym-disc} and  $\lambda^-(f,\mu)=-\lambda^+(f^{-1},\mu)$,  we obtain \eqref{stronger-ineq}.
{Since}  $0<-\lambda^-(f,\mu)<\lambda^+(f,\mu)$, this is stronger than the conclusion of Theorem \ref{t.ergodic}.
\qed

\subsection{Sequences of non-ergodic measures}
We state and prove a version of Theorem~\ref{t.ergodic4} removing its assumption that the converging measures $\nu_k$ are ergodic.

\begin{corollary}\label{c.non-ergodic-nu}
Fix a real number $r>2$.
For every $k\geq 1$, let  $f_k\in\Diff^r(M)$,
$\nu_k$ be an $f_k$-invariant measure, not necessarily ergodic, and $\hnu_k$ be a lift satisfying $\hlambda(\hf_k,\hnu_k)=\lambda^+(f_k,\nu_k)$. Let us assume that:
\begin{enumerate}[\quad(1)]
  \item $\lim_k\lambda^+(f_k,\nu_k)$ and $\lim_k h(f_k,\nu_k)$ exist
  and $\liminf_k \int\min( \lambda^+(f_k,x), 0) d\nu_k(x)= 0,$
  \item $f_k\overset{\scriptscriptstyle r-\text{bd}}\longrightarrow f$ for some $f\in\Diff^r(M)$,
  \item $\widehat \nu_k\overset{w^*}\to \widehat \mu$ for some $\widehat f$-invariant probability measure $\widehat \mu$ (perhaps non-ergodic) on $\widehat M$.
 \end{enumerate}
Then there exist $\beta\in [0,1]$, two
$f$-invariant measures $\mu_0,\mu_1$
with $\hf$-invariant lifts $\hmu_0,\hmu_1$ s.t.

\begin{enumerate}[\quad(a)]
\item\label{e.eqD0} $\widehat \mu=(1-\beta)\widehat \mu_0+\beta\widehat \mu_1;$
\item\label{e.eqD1} $\lim_{k\to\infty} \lambda^+(f_k,\nu_k) =  \beta\lambda^+(f,\mu_1);$
\item\label{e.eqD2} $\lim_{k\to\infty} h(f_k,\nu_k)-\tfrac1{r-1}\lambda(\hf)\leq \beta \left(h(f,\mu_1)+\tfrac{\lambda(f)}{r}\right);$
\item\label{e.eqD3} if $\beta<1$,\ $\hlambda(\hf,\hmu_0)=0$;
\item\label{e.eqD4} if $\beta>0$, then $\hlambda(\hf,\hmu_1)=\lambda^+(f,\mu_1)$ and $\lambda^+(f,x)>0$ for
$\mu_1$-a.e. $x$.
\end{enumerate}
\end{corollary}

Note that if $f_k\in \Diff^\infty(M)$ and $f_k\to f$ in $C^\infty$, then the corollary applies for all $r>2$, and Property~\eqref{e.eqD2} becomes $
\lim_{k\to\infty} h(f_k,\nu_k)\leq \beta h(f,\mu_1)$.

Note also that the conclusions (a)-(e) are the same as in Theorem~\ref{t.ergodic4}, except for the extra term $\lambda(f)/r$ on the right hand side. See the following remark on this term.

\begin{remark}
Our proof relies on discretizing the ergodic decompositions of the measures $\nu_k$ and applying Theorem~\ref{t.ergodic4} to the atoms thus defined and taking a limit.
This limiting process is responsible for the term $\lambda(f)/r$ in the entropy estimate in Property~\eqref{e.eqD2}.

Using the decomposition in the projective bundle is the key to avoid any such loss in the Lyapunov estimate~eq.\ \eqref{e.eqD1} and is therefore essential for our proof of this generalization.
\end{remark}

\newcommand\hZ{\widehat Z}

\begin{proof}
Let $\mathbb P(M)$ denote the set of Borel probability measures on $M$, and let $d$ be the $L^1$-Wasserstein distance over $\Prob(M)$.  This distance  is compatible with the \weakstar\ topology and satisfies $d(\sum_{i=1}^N \alpha_i\mu_i,\sum_{i=1}^N \alpha_i\nu_i)\leq \sum_{i=1}^N \alpha_i d(\mu_i,\nu_i)$ for all convex combinations.

We fix some $\eps>0$ and discretize the ergodic decompositions
 $$
   \nu_k = \int_X \nu_{k,\xi}\, dP_k(\xi).
 $$
By compactness of $\Prob(M)$, there are 
measurable partitions $X=X^\eps_{k,1}\sqcup\dots\sqcup X^\eps_{k,{ N^\eps}}$ with number of elements $N^\eps$ independent of $k$, and  with the following property for every $1\leq i\leq N^\eps$:
There is an $\hf_k$-ergodic measure $\hnu^\eps_{k,i}$ with projection $\nu^\eps_{k,i}$ satisfying $\hlambda(\hf_k,\hnu_{k,i}^\eps)=\lambda^+(f_k,\nu_{k,i}^\eps)$ 
and, for $P_k$-a.e. $\xi\in X^\eps_{k,i}$,
 \begin{equation}\label{eq-discretize}
 d(\nu_{k,\xi},\nu^\eps_{k,i})<\eps,\; |h(f_k,\nu_{k,\xi})-h(f_k,\nu^\eps_{k,i})|<\eps,\text{ and }
  \end{equation}
  \begin{equation}\label{eq-discretize2}
 \lambda^+(f_k,\nu_{k,\xi})-\tfrac 1 k \;<\;  \lambda^+(f_k,\nu^\eps_{k,i})\; <\;  \lambda^+(f_k,\nu_{k,\xi})+\eps.
 \end{equation}
Passing to a subsequence, we may assume without loss of generality the existence of the following limits:
 $$
   \hmu^{\eps}_i:=\lim_k \hnu^{\eps}_{k,i},\;\; \lim_k h(f_k,\nu^\eps_{k,i}),\; \;\lim_k \lambda^+(f_k,\nu^\eps_{k,i}), \; \;\alpha^\eps_i := \lim_k P_k(X^\eps_{k,i}).
 $$

For each $i=1,\dots,N$, we set $ \mu^\eps_i:=\hpi_*(\hmu^\eps_i)$.
By inequalities~\eqref{eq-discretize2} and
by our assumption $\liminf_k \int\min( \lambda^+(f_k,x), 0) d\nu_k(x)= 0$,
it follows that either $\alpha^\varepsilon_i=0$ or $\lim_k\lambda^+(f_k,\nu^\eps_{k,i})\geq 0$.
One can thus apply Theorem~\ref{t.ergodic4}
\color{black}
to the sequence $(f_k,\hnu^\eps_{k,i})_{k\geq1}$ converging to $(f,\hmu^\eps_i)$ and obtain  a decomposition
 $$
   \hmu^\eps_i=(1-\beta^\eps_i)\hmu^\eps_{0,i}+\beta^\eps_i\hmu^\eps_{1,i}
 $$
for some $0\leq\beta^\eps_i\leq1$ and $\hmu^\eps_{0,i},\hmu^\eps_{1,i}$ $\hf$-invariant measures such that
$$
    \lim_k h(f_k,\nu^\eps_{k,i})\leq \beta^\eps_i h(f,\mu^\eps_{1,i}) + \frac{\lambda(\hf)}{r-1}
     \text{ and } \lim_k\lambda^+(f_k,\nu^\eps_{k,i})=\beta^\eps_i \lambda^+(f,\mu^\eps_{1,i})
    =\beta^\eps_i \hmu^\eps_{1,i}(\vf).
    $$
We collect all the pieces, setting:
 $$
   \hnu^{\eps}_k:=\sum_{i=1}^{ N^\eps} \alpha^\eps_{k,i}\hnu^{\eps}_{k,i},\;
   \hmu^\eps:= \sum_{i=1}^{N^\eps} \alpha^\eps_i \,\hmu^{\eps}_{i},\;
   \beta^\eps:= \sum_{i=1}^{ N^\eps} \alpha^\eps_i \beta^{\eps}_{i},\;
   \hmu^\eps_s:= \sum_{i=1}^{ N^\eps} \alpha^\eps_i \hmu^{\eps}_{s,i} \text{ for }s=0,1.
 $$
We denote by $\nu^\eps_k,\mu^\eps,\mu^\eps_s$ the projections by $\hpi$.
We have that:
 $$
   \hmu^\eps = (1-\beta^\eps)\hmu^\eps_0 + \beta^\eps\hmu^\eps_1.
 $$
By eq.~\eqref{eq-discretize},
 $$
   d(\nu^\eps_k,\nu_k)\le\eps,\;
   |h(f_k,\nu^\eps_k) - h(f_k,\nu_k)|<\eps,\;
   |\lambda^+(f_k,\nu^\eps_k) - \lambda^+(f_k,\nu_k)|<\eps.
 $$
Using that the entropy and average exponents are affine functions, we get
 \begin{equation}\label{eq-h-eps}
     \lim_k h(f_k,\nu_k) \leq \liminf_k h(f_k,\nu^\eps_k) +\varepsilon \leq  \beta^\eps h(f,\mu^\eps_1) + \frac{\lambda(\hf)}{r-1} + \eps
 \end{equation}
and
 \begin{equation}\label{eq-lambda-eps}
     \left|\lim_k\lambda^+(f_k,\nu_k) - \lim_k\lambda^+(f_k,\nu^\eps_k)\right|\le\eps \text{ so }
     \left|\lim_k\lambda^+(f_k,\nu_k) -\beta^\eps \lambda^+(f,\mu^\eps_{1})\right|\le\eps.
 \end{equation}
To conclude, we pick a sequence of numbers $(\eps_j)_{j\geq1}$ decreasing to zero along which the three sequences $\hmu^\eps$, $\hmu^\eps_0$ and $\hmu^\eps_1$ converge to measures $\hmu$, $\hmu_0$ and $\hmu_1$. Since $d(\mu^\eps,\mu)\leq\eps$, we must have $\hpi_*\hmu=\mu$. We define $\mu_s:=\pi_*(\hmu_s)$ for $s=0,1$. Property~\eqref{e.eqD0} follows by continuity.

By Yomdin theory (see the discussion after eq.\ \eqref{e.improved-eq.2}),
 $$
   \lim_j h(f,\mu^{\eps_j}_1) \leq h(f,\mu_1) + \lambda(f)/r,
 $$
yielding Property~\eqref{e.eqD2}.
The decomposition converges in the projective bundle, hence, recalling that $\vf(x,E):=\log\|D_xf|_E\|$ is  a continuous function,
 $$
   \lim_j \lambda^+(f,\mu^{\eps_j}_1) = \lim_j \hmu^{\eps_j}_1(\vf) = \lambda^+(f,\mu_1).
 $$
Property \eqref{e.eqD1} follows.
For each $i$, one has $\hmu^\eps_{0,i}(\vf)=0$ once $\beta^\varepsilon_i\neq 0$;
this implies Property \eqref{e.eqD3}.

We now turn to Property \eqref{e.eqD4}.
We have $\beta \hlambda(\hf,\hmu_1)=\hlambda(\hf,\hmu)=\lim_{k\to\infty} \hlambda(f_k,\hnu_k)$
and by Property~\eqref{e.eqD1}:
$\lim_{k\to\infty} \hlambda(f_k,\hnu_k)=\lim_{k\to\infty} \lambda^+(f_k,\nu_k) =  \beta\lambda^+(f,\mu_1)$.
This gives the first part, assuming $\beta>0$.

Note that $\lambda^+(f,x)\geq0$ for $\mu$-a.e. $x$ since otherwise the ergodic decomposition of $\mu$ would contain a source as an atom and therefore $\nu_k$ would contain the same atom with uniform weight for large $k$, in contradiction to our assumption
$\liminf_k \int\min( \lambda^+(f_k,x), 0) d\nu_k(x)= 0$. Let
 $$
    \hZ:=\{\hx\in\hM : \lambda^+(f,\hpi(\hx))=0\}.
 $$
We assume $0<\hmu_1(\hZ)<1$ since it is otherwise easy to conclude. We set:
 $$
    \beta':=\beta(1-\hmu_1(\hZ)),\;
    (1-\beta')\hmu_0':=(1-\beta)\hmu_0+\beta\frac{\hmu_1(\,\cdot\cap\hZ)}{\hmu_1(\hZ)},\;
    \beta'\hmu_1':=\beta\frac{\hmu_1(\,\cdot\setminus\hZ)}{1-\hmu_1(\hZ)}
 $$
obtaining a new decomposition $\hmu=(1-\beta')\hmu'_0+\beta'\hmu_1'$.
Note that, setting $\mu'_1:=\hpi_*\hmu'_1$,
 $$
    \beta' h(f,\mu'_1) = \beta h(f,\mu_1) \text{ and } \beta'\lambda^+(f,\mu'_1)= \beta\lambda^+(f,\mu_1).
 $$
since the probability measure $\frac{\mu_1(\,\cdot\cap Z)}{\hmu_1(Z)}$ has both  zero entropy and zero top Lyapunov exponent.
This concludes the proof of the Corollary~\ref{c.non-ergodic-nu}.
\end{proof}

\small
\bibliographystyle{plain-like-initial}
\bibliography{ExponentContinuity-bib-file}

\def\cprime{$'$} \def\cprime{$'$} \def\cprime{$'$} \def\cprime{$'$}
\begin{thebibliography}{10}

\bibitem{Avila-Viana}
A.~Avila and M.~Viana.
\newblock Extremal {L}yapunov exponents: an invariance principle and
  applications.
\newblock {\em Invent. Math.} \textbf{181} (2010), 115--189.

\bibitem{Backes-Brown-Butler}
L.~Backes, A.~Brown, and C.~Butler.
\newblock Continuity of {L}yapunov exponents for cocycles with invariant
  holonomies.
\newblock {\em J. Mod. Dyn.} \textbf{12} (2018), 223--260.

\bibitem{Barreira-Gelfert}
L.~Barreira and K.~Gelfert.
\newblock Dimension estimates in smooth dynamics: a survey of recent results.
\newblock {\em Ergodic Theory Dynam. Systems} \textbf{31} (2011), 641--671.

\bibitem{Barreira-Pesin-Schmeling}
L.~Barreira, Y.~Pesin, and J.~Schmeling.
\newblock Dimension and product structure of hyperbolic measures.
\newblock {\em Ann. of Math.} \textbf{149} (1999), 755--783.

\bibitem{Barreira-Wolf}
L.~Barreira and C.~Wolf.
\newblock Measures of maximal dimension for hyperbolic diffeomorphisms.
\newblock {\em Comm. Math. Phys.} \textbf{239} (2003), 93--113.

\bibitem{Bochi}
J.~Bochi.
\newblock Genericity of zero {L}yapunov exponent.
\newblock {\em Ergodic Theory Dynam. Systems} \textbf{22} (2002), 1667--1696.

\bibitem{Bocker-Viana}
C.~Bocker-Neto and M.~Viana.
\newblock Continuity of {L}yapunov exponents for random two-dimensional
  matrices.
\newblock {\em Ergodic Theory Dynam. Systems} \textbf{37} (2017), 1413--1442.

\bibitem{Bowen}
R.~Bowen.
\newblock Entropy for group endomorphisms and homogeneous spaces.
\newblock {\em Trans. Amer. Math. Soc.} \textbf{153} (1971), 401--414.

\bibitem{Boyle-Downarowicz-2004}
M.~Boyle and T.~Downarowicz.
\newblock The entropy theory of symbolic extensions.
\newblock {\em Invent. Math.} \textbf{156} (2004), 119--161.

\bibitem{Burguet}
D.~Burguet.
\newblock Symbolic extensions in intermediate smoothness on surfaces.
\newblock {\em Ann. Sci. \'Ec. Norm. Sup\'er.} \textbf{45} (2012), 337--362.

\bibitem{BurguetFibered2017}
D.~Burguet.
\newblock Usc/fibred entropy structure and applications.
\newblock {\em Dyn. Syst.} \textbf{32} (2017), 391--409.

\bibitem{BuzziPhD}
J.~Buzzi.
\newblock {\em Repr\'esentation markovienne des applications r\'eguli\`eres de
  l'intervalle}.
\newblock PhD thesis, Universit\'e Paris-Sud, Orsay, 1995.

\bibitem{BuzziSIM}
J.~Buzzi.
\newblock Intrinsic ergodicity of smooth interval maps.
\newblock {\em Israel J. Math.} \textbf{100} (1997), 125--161.

\bibitem{BuzziNoMax}
J.~Buzzi.
\newblock {$C^r$} surface diffeomorphisms with no maximal entropy measure.
\newblock {\em Ergodic Theory Dynam. Systems} \textbf{34} (2014), 1770--1793.

\bibitem{ETCS}
M.~Denker, C.~Grillenberger, and K.~Sigmund.
\newblock {\em Ergodic theory on compact spaces}.
\newblock Lecture Notes in Mathematics, Vol. 527. Springer-Verlag, Berlin-New
  York, 1976.

\bibitem{Downarowicz-Entropy}
T.~Downarowicz.
\newblock Entropy structure.
\newblock {\em J. Anal. Math.} \textbf{96} (2005), 57--116.

\bibitem{Downarowicz-Newhouse-2005}
T.~Downarowicz and S.~Newhouse.
\newblock Symbolic extensions and smooth dynamical systems.
\newblock {\em Invent. Math.} \textbf{160} (2005), 453--499.

\bibitem{Einsiedler-Kadyrov-Pohl}
M.~Einsiedler, S.~Kadyrov, and A.~Pohl.
\newblock Escape of mass and entropy for diagonal flows in real rank one
  situations.
\newblock {\em Israel J. Math.} \textbf{210} (2015), 245--295.

\bibitem{Furstenberg-1963}
H.~Furstenberg.
\newblock Noncommuting random products.
\newblock {\em Trans. Amer. Math. Soc.} \textbf{108} (1963), 377--428.

\bibitem{Furstenberg-Kifer}
H.~Furstenberg and Y.~Kifer.
\newblock Random matrix products and measures on projective spaces.
\newblock {\em Israel J. Math.} \textbf{46} (1983), 12--32.

\bibitem{Iommi-Riquelme-Velozo}
G.~Iommi, F.~Riquelme, and A.~Velozo.
\newblock Entropy in the cusp and phase transitions for geodesic flows.
\newblock {\em Israel J. Math.} \textbf{225} (2018), 609--659.

\bibitem{Iommi-Todd-Velozo}
G.~Iommi, M.~Todd, and A.~Velozo.
\newblock Escape of entropy for countable markov shifts.
\newblock ArXiv:1908.10741.

\bibitem{Iommi-Todd-Velozo-USC}
G.~Iommi, M.~Todd, and A.~Velozo.
\newblock Upper semi-continuity of entropy in non-compact settings.
\newblock {\em Math. Res. Lett.} \textbf{27} (2020), 1055--1077.

\bibitem{Kadyrov-Effective}
S.~Kadyrov.
\newblock Effective equidistribution of periodic orbits for subshifts of finite
  type.
\newblock {\em Colloq. Math.} \textbf{149} (2017), 93--101.

\bibitem{Kadyrov-Pohl}
S.~Kadyrov and A.~Pohl.
\newblock Amount of failure of upper-semicontinuity of entropy in non-compact
  rank-one situations, and {H}ausdorff dimension.
\newblock {\em Ergodic Theory Dynam. Systems} \textbf{37} (2017), 539--563.

\bibitem{KatokIHES}
A.~Katok.
\newblock Lyapunov exponents, entropy and periodic orbits for diffeomorphisms.
\newblock {\em Inst. Hautes \'Etudes Sci. Publ. Math.} \textbf{51} (1980),
  137--173.

\bibitem{Ledrappier-Young-II}
F.~Ledrappier and L.-S. Young.
\newblock The metric entropy of diffeomorphisms. {II}. {R}elations between
  entropy, exponents and dimension.
\newblock {\em Ann. of Math.} \textbf{122} (1985), 540--574.

\bibitem{Ledrappier1982}
F.~Ledrappier.
\newblock Quelques propri\'et\'es des exposants caract\'eristiques.
\newblock {\em Lecture Notes in Math.} \textbf{1097} (1984), 185--202.
\newblock \'Ecole d'\'et\'e de probabilit\'es de Saint-Flour XII, 1982.

\bibitem{Ledrappier-Strelcyn}
F.~Ledrappier and J.-M. Strelcyn.
\newblock A proof of the estimation from below in {P}esin's entropy formula.
\newblock {\em Ergodic Theory Dynam. Systems} \textbf{2} (1982), 203--219.

\bibitem{Ledrappier-Walters}
F.~Ledrappier and P.~Walters.
\newblock A relativised variational principle for continuous transformations.
\newblock {\em J. London Math. Soc.} \textbf{16} (1977), 568--576.

\bibitem{Ledrappier-Young-I}
F.~Ledrappier and L.-S. Young.
\newblock The metric entropy of diffeomorphisms. {I}. {C}haracterization of
  measures satisfying {P}esin's entropy formula.
\newblock {\em Ann. of Math.} \textbf{122} (1985), 509--539.

\bibitem{Mane-Book}
R.~Ma\~{n}\'{e}.
\newblock {\em Ergodic theory and differentiable dynamics}, volume~8 of {\em
  Ergebnisse der Mathematik und ihrer Grenzgebiete}.
\newblock Springer-Verlag, Berlin, 1987.

\bibitem{misiurewicz}
M.~Misiurewicz.
\newblock Diffeomorphism without any measure with maximal entropy.
\newblock {\em Bull. Acad. Polon. Sci. S\'{e}r. Sci. Math. Astronom. Phys.}
  \textbf{21} (1973), 903--910.

\bibitem{Ne79}
S.~Newhouse.
\newblock The abundance of wild hyperbolic sets and nonsmooth stable sets for
  diffeomorphisms.
\newblock {\em Publ. Math. Inst. Hautes \'Etudes Sci.} \textbf{50} (1979),
  101--151.

\bibitem{Newhouse1990}
S.~Newhouse.
\newblock Entropy in smooth dynamical systems.
\newblock {\em Proceedings of the International Congress of Mathematicians,
  Kyoto 1990} \textbf{} (1991), 1285--1294.

\bibitem{palis-takens}
J.~Palis and F.~Takens.
\newblock {\em Hyperbolicity and sensitive chaotic dynamics at homoclinic
  bifurcations}, volume~35 of {\em Cambridge Studies in Advanced Mathematics}.
\newblock Cambridge University Press, 1993.

\bibitem{palis-viana}
J.~Palis and M.~Viana.
\newblock On the continuity of {H}ausdorff dimension and limit capacity for
  horseshoes.
\newblock {\em Lecture Notes in Math} \textbf{1331} (1988), 150--160.
\newblock Dynamical systems, Valparaiso 1986.

\bibitem{Pesin-Izvestia-1976}
Y.~Pesin.
\newblock Families of invariant manifolds that correspond to nonzero
  characteristic exponents.
\newblock {\em Izv. Akad. Nauk SSSR Ser. Mat.} \textbf{40} (1976), 1332--1379.

\bibitem{pesin-formula}
Y.~Pesin.
\newblock Characteristic {L}yapunov exponents and smooth ergodic theory.
\newblock {\em Russian Math. Surveys} \textbf{32} (1977), 55--114.

\bibitem{Pliss}
V.~A. Pliss.
\newblock On a conjecture of {S}male.
\newblock {\em Differencial\cprime nye Uravnenija} \textbf{8} (1972), 268--282.

\bibitem{Polo}
F.~Polo.
\newblock {\em Equidistribution in chaotic dynamical systems}.
\newblock PhD thesis, The Ohio State University, 2011.

\bibitem{Riquelme-Velozo}
F.~Riquelme and A.~Velozo.
\newblock Escape of mass and entropy for geodesic flows.
\newblock {\em Ergodic Theory Dynam. Systems} \textbf{39} (2019), 446--473.

\bibitem{Ruelle-Entropy-Inequality}
D.~Ruelle.
\newblock An inequality for the entropy of differentiable maps.
\newblock {\em Bol. Soc. Brasil. Mat.} \textbf{9} (1978), 83--87.

\bibitem{Ruelle-exponent-continuity}
D.~Ruelle.
\newblock Analycity properties of the characteristic exponents of random matrix
  products.
\newblock {\em Adv. in Math.} \textbf{32} (1979), 68--80.

\bibitem{Ruhr}
R.~R\"{u}hr.
\newblock Effectivity of uniqueness of the maximal entropy measure on
  {$p$}-adic homogeneous spaces.
\newblock {\em Ergodic Theory Dynam. Systems} \textbf{36} (2016), 1972--1988.

\bibitem{Viana-book}
M.~Viana.
\newblock {\em Lectures on {L}yapunov exponents}, volume 145 of {\em Cambridge
  Studies in Advanced Mathematics}.
\newblock Cambridge University Press, 2014.

\bibitem{Viana-survey}
M.~Viana.
\newblock (dis)continuity of {L}yapunov exponents.
\newblock {\em Ergodic Theory Dynam. Systems} \textbf{40} (2020), 577--611.

\bibitem{Yomdin-Volume-Growth-Entropy}
Y.~Yomdin.
\newblock Volume growth and entropy.
\newblock {\em Israel J. Math.} \textbf{57} (1987), 285--300.

\bibitem{young}
L.-S. Young.
\newblock Dimension, entropy and {L}yapunov exponents.
\newblock {\em Ergodic Theory Dynam. Systems} \textbf{2} (1982), 109--124.

\bibitem{Zang-EntropyVolume}
Y.~Zang.
\newblock {Entropies and volume growth of unstable manifolds}.
\newblock ArXiv:1912.13189.

\end{thebibliography}
\bigskip

 \end{document}